\documentclass[a4paper, 11pt]{article}

\usepackage[utf8]{inputenc}
\usepackage{amsmath}
\usepackage{amsthm}
\usepackage{amsfonts}
\usepackage{amssymb}
\usepackage{amstext}
\usepackage{mathscinet}
\usepackage{dsfont}
\usepackage{graphicx}
\usepackage{color}
\usepackage[colorlinks]{hyperref}
\usepackage{epigraph}
\usepackage[left=3.5cm, right=3cm,nohead,nofoot]{geometry}
\usepackage{tikz}
\usepackage{tkz-graph}
\usepackage{tkz-berge}
\usetikzlibrary{arrows,shapes,matrix}

\graphicspath{{Figures/}{.}}%

\title{ Littelmann path model for geometric crystals }
\author{Reda \textsc{Chhaibi}        \footnote{Universit\"at Z\"urich. Email: \texttt{reda.chhaibi@math.uzh.ch}} } 
\date{}

\allowdisplaybreaks[4]

\DeclareMathOperator{\Ad}{Ad}
\DeclareMathOperator{\ad}{ad}
\DeclareMathOperator{\hw}{hw}
\DeclareMathOperator{\Inv}{Inv}
\DeclareMathOperator{\lw}{lw}

\DeclareMathOperator{\sh}{sh}

\def\half{\frac{1}{2}}

\def\A{{\mathbb A}}
\def\N{{\mathbb N}}
\def\Z{{\mathbb Z}}
\def\Q{{\mathbb Q}}
\def\R{{\mathbb R}}
\def\C{{\mathbb C}}

\def\Bc{{\mathcal B}}
\def\Cc{{\mathcal C}}

\def\Ec{{\mathcal E}}
\def\Fc{{\mathcal F}}

\def\Pc{{\mathcal P}}

\def\Tc{{\mathcal T}}
\def\Uc{{\mathcal U}}
\def\Zc{{\mathcal Z}}

\def\Bfrak{{\mathfrak B}}
\def\afrak{{\mathfrak a}}
\def\bfrak{{\mathfrak b}}
\def\gfrak{{\mathfrak g}}
\def\hfrak{{\mathfrak h}}

\def\nfrak{{\mathfrak n}}

\def\ufrak{{\mathfrak u}}

\numberwithin{equation}{section}
\numberwithin{figure}{section}

\newtheorem{thm}[equation]{Theorem}
\newtheorem{proposition}[equation]{Proposition}
\newtheorem{corollary}[equation]{Corollary}

\newtheorem{definition}[equation]{Definition}
\newtheorem{example}[equation]{Example}
\newtheorem{lemma}[equation]{Lemma}
\newtheorem{properties}[equation]{Properties}

\newtheorem{rmk}[equation]{Remark}

\begin{document}
\maketitle

\begin{abstract}
We construct a path model for geometric crystals in the sense of Berenstein and Kazhdan. Our model is in every way similar to Littelmann's and tropicalizes to his path model. This paper lays the foundational material for the subsequent work \cite{bib:chh14c} where we examine the measure induced on geometric crystals by Brownian motion .

If we call Berenstein and Kazhdan's realization of geometric crystals the group picture, we prove that the path model projects onto the group picture thanks to a morphism of crystals that restricts to an isomorphism on connected components. This projection is in fact the geometric analogue of the Robinson-Schensted correspondence and involves solving a left-invariant differential equation on the Borel subgroup.

Moreover, we identify the geometric Pitman transform $\Tc_{w_0}$ introduced by Biane, Bougerol and O'Connell as the transform giving the path with highest weight, in the geometric crystal path model. This allows to prove a geometric version of Littelmann's independence theorem. The geometric Robinson-Schensted correspondence is detailed in a special section, because of its importance.

Finally, we exhibit the Kashiwara and Schützenberger involutions in both the group picture and the path model. 

In an appendix, we explain how the left-invariant flow is related to the image of the Casimir element in Kostant's Whittaker model.
\end{abstract}

\clearpage
\tableofcontents
\clearpage

\section{Introduction}
Let $G$ be a simply-connected complex semi-simple group with Lie algebra $\gfrak$. Let $\hfrak$ be a Cartan subalgebra and $\afrak$ is the subspace of $\hfrak$ where simple roots are real-valued. Equivalently $\afrak$ is the Cartan subalgebra of the split real form of $G$.

The theory of crystal bases was initiated after Lusztig introduced the canonical basis for the quantum groups $\Uc_q\left( \gfrak \right)$ (\cite{bib:Lusztig93}). At $q=0$, one has Kashiwara's crystal ``basis'' which is well behaved in many aspects (see \cite{bib:Kashiwara95}). Littelmann's work ( \cite{bib:Littelmann95}, \cite{bib:Littelmann95a}, \cite{bib:Littelmann95b}, \cite{bib:Littelmann96}) realises crystals as sets of paths in $\afrak^*$, in the more general context of Kac-Moody groups. Consequently, Berenstein and Kazhdan (\cite{bib:BK00}, \cite{bib:BK04}) defined geometric crystals as algebro-geometric objects that degenerate to Kashiwara crystals after tropicalization. Their construction fundamentally relies on the fact that totally positive varieties in $G$ are the geometric liftings of combinatorial $G^\vee$-crystals, $G^\vee$ being the Langlands dual of $G$.

In this paper, we construct a path model for geometric crystals, in the same spirit as Littelmann. The weight function is the endpoint of a path and tensor product of crystals is given by path concatenation. However root operators are defined as integral transforms, which are the geometric lifting of Littelmann's piece-wise reflection scheme. In fact, the path model we describe is naturally related to geometric crystals for the Langlads dual $G^\vee$. As such, crystal elements are paths in the real Cartan subalgebra $\afrak$, instead of $\afrak^*$. 

We start in section \ref{section:abstract_geom_crystals} by defining a notion of abstract geometric crystal in a less restrictive sense than the original definition of Berenstein and Kazhdan. Whereas they defined geometric crystals as affine varieties over $\C$, we will simply consider them as topological sets with structural maps. This will allow us later to consider the geometric crystal of continuous paths valued in the real Cartan subalgebra $\afrak$.

In section \ref{section:geom_lifting}, we consider the totally positive group elements $\Bc$ inside the Borel subgroup $B$, which form the typical positive geometric crystal in the sense of Berenstein and Kazhdan. This is what we call the group picture. We fix a certain number of coordinate charts for $\Bc$ (figure \ref{fig:geom_parametrizations}) and detail its geometric crystal structure. If the geometric liftings of Lusztig's and Kashiwara's parametrizations of crystal bases are well known (\cite{bib:BZ01}), our contribution consists in fitting them together into geometric crystals. Notice that, instead of defining a positive structure separately (as in \cite{bib:BK06} section 3) we will be directly working at the level of the totally non-negative elements in $G$. Finally, we will favor using the additive group $\R$ instead of the multiplicative $\R_{>0}$ hence the presence of numerous logarithms.

Fix once and for all a time horizon $T>0$. Let $\Cc_0\left( [0, T], \afrak \right)$ be the space of continuous functions taking values in $\afrak$ and starting from zero. In section \ref{section:path_model}, we construct a geometric crystal structure on $\Cc_0\left( [0, T], \afrak \right)$. By rescaling paths using a real parameter $h>0$, we have an entire family of equivalent geometric path models. Taking $h \rightarrow 0$, the geometric crystal structure tropicalizes to the continuous version of Littelmann's path model (\cite{bib:BBO}, \cite{bib:BBO2}). Here, an interesting feature is that the tropical expression found for root operators $e^{c}_\alpha$ is the same regardless of the sign of $c \in \R$. In order to see that the result matches with previous work, the necessary verifications are made in subsection \ref{subsection:classical_littemann_limit}.

A projection $p: \Cc\left( [0, T], \afrak \right) \rightarrow \Bc$ of a path $\pi$ onto its group picture is obtained by solving a left-invariant differential equation on $B$ driven by $\pi$ (Section \ref{section:path_transforms}). If the solution to this equation, first considered in \cite{bib:BBO}, is written $\left( B_t(\pi); 0 \leq t \leq T \right)$, then:
$$ p: \pi \mapsto B_T(\pi)$$
is a morphism of crystals (Theorem \ref{thm:p_is_morphism}).

Isomorphism classes of geometric crystals in the group picture are indexed by a single vector $\lambda$ that is interpreted as a highest weight. By simply transporting the structure to the path model, we see that in order to obtain the isomorphism class of a crystal generated by $\pi$, a remarkable transform on paths $\Tc_{w_0}$ has to be applied. $\Tc_{w_0} \pi$ plays the role of a dominant path in the geometric path model. The highest weight $\lambda \in \afrak$ is the endpoint of $\Tc_{w_0} \pi$. This transform was introduced in \cite{bib:BBO2} as a geometric lifting of the Pitman operator $\Pc_{w_0}$ and will be referred to as the geometric Pitman transform. A key difference with the 'crystallized' case obtained in \cite{bib:BBO, bib:BBO2} is that the path $\Tc_{w_0} \pi$ does not belong to the crystal $\langle \pi \rangle$ generated by $\pi$. For instance, it is not defined at zero. It is however invariant under crystal actions. This difference would explain why the naturality of $\Tc_{w_0}$ has been 
elusive, so far. In a sense to be made precise in section \ref{section:paths_on_the_edge}, it is a path on the boundary of our geometric crystal.

In that section, we detail the properties of geometric Pitman transforms and examine the behavior of paths as we move to the edges of a crystal. Such an analysis is crucial in parametrizing geometric path crystals (section \ref{section:parametrizations_in_path_model}). There, we establish a commutative diagram (figure \ref{fig:parametrizations_diagram}) that shows that parametrizations in the path model and the group picture are parallel. Only then, our choices in parametrizing the group picture appear as natural.

Thanks to the parametrizations of geometric path crystals, we are able to exhibit a certain number of isomorphism results in section \ref{section:isom_results}. Theorem \ref{thm:static_geometric_rs} tells us that, in fact, the projection $p$ restricts to an isomorphism on connected components. Also, we have the geometric version of Littelmann's independence theorem: the crystal structure of a crystal generated by a path $\pi$ depends only on the end point of the dominant path $\Tc_{w_0} \pi$ (Theorem \ref{thm:geometric_littelmann_independence}). Finally, the group picture is ``minimal'' in the sense that there are no automorphisms in the group picture aside from the identity (Theorem \ref{thm:minimality}). 

In the end, for every $T>0$, we obtain a map which is a bijection onto its image (Theorem \ref{thm:dynamic_rs_correspondence}):
$$ \begin{array}{cccc}
  \Cc_0\left( [0, T], \afrak \right) & \longrightarrow & \left( \Bc   , \Cc^{high}_{w_0}\left( (0, T], \afrak \right) \right) \\
  \pi                                & \mapsto	        & \left( p(\pi), \Tc_{w_0} \pi \right)\\
  \end{array}
$$
Due to the importance of this map, we explain in section \ref{section:rs_correspondences} why this bijection is the geometric counterpart of the Robinson-Schensted correspondence (see \cite{bib:Fulton97}). There, we detail the classical Robinson-Schensted correspondence, its generalization thanks to combinatorial crystals and the continuous tropical correspondence. It is only then that we can state the geometric Robinson-Schensted correspondence in a form similar to its predecessors. We feel that this review is necessary as it is not explicit in the litterature. The path $\pi$ plays the role of a word which is inserted into a geometric crystal element by solving the fore mentioned differential equation. The recording tableau keeping track of the shape is the highest weight path $\Tc_{w_0} \pi$. 

Finally, in section \ref{section:involutions}, we justify why the antimorphisms we introduced as the Kashiwara and Schützenberger involutions are indeed as claimed.

In appendix \ref{appendix:whittaker_model}, when $\pi$ is a Brownian motion, we explain why the previous left-invariant differential equation driven by $\pi$ gives a Markov process whose infinitesimal generator is closely related to the Casimir element in Kostant's Whittaker model. In a way, the present work is the algebraic part of a project we continue in the paper \cite{bib:chh14c}. There, we perform the geometric Robinson-Schensted correspondence with a Brownian path as input and describe the output. Decorations of geometric crystals will naturally appear.

\section{Preliminaries}
\label{section:preliminaries}

Let us define the (mostly standard) objects we will be using. The reader familiar with representation theory, including the Lusztig's canonical bases, can skip this section.

\subsection{Classical Lie theory}

\paragraph{On the structure of Lie algebras (\cite{bib:Humphreys72}):}
Let $\left( \gfrak, \left[ \cdot, \cdot \right] \right)$ be a complex semi-simple Lie algebra of rank $r$. The Cartan subalgebra is a maximal abelian subalgebra denoted by $\hfrak \approx \C^r$. The Lie bracket defines the adjoint action $ad: \gfrak \longrightarrow End(\gfrak)$. For $x, y \in \gfrak$, $\ad(x)(y) = [x,y]$. A fundamental idea in the classification of semi-simple Lie algebras is that the adjoint action is codiagonalizable, once restricted to $\hfrak$. Hence the root-space decomposition:
$$\gfrak = \bigoplus_{\beta \in \{0\} \sqcup \Phi} \gfrak_{\beta} = \gfrak_0 \oplus \bigoplus_{\beta \in \Phi^+} \left( \gfrak_{\beta} \oplus \gfrak_{-\beta} \right)$$
where $\gfrak_0 = \hfrak$ and $\Phi \subset \hfrak^*$ is the set of roots. Root spaces are one dimensional:
$$\gfrak_\beta = \left\{ x \in \gfrak \ | \ \forall h \in \hfrak, \ad(h)(x) = \beta(h)x \right\}$$

The set of roots can be written as a disjoint union of positive and negative roots $\Phi = \Phi^+ \bigsqcup \Phi^-$, with $\Phi^- := -\Phi^+$. This choice uniquely determines $\Delta = (\alpha_i)_{i \in I} \subset \Phi^+$ a simple system such that every positive root is a sum with positive integer coefficients of simple roots - and reciprocally, a simple system uniquely determines a positive system. Moreover, the simple system $\Delta$ forms a basis of $\hfrak^*$. When convenient, we will index simple roots by $I = \left\{ 1, 2, \dots, r \right\}$. The choice of a simple system $\Delta$ fixes an open Weyl chamber:
$$C := \left\{ x \in \afrak | \forall \alpha \in \Delta, \alpha(x) > 0 \right\}$$

The Cartan subalgebra has a decomposition $\hfrak = \afrak + i \afrak$ with $\afrak$ chosen to be the real subspace of $\hfrak$ where roots are real valued. By Cartan's criterion, since $\gfrak$ is semi-simple, the Killing form is non-degenerate. Its restriction to $\hfrak$ is in fact a scalar product written $\left( \cdot, \cdot \right)$. In the identification of $\hfrak$ and $\hfrak^*$ thanks to the Killing form, it is customary to write the coroot $\beta^\vee$ as $\beta^\vee = \frac{2 \beta}{\left( \beta, \beta \right)}$ for $\beta \in \Phi$.

For each positive root $\beta \in \Phi^+$, we can choose an $\mathfrak{sl}_2$-triplet $(e_\beta, f_\beta, h_\beta = \beta^\vee) \in \gfrak_\beta \times \gfrak_{-\beta} \times \hfrak$ such that $[e_\beta, f_\beta] = h_\beta$. $(e_\alpha, f_\alpha, h_\alpha)_{\alpha \in \Delta}$ is the set of simple $\mathfrak{sl}_2$-triplets, also known as Chevalley generators. Such a choice gives rise to Lie algebra homomorphisms $\phi_\beta: \mathfrak{sl}_2 \longrightarrow \gfrak$ such that:
$$\left\{ \begin{array}{ll}
\phi_\beta\left( x \right) = e_\beta\\
\phi_\beta\left( y \right) = f_\beta\\
\phi_\beta\left( h \right) = h_\beta\\
\end{array} \right.$$
where $x = \left( \begin{array}{cc} 0 & 1 \\ 0 &  0 \end{array} \right) $, $y = \left( \begin{array}{cc} 0 & 0 \\ 1 &  0 \end{array} \right)$, $ h = \left( \begin{array}{cc} 1 & 0 \\ 0 & -1 \end{array} \right) $.

Relations between Chevalley generators are entirely encoded by the Cartan matrix $A = \left( a_{i,j} = \alpha_i\left( \alpha_j^\vee \right) \right)_{1 \leq i,j \leq r} \in M_n(\Z)$, and therefore characterizes the Lie algebra (theorem 2.111 in \cite{bib:Knapp02}). The dual Lie algebra $\gfrak^\vee$ is the semi-simple algebra with Cartan matrix the transpose of $A$.

The fundamental weights $\left( \omega_\alpha \right)_{\alpha \in \Delta}$ form the dual basis of simple coroots. They form a $\Z$-basis of the weight lattice:
$$ P := \left\{ x \in \hfrak^* \ | \ \forall \alpha \in \Delta, x(h_\alpha) \in \Z \right\}
      = \bigoplus_{\alpha \in \Delta} \Z \omega_{\alpha}$$
The dominant weights are $ P^+ := \bigoplus_{\alpha \in \Delta} \N \omega_{\alpha}$. Similarly, define the fundamental coweights $P^{\vee} := \left( \omega_\alpha^\vee \right)_{\alpha \in \Delta} \subset \afrak$ as the dual basis of simple roots.

The Lie algebra $\gfrak$ has a triangular decomposition $\gfrak = \nfrak \oplus \hfrak \oplus \ufrak$ where $\ufrak$ ( resp. $\nfrak$) is the algebra generated by the $(e_\beta)_{\beta \in \Phi^+}$ (resp. $(f_\beta)_{\beta \in \Phi^+}$). Moreover, define the pair of opposite Borel subalgebras $\left( \mathfrak{b}, \mathfrak{b^+} \right)$ by $\mathfrak{b} := \nfrak \oplus \hfrak $ and $\mathfrak{b^+} := \hfrak \oplus \ufrak $.

\paragraph{On the structure of Lie groups (\cite{bib:Springer09}):}
Let $G$ be a simply-connected complex semi-simple Lie group with Lie algebra $\gfrak$.  The Langlands dual $G^\vee$ is an adjoint semi-simple complex Lie group with Lie algebra $\gfrak^\vee$. The subgroups $H$, $N$, $U$, $B$ and $B^+$ are the subgroups of $G$ with Lie subalgebras $\hfrak$, $\nfrak$, $\ufrak$, $\bfrak$ and $\bfrak^+$. $H$ is a maximal torus. $N$ and $U$ are referred to as the lower and upper unipotent subgroups, while $B$ and $B^+$ are the lower and upper Borel subgroups. The exponential map is denoted $\exp: \gfrak \rightarrow G$ and the adjoint representation of the Lie group $G$ is written $\Ad: G \longrightarrow GL(\gfrak)$.

The exponential map $\exp: \gfrak \rightarrow G$ lifts the homomorphisms $\left( \phi_\beta \right)_{\beta \in \Phi^+}$ from the Lie algebra $\gfrak$ to the group $G$: each $\phi_\beta$ gives rise at the group level to a Lie group homomorphism that embed $SL_2$ in $G$ and that will be denoted in the same way. The following notations are common for $t \in \C$:
\begin{align*}
 t^{h_\beta} & = e^{ \log(t) h_\beta} = \phi_\beta\left(  \left( \begin{array}{cc} t & 0 \\ 0 & t^{-1} \end{array}\right) \right), t \neq 0 \\
 x_\beta(t) & = e^{ t e_\beta} = \phi_\beta\left(  \left( \begin{array}{cc} 1 & t \\ 0 & 1 \end{array}\right) \right)\\
 y_\beta(t) & = e^{ t f_\beta} = \phi_\beta\left(  \left( \begin{array}{cc} 1 & 0 \\ t & 0 \end{array}\right) \right)\\
 x_{-\beta}(t) & = y_\beta(t) t^{-h_\beta} = \phi_\beta\left(  \left( \begin{array}{cc} t^{-1} & 0 \\ 1 & t \end{array}\right) \right)\\
 y_{-\beta}(t) & = t^{-h_\beta} x_\beta(t) = \phi_\beta\left(  \left( \begin{array}{cc} t^{-1} & 1 \\ 0 & t \end{array}\right) \right)
\end{align*}
Also, if $\gamma \in P$ is a weight, it lifts to a character of the torus $H$ whose value on $a \in H$ is denoted $a^\gamma$.

The Bruhat decomposition states that $G$ is the disjoint union of cells:
$$ G = \bigsqcup_{\omega \in W} B^+ \omega B^+  = \bigsqcup_{\tau \in W} B \tau B^+ $$
In the largest opposite Bruhat cell $B B^+ = N H U$, every element $g$ admits a unique Gauss decomposition in the form $g = n a u$ with $ n \in N$, $a \in H$, $u \in U$. In the sequel, we will write $g = [g]_- [g]_0 [g]_+$, $[g]_- \in N$, $[g]_0 \in H$ and $[g]_+ \in U$ for the Gauss decomposition. Also $[g]_{-0} := [g]_- [g]_0$ and $[g]_{0+} := [g]_0 [g]_+$. Useful identities that can be proven writing the full Gauss decomposition in two forms then identifying terms, when they exist:
\begin{eqnarray}
  \label{lbl:gauss1}
  \forall (g_1, g_2) \in N H U \times N H U, [g_1 g_2]_{0+} & = & [ [g_1]_{0+} g_2]_{0+}
\end{eqnarray}
\begin{eqnarray}
  \label{lbl:gauss2}
  \forall (g, a) \in N H U \times H, [g a]_{+} & = & a^{-1} [g]_{+} a
\end{eqnarray}

\paragraph{On the representation theory and highest weight modules:}
We are only concerned by complex finite dimensional modules. Since every homomorphism of Lie algebras lifts to a homomorphism of the corresponding simply connected Lie groups, every $\gfrak$-module $V$ lifts to a unique Lie group representation $G \rightarrow GL(V)$. 

Every module $V$ has a weight space decomposition $ V = \bigoplus_{\mu \in P} V_\mu$ where $V_\mu = \left\{ v \in V \ | \forall h \in \hfrak, h v = \mu(h) v \right\}$. The non-zero $V_\mu$ are called weight spaces, and their vectors weight vectors of weight $\mu$. A highest weight vector in $V$ is a non-zero weight vector $v$ such that $\ufrak \cdot v = \{0\}$. It is well known that simple modules are highest weight modules indexed by dominant weights $\lambda \in P^+$. To every $\lambda \in P^+$ corresponds a unique highest weight module $V(\lambda)$. The highest weight vector $v_\lambda \in V(\lambda)$ is unique up to scalar multiplication. 

\subsection{Reflection groups and root systems}

\paragraph{Weyl group:} 
The Weyl group of $G$ is defined as $W := \textrm{Norm}_G( H )/ H $, $\textrm{Norm}_G( H )$ being the normalizer of $H$ in $G$. It acts on the torus $H$ by conjugation and hence on $\hfrak$. To every linear form $\beta \in \hfrak^*$, define the associated reflection $s_\beta$ on $\hfrak$ with:
$$\forall \lambda \in \hfrak,  s_\beta \lambda := \lambda - \beta\left( \lambda \right) \beta^\vee $$
The reflections $(s_\alpha)_{\alpha \in \Delta}$ are called simple reflections and they generate a finite Coxeter group isomorphic to $W$. Define $m_{s,s'}$ as the order of the element $s s'$. For $w \in W$, a reduced expression is given by writing $w$ as product of simple reflections with minimal length:
$$ w = s_{i_1} s_{i_2} \dots s_{i_{\ell}}$$
A reduced word is such a tuple ${\bf i} = \left( i_1, \dots, i_{\ell} \right)$ and the set of reduced words for $w \in W$ is denoted by $R(w)$. Since all reduced expressions have necessarily the same length, it defines unambiguously the length function $\ell: W \rightarrow \N$. The unique longest element is denoted by $w_0$ and we set $m=\ell(w_0)$.

If $\left( s, s' \right) \in W \times W$ are simple reflections, a braid relationship in $W$ is the equality between $d=m_{s,s'}$ terms:
$$ s s' s \dots = s' s s' \dots \ $$
with $m_{s,s'}$ being the order of $ss'$. A braid move or a $d$-move occurs when substituting $ s s' s \dots$ for $s' s s' \dots $ within a reduced word. An important theorem is the following:
\begin{lemma}[Matsumoto \cite{bib:Matsumoto64} and Tits \cite{bib:Tits69}]
\label{lemma:tits_lemma}
Two reduced expressions of the same $w \in W$ can be derived from each other using braid moves. 
\end{lemma}

\paragraph{Representatives:}
A common set of representatives in $G$ for the generating reflections $(s_i)_{i \in I} \subset W$ is:
$$ \bar{s}_i := \phi_i\left(  \left( \begin{array}{cc} 0 & -1 \\ 1 & 0 \end{array}\right) \right) = e^{-e_i} e^{f_i} e^{-e_i} = e^{f_i} e^{-e_i} e^{f_i}$$
Another common choice is:
$$ \bar{\bar{s}}_i := \bar{s}_i^{-1} = \phi_i\left(  \left( \begin{array}{cc} 0 & 1 \\ -1 & 0 \end{array}\right) \right) = e^{e_i} e^{-f_i} e^{e_i} = e^{-f_i} e^{e_i} e^{-f_i}$$

\begin{lemma}[ \cite{bib:KacPeterson} Lemma 2.3 ]
The Weyl group representatives $\bar{s}_i$ (resp. $\bar{\bar{s}}_i$) satisfy the braid relationships:
$$ \bar{s}_i \bar{s}_j \bar{s}_i \dots = \bar{s}_j \bar{s}_i \bar{s}_j \dots $$
It allows us to define unambiguously $\bar{w} = \bar{u} \bar{v}$ if $w = uv$ and $\ell(w) = \ell(u) + \ell(v)$.
\end{lemma}

However they do not form a presentation of the Weyl group, since for example $ (\bar{s}_i)^2 = \phi_i(-id) \neq id $. The representative $\bar{w}_0$ of the longest element $w_0$ has an important property we will later use:
\begin{proposition}[\cite{bib:BBBR92} Lemma 4.9]
\label{proposition:w_0_action_ad}
Via the $\Ad$ action, $\bar{w}_0$ acts on the Chevalley generators as:
$$ \forall \alpha \in \Delta, \Ad\left( \bar{w}_0 \right)\left( e_\alpha \right) = - f_\alpha$$
\end{proposition}

\paragraph{Positive roots enumerations:} For $w \in W$, the inversion set of $w$ is defined as $\Inv\left(w\right) := \left\{ \beta \in \Phi^+, w\beta \in \Phi^{-} \right\} $. The length function has a characterization as the cardinal of the inversion set i.e $\ell(w) = \left|\Inv(w)\right|$ for every $w \in W$. Moreover:
\begin{lemma}[see \cite{bib:Humphreys90}]
\label{lemma:positive_roots_enumeration}
Let $(i_1, \dots, i_k)$ be a reduced expression of $w \in W$. Then for $j=1, 2, \dots, k$:
$$\beta_{{\bf i}, j} := s_{i_1} \dots s_{i_{j-1}} \alpha_{i_j}$$
produces all the positive roots in $\Inv(w)$. For $w = w_0$, it produces all positive roots.
\end{lemma}
When the chosen reduced expression is obvious from context, we will drop the subscript ${\bf i}$.  There is a very simple yet very useful identity that we will use several times. 
\begin{lemma}[Corollary 1.3.22 \cite{bib:Kumar02}]
\label{lbl:kumar} 
Let $\lambda \in \mathfrak{a}$ and $w = s_{i_1} \dots s_{i_\ell}$ a reduced expression for a Weyl group element $w$ of length $\ell$. If $\left( \beta_1, \dots, \beta_\ell \right)$ is the associated positive roots enumeration, then we have:
\begin{eqnarray}
\lambda - w \lambda & = & \sum_{k=1}^\ell \alpha_{i_k}(\lambda) \beta_k^\vee\\
\lambda - w^{-1} \lambda & = & \sum_{k=1}^\ell \beta_{k}(\lambda) \alpha_{i_k}^\vee
\end{eqnarray}
\end{lemma}

\subsection{Involutions}
Since $w_0 \in W$ transforms all simple positive roots to simple negative roots, there is an involution on $\Delta$ (or equivalently the index set $I$) denoted by $*$ such that:
$$ \forall \alpha \in \Delta, \beta^* = -w_0 \alpha$$

We define the following group antimorphisms of $G$ by their actions on a torus element $a \in H$ and the one-parameters subgroups generated by the Chevalley generators. For convenience, we also give their action at the level of the Lie algebra $\gfrak$.
\begin{itemize}
 \item The transpose:
       $$ \begin{array}{ccc}
           a^T = a          & x_i(t)^T     = y_i(t) & y_i(t)^T     = x_i(t)
          \end{array}
       $$
       $$ \forall \alpha \in \Delta, 
          \begin{array}{ccc}
           h_\alpha^T = h_\alpha &  e_\alpha^T = f_\alpha & f_\alpha^T = e_\alpha
          \end{array}
       $$
 \item The positive inverse or Kashiwara involution ( \cite{bib:Kashiwara91} (1.3) ):
       $$ \begin{array}{ccc}
           a^\iota = a^{-1} & x_i(t)^\iota = x_i(t) & y_i(t)^\iota = y_i(t)
          \end{array}
       $$
       $$ \forall \alpha \in \Delta, 
          \begin{array}{ccc}
           h_\alpha^\iota = -h_\alpha &  e_\alpha^\iota = e_\alpha & f_\alpha^\iota = f_\alpha
          \end{array}
       $$
 \item Sch\"utzenberger involution:
       $$S(x) = \bar{w}_0 \left( x^{-1} \right)^{\iota T} \bar{w}_0^{-1} = \bar{w}_0^{-1} \left( x^{-1} \right)^{\iota T} \bar{w}_0$$
       It acts as ( relation 6.4 in \cite{bib:BZ01} or using proposition \ref{proposition:w_0_action_ad}):
       $$ S\left( x_{i_1}(t_1) \dots x_{i_q}(t_q) \right) = x_{i_q^*}(t_q) \dots x_{i_1^*}(t_1)$$
       $$ \forall \alpha \in \Delta, 
          \begin{array}{ccc}
           S(h_\alpha) = h_{\alpha^*} &  S(e_\alpha) = e_{\alpha*} & S(f_\alpha) = f_{\alpha*}
          \end{array}
       $$
       Notice that $S = \iota \circ S \circ \iota$
\end{itemize}
More is said in section \ref{section:involutions} about involutions, detailing their effect on crystals and justifying the terminology.

\subsection{On Lusztig's canonical basis}
\label{subsection:Lusztig_canonical_basis}
In the sequel, we will never use Lusztig's canonical basis in itself. However, we will be extensively interested in its parametrizations. As such it is important to review this mathematical object. We will remain elusive concerning its precise definition, though. Consider the quantum group $\Uc_q\left( \gfrak \right)$. In \cite{bib:Lusztig90a} and \cite{bib:Lusztig90b}, Lusztig introduced a basis $\Bfrak\left( \infty \right)$ of the subalgebra $\Uc_q\left( \nfrak \right)$ called the canonical basis.

Later, in an independent work, Kashiwara introduced crystal bases for integrable modules whose combinatorics are particularly simple at $q=0$. Thanks to the work of Grojnowski and Lusztig \cite{bib:GL92}, Lusztig's canonical basis and Kashiwara's global crystal basis are in fact the same.

\paragraph{Parametrizations:}
There are two common parametrizations of the canonical basis. Both depend on a choice of reduced word for the longest Weyl group element $w_0$. Let ${\bf i} \in R(w_0)$ and recall that $m = \ell(w_0)$.

The Lusztig parametrization is a bijection (Proposition 42.1.5 in \cite{bib:Lusztig93} for instance):
$$
 \begin{array}{llll}
\left[x_{\bf i}\right]: & \N^m              & \rightarrow & \Bfrak\left( \infty \right)\\
                        & (t_1, \dots, t_m) & \mapsto     & \left[x_{\bf i}\right](t_1, \dots, t_m)
 \end{array}
$$
The string (or Kashiwara) parametrization uses the integer points of a convex polyhedral cone $\Cc_{\bf i} \subset \R_+^m$ (proposition 1.5 in \cite{bib:Littelmann98}) called the string cone. It is given by a bijection:
$$
 \begin{array}{llll}
\left[x_{\bf-i}\right]: & \Cc_{\bf i} \cap \N^m & \rightarrow & \Bfrak\left( \infty \right)\\
                        & (c_1, \dots, c_m)     & \mapsto      & \left[x_{\bf-i}\right](c_1, \dots, c_m)
 \end{array}
$$

\paragraph{Kashiwara operators (see \cite{bib:HongKang}):}
These are linear operators on $\Uc_q(\nfrak)$. Their action on the canonical basis $\Bfrak\left( \infty \right)$ can be independently defined but we only give a description in the parametrizations associated to ${\bf i} \in R(w_0)$. Let $\alpha = \alpha_{i_1}$ be first simple root in ${\bf i}$ and ${\bf 0}$ be the zero element in $\Uc_q(\nfrak)$. The Kashiwara operators $\tilde{e}_\alpha, \tilde{f}_\alpha: \Bfrak\left( \infty \right) \rightarrow \Bfrak\left( \infty \right) \bigsqcup \{ \bf 0 \}$ satisfy:

\begin{align}
\label{eqn:kashiwara_operator_f_1}
\tilde{f}_\alpha\left( [x_{\bf i}](t_1, \dots, t_m) \right) & = [x_{\bf i}](t_1 + 1, \dots, t_m)
\end{align}

\begin{align}
\label{eqn:kashiwara_operator_f_2}
\tilde{f}_\alpha\left( [x_{\bf-i}](c_1, \dots, c_m) \right) & = [x_{\bf-i}](c_1 + 1, \dots, c_m)
\end{align}

\begin{align}
\label{eqn:kashiwara_operator_e_1}
\tilde{e}_\alpha\left( [x_{\bf i}](t_1, \dots, t_m) \right) & = [x_{\bf i}](t_1 - 1, \dots, t_m) \textrm{ or } {\bf 0} \textrm{ if } t_1 = 0
\end{align}

\begin{align}
\label{eqn:kashiwara_operator_e_2}
\tilde{e}_\alpha\left( [x_{\bf-i}](c_1, \dots, c_m) \right) & = [x_{\bf-i}](c_1 - 1, \dots, c_m) \textrm{ or } {\bf 0} \textrm{ if } c_1 = 0
\end{align}

\begin{rmk}
 The Kashiwara operators are quasi-inverses of each other, in the sense that:
$$ \forall b \in \Bfrak\left( \infty \right), \tilde{e}_\alpha \circ \tilde{f}_\alpha(b) = b$$
$$ \forall b \in \Bfrak\left( \infty \right), \tilde{e}_\alpha(b) \neq {\bf 0} \Rightarrow \tilde{f}_\alpha \circ \tilde{e}_\alpha(b) = b$$
\end{rmk}

\paragraph{Compatibility properties:}
The desirable properties of the canonical basis are compatibility properties regarding highest weight modules. Fix $\lambda \in P^+$ and consider the highest weight module $V(\lambda)$ - for the quantum group. Recall that $v_\lambda$ is a highest weight vector. It is well known that the linear map:
$$
 \begin{array}{llll}
\pi_\lambda: & \Uc_q(\nfrak) & \rightarrow & V(\lambda)\\
             & n             & \mapsto     & n v_\lambda
 \end{array}
$$
has the kernel:
$$ \ker\left( \pi_\lambda \right) = \sum_{\alpha \in \Delta} \Uc_q(\nfrak) f_\alpha^{\lambda(\alpha^\vee)+1}$$
Therefore, the image of $\Bfrak\left( \infty \right)$ by $\pi_\lambda$ is:
$$ \Bfrak\left( \lambda \right) := \Bfrak\left( \infty \right) - \Bfrak\left( \infty \right) \cap \ker \pi_\lambda$$
Define $\Bfrak$ as the disjoint union $\Bfrak := \bigsqcup_{\lambda \in P^+} \Bfrak\left( \lambda \right)$. If $\mu \leq \lambda$ in the convex ordering given by the Weyl chamber, we clearly have $\ker\left( \pi_\lambda \right) \subset \ker\left( \pi_\mu \right)$, hence a natural injection $\Bfrak\left( \mu \right) \hookrightarrow \Bfrak\left( \lambda \right)$. We see that 
the canonical basis is a direct limit $\Bfrak\left( \infty \right) = \varinjlim \Bfrak\left( \lambda \right)$, explaining the notation $\Bfrak\left( \infty \right)$. We have:
\begin{thm}[ \cite{bib:Lusztig93} ]
 $\Bfrak\left( \lambda \right) v_\lambda$ is a basis for $V(\lambda)$ made of weight vectors.
\end{thm}
Therefore, the subsets $\left( \Bfrak\left( \lambda \right), \lambda \in P^+ \right)$ of the canonical basis $\Bfrak\left( \infty \right)$, once identified with $\Bfrak\left( \lambda \right) v_\lambda$, form compatible bases of highest weight modules.

For $b \in \Bfrak\left( \lambda \right)$, denote by $\gamma(b)$ its weight in the representation $V(\lambda)$. Let us record the expressions for the weight in terms of parametrizations. If $b = \left[ x_{\bf i} \right](t_1, \dots, t_m)$ then, using the positive roots enumeration $(\beta_{{\bf i}, 1}, \dots, \beta_{{\bf i}, m})$ associated to ${\bf i}$:
\begin{align}
\label{eqn:weight_lusztig}
\gamma(b) & = \lambda - \sum_{j=1}^m t_j \beta_{{\bf i}, j}
\end{align}
If $b = \left[x_{\bf-i}\right](c_1, \dots, c_m)$ then:
\begin{align}
\label{eqn:weight_string}
\gamma(b) & = \lambda - \sum_{j=1}^m c_j \alpha_{i_j}
\end{align}

Thanks to the weight map $\gamma$ and the Kashiwara operators, every set $\Bfrak\left( \lambda \right)$ is given the structure of a combinatorial crystal (Definition 4.5.1 in \cite{bib:HongKang}). We now introduce the abstract geometric analogues.

\section{Abstract geometric crystals}
\label{section:abstract_geom_crystals}

\begin{definition}
\label{def:crystal}
An abstract geometric crystal is a topological space $L$ equipped with the following continuous structural maps:
\begin{itemize}
 \item a weight map $\gamma: L \rightarrow \afrak$.
 \item $\varepsilon_\alpha, \varphi_\alpha: L \rightarrow \R$ defined for every $\alpha \in \Delta$
 \item $e^c_\alpha: L \rightarrow L$, $c \in \R$, $\alpha \in \Delta$
\end{itemize}
and satisfying the following properties for $\pi \in L$:
\begin{itemize}
 \item[(C1)] $\varphi_\alpha(\pi) = \varepsilon_\alpha(\pi) + \alpha\left( \gamma(\pi) \right)$
 \item[(C2)] $\gamma\left( e^c_\alpha \cdot \pi \right) = \gamma\left( \pi \right) + c \alpha^\vee$
 \item[(C3)] $\varepsilon_\alpha\left( e^c_\alpha \cdot \pi \right) = \varepsilon_\alpha\left( \pi \right) - c$
 \item[(C3')]$\varphi_\alpha\left( e^c_\alpha \cdot \pi \right) = \varphi_\alpha\left( \pi \right) + c$
 \item[(C4)] $e^._\alpha$ are actions: $e^0 = id$ and $e^{c+c'}_\alpha = e^c_\alpha \cdot e^{c'}_\alpha$
\end{itemize}
 Clearly, (C3) and (C3') are equivalent once (C1) and (C2) are assumed.
\end{definition}

Here, unlike the standard object defined by Kashiwara, there is no ghost element ${ \bf 0 }$, the crystal has free actions and coefficients are real. One could use the term 'free continuous crystal'. Later on, we will also require a certain type of commutation relations between the actions $(e^._\alpha)_{\alpha \in \Delta}$, identified as Verma relations. Berenstein and Kazhdan refer to such structure as a 'pre-crystal' if Verma relations are not available.

\paragraph{Generated crystals:} Given a subset $S$ of a crystal $L$, define $\langle S \rangle$ as the smallest subcrystal of $L$ containing $S$, endowed with the subspace topology. Since intersections of crystals are crystals, we can define it as:
\begin{align*}
\langle S \rangle & := \bigcap_{ \substack{S \subset \Cc \\ \textrm{ subcrystal of L } } } \Cc\\
                  & = \left\{ e^{c_1}_{\alpha_{i_1}} \cdot e^{c_2}_{\alpha_{i_2}} \cdot \dots e^{c_l}_{\alpha_{i_k}} \cdot x | x \in S, k \in \N, \left( c_1, \dots, c_k\right) \in \R^k , \left( i_1, i_2, \dots, i_k\right) \in I^k  \right\}
\end{align*}

\paragraph{Crystal connected components:} This notion is not to be confused with the topological notion of connectedness. A crystal is connected if given two elements $x$ and $y$, there is $k \in \N$, $\left( c_1, c_2, \dots, c_k\right) \in \R^k$ and $\left( i_1, i_2, \dots, i_k\right) \in I^k$ such that:
$$ y = e^{c_1}_{\alpha_{i_1}} \cdot e^{c_2}_{\alpha_{i_2}} \cdot \dots e^{c_k}_{\alpha_{i_k}} \cdot x  $$
Any crystal is the disjoint union of its connected components, since 'being connected' is an equivalence relation. Also a connected component is generated by any of its elements, and connected components are subcrystals.

\paragraph{Morphism of crystals:} A morphism of crystals is a map $\psi$ that preserves the structure. It is an isomorphism if invertible, and the inverse map is a morphism.

\paragraph{h-Tensor product of crystals:} In the sequel, the crystal structure itself will depend on a parameter $h \geq 0$. For $h\geq0$, we define the $h$-tensor product of two crystals $B_1$ and $B_2$ as the set $ B_1 \otimes_h B_2 = B_1 \times B_2$ endowed with the product topology and the following structural maps. For $h>0$, they are given by:
\begin{itemize}
 \item $\gamma\left( b_1 \otimes_h b_2 \right) = \gamma\left( b_1 \right) + \gamma\left( b_2 \right)$
 \item $\varepsilon_\alpha\left( b_1 \otimes_h b_2 \right) = \varepsilon_\alpha\left( b_1 \right) + 
         h \log\left( 1 + e^{\frac{ \varepsilon_\alpha(b_2) - \varphi_\alpha(b_1) }{h}}\right)$
 \item $\varphi_\alpha\left( b_1 \otimes_h b_2 \right) = \varphi_\alpha\left( b_2 \right) + 
         h \log\left( 1 + e^{\frac{ \varphi_\alpha(b_1) - \varepsilon_\alpha(b_2) }{h}}\right)$
 \item The actions are defined as 
       $e^c_\alpha\left( b_1 \otimes_h b_2 \right) = \left( e^{c_1}_\alpha \cdot b_1 \right) \otimes_h \left( e^{c_2}_\alpha \cdot b_2 \right)$\\
       where
       \begin{align*}
        c_1 & = h \log\left( \frac{ e^{\frac{c + \varphi_\alpha(b_1) }{h}} + e^{\frac{   \varepsilon_\alpha(b_2) }{h}} }
                                  { e^{\frac{    \varphi_\alpha(b_1) }{h}} + e^{\frac{   \varepsilon_\alpha(b_2) }{h}} } \right)\\
            & = h \log\left( e^{\frac{c}{h}} + e^{\frac{\varepsilon_\alpha(b_2) - \varphi_\alpha(b_1)}{h}} \right) - 
                h \log\left( 1 + e^{\frac{\varepsilon_\alpha(b_2) - \varphi_\alpha(b_1)}{h}} \right)\\
        c_2 & = h \log\left( \frac{ e^{\frac{    \varphi_\alpha(b_1) }{h}} + e^{\frac{   \varepsilon_\alpha(b_2) }{h}} }
                                  { e^{\frac{    \varphi_\alpha(b_1) }{h}} + e^{\frac{-c+\varepsilon_\alpha(b_2) }{h}} } \right)\\
            & = -h \log\left( e^{\frac{-c}{h}} + e^{\frac{\varphi_\alpha(b_1) - \varepsilon_\alpha(b_2)}{h}} \right) +
                h \log\left( 1 + e^{\frac{\varphi_\alpha(b_1) - \varepsilon_\alpha(b_2)}{h}} \right)\\
       \end{align*}
\end{itemize}

\begin{rmk}
 $$ c_1 + c_2 = c$$
\end{rmk}

\begin{rmk}
The structural maps in the $h$-tensor product are also continuous in $h$. By letting the parameter $h \rightarrow 0$, one recovers the same crystalline axioms for tensor product as in the original work of Kashiwara (Definition 4.5.3 in \cite{bib:HongKang}):
\begin{itemize}
 \item $\gamma\left( b_1 \otimes b_2 \right) = \gamma\left( b_1 \right) + \gamma\left( b_2 \right)$
 \item $\varepsilon_\alpha\left( b_1 \otimes b_2 \right) = \varepsilon_\alpha\left( b_1 \right) + 
         \left( \varepsilon_\alpha(b_2) - \varphi_\alpha(b_1) \right)^+$
 \item $\varphi_\alpha\left( b_1 \otimes b_2 \right) = \varphi_\alpha\left( b_2 \right) + 
         \left( \varphi_\alpha(b_1) - \varepsilon_\alpha(b_2)\right)^+$
 \item $e^c_\alpha\left( b_1 \otimes b_2 \right) = \left( e^{c_1}_\alpha \cdot b_1 \right) \otimes \left( e^{c_2}_\alpha \cdot b_2 \right)$\\
       where
       \begin{align*}
        c_1 & = \max\left( c, \varepsilon_\alpha(b_2) - \varphi_\alpha(b_1) \right) - 
                \left( \varepsilon_\alpha(b_2) - \varphi_\alpha(b_1) \right)^+\\
        c_2 & = \min\left( c, \varepsilon_\alpha(b_2) - \varphi_\alpha(b_1) \right) +
                \left( \varphi_\alpha(b_1) - \varepsilon_\alpha(b_2) \right)^+\\
       \end{align*}
\end{itemize}
\end{rmk}

It is easy to check the following:
\begin{proposition}
\label{proposition:tensor_product_is_crystal}
For all $h\geq 0$, $B_1 \otimes_h B_2$ is a crystal.
\end{proposition}
\begin{proof}
Let us verify axioms for crystals from $(C_1)$ to $(C4)$, in the case $h>0$, then $h=0$ will follow by a limit argument.

$(C1)$ 
\begin{align*}
&  \varphi_\alpha\left( b_1 \otimes_h b_2 \right) - \varepsilon_\alpha\left( b_1 \otimes_h b_2 \right)\\
& = \varphi_\alpha\left( b_2 \right) + h \log\left( 1 + e^{\frac{ \varphi_\alpha(b_1) - \varepsilon_\alpha(b_2) }{h}}\right)\\
&  - \varepsilon_\alpha\left( b_1 \right) - h \log\left( 1 + e^{\frac{ \varepsilon_\alpha(b_2) - \varphi_\alpha(b_1) }{h}}\right)\\
& = \varphi_\alpha\left( b_2 \right) - \varepsilon_\alpha\left( b_1 \right)  + \varphi_\alpha\left(b_1\right) - \varepsilon_\alpha\left(b_2\right)\\
& = \alpha\left( \gamma\left(b_1\right) \right) + \alpha\left( \gamma\left(b_2\right) \right)\\
& = \alpha\left( \gamma\left(b_1 \otimes_h b_2 \right) \right)
\end{align*}

$(C2)$
\begin{align*}
& \gamma\left( e^c_\alpha \cdot (b_1 \otimes_h b_2) \right)\\
& = \gamma\left( e^{c_1}_\alpha \cdot b_1 \right) + \gamma\left( e^{c_2}_\alpha \cdot b_2 \right)\\
& = \gamma\left( b_1 \otimes_h b_2 \right) + \left( c_1 + c_2 \right)\alpha^\vee
\end{align*}
using the remark that $c=c_1+c_2$, the second axiom is checked.

$(C3)$ We will only check:
\begin{align*}
& \varepsilon_\alpha\left( e^c_\alpha \cdot (b_1 \otimes_h b_2) \right)\\
& = \varepsilon_\alpha\left( e^{c_1}_\alpha \cdot b_1 \right) + 
    h \log\left( 1 + e^{\frac{ \varepsilon_\alpha( e^{c_2}_\alpha \cdot b_2) - \varphi_\alpha( e^{c_1}_\alpha \cdot b_1) }{h}}\right)\\
& = -c_1 + \varepsilon_\alpha\left( b_1 \right) + 
    h \log\left( 1 + e^{\frac{ -c + \varepsilon_\alpha(b_2) - \varphi_\alpha(b_1) }{h}}\right)\\
& = -c_1 + \varepsilon_\alpha\left( b_1 \right) - c +
    h \log\left( e^c + e^{\frac{ \varepsilon_\alpha(b_2) - \varphi_\alpha(b_1) }{h}}\right)\\
& = \varepsilon_\alpha\left( b_1 \right) - c +
    h \log\left( 1 + e^{\frac{ \varepsilon_\alpha(b_2) - \varphi_\alpha(b_1) }{h}}\right)\\
& = -c + \varepsilon_\alpha\left( b_1 \otimes_h b_2 \right) 
\end{align*}

$(C4)$ We know that
$$ e^c_\alpha \cdot e^{c'}_\alpha \cdot \left( b_1 \otimes_h b_2 \right) = e^{c_1 + c_1'} \cdot b_1 \otimes_h e^{c_2 + c_2'} \cdot b_2 $$
{ \centering where }
\begin{align*}
c_1' & = h \log\left( \frac{ e^{\frac{c' + \varphi_\alpha(b_1) }{h}} + e^{\frac{   \varepsilon_\alpha(b_2) }{h}} }
                           { e^{\frac{    \varphi_\alpha(b_1) }{h}} + e^{\frac{   \varepsilon_\alpha(b_2) }{h}} } \right)\\
c_1  & = h \log\left( \frac{ e^{\frac{c  + \varphi_\alpha(e^{c_1'}_\alpha \cdot b_1) }{h}} + e^{\frac{   \varepsilon_\alpha(e^{c_2'}_\alpha \cdot b_2) }{h}} }
                           { e^{\frac{    \varphi_\alpha(e^{c_1'}_\alpha \cdot b_1) }{h}} + e^{\frac{   \varepsilon_\alpha(e^{c_2'}_\alpha \cdot b_2) }{h}} } \right)\\
c_2' & = c - c_1'\\
c_2  & = c - c_1
\end{align*}
We simplify:
\begin{align*}
c_1 & = h \log\left( \frac{ e^{\frac{c  + \varphi_\alpha(b_1) + c_1'}{h}} + e^{\frac{   \varepsilon_\alpha(b_2) - c_2'}{h}} }
                           { e^{\frac{    \varphi_\alpha(b_1) + c_1'}{h}} + e^{\frac{   \varepsilon_\alpha(b_2) - c_2'}{h}} } \right)\\
& = h \log\left( \frac{ e^{\frac{c + c' + \varphi_\alpha(b_1)}{h}} + e^{\frac{   \varepsilon_\alpha(b_2)}{h}} }
                           { e^{\frac{c'+ \varphi_\alpha(b_1)}{h}} + e^{\frac{   \varepsilon_\alpha(b_2)}{h}} } \right)
\end{align*}
Hence:
$$ c_1 + c_1' = h \log\left( \frac{ e^{\frac{c + c' + \varphi_\alpha(b_1)}{h}} + e^{\frac{   \varepsilon_\alpha(b_2)}{h}} }
                                  { e^{\frac{         \varphi_\alpha(b_1)}{h}} + e^{\frac{   \varepsilon_\alpha(b_2)}{h}} } \right)$$
and
$$ c_2 + c_2' = c + c' - \left( c_1 + c_1' \right)
              = h \log\left( \frac{ e^{\frac{ \varphi_\alpha(b_1)}{h}} + e^{\frac{        \varepsilon_\alpha(b_2)}{h}} }
                                  { e^{\frac{ \varphi_\alpha(b_1)}{h}} + e^{\frac{ -c-c'+ \varepsilon_\alpha(b_2)}{h}} } \right)$$
In the end:
$$ e^c_\alpha \cdot e^{c'}_\alpha \cdot \left( b_1 \otimes_h b_2 \right)
 = e^{c_1 + c_1'} \cdot b_1 \otimes_h e^{c_2 + c_2'} \cdot b_2
 = e^{c   + c   } \cdot \left( b_1 \otimes_h b_2 \right)$$

\end{proof}

\section{Geometric crystals: the group picture of Berenstein and Kazhdan}
\label{section:geom_lifting}

In this section, we introduce totally positive elements in $G$. Then we define the group picture $\Bc \subset B$ of geometric crystals, which is made of totally positive elements in $B$. For the purposes of the next sections, our aim is to specify the coordinate charts of figure \ref{fig:geom_parametrizations}, which are the geometric liftings of Lusztig's and Kashiwara's parametrizations of the canonical basis. 

It is important to have in mind that many choices are possible. Figure \ref{fig:geom_parametrizations} fits with parametrizations in the path model (figure \ref{fig:parametrizations_diagram}), which indicates that we have indeed made the right choices. Moreover, in these charts, we give explicit expressions of the weight map in subsection \ref{subsection:geom_weight_map}.

Finally, we state that $\Bc$ is an abstract geometric crystal in theorem \ref{thm:geom_crystal_is_crystal}, due to Berenstein and Kazhdan. Additional structure is also specified.

\subsection{On total positivity}

For a survey, see \cite{bib:Lusztig08}. Lusztig got interested in the subject after Kostant pointed out that the combinatorics of the canonical basis should be related with the combinatorics of total positivity. Later, he understood that totally positive varieties in $G$ have the same parametrizations as the canonical basis of the Langlands dual $G^\vee$, after a tropicalization procedure (see \cite{bib:BFZ96}, or paragraph 42.2 in \cite{bib:Lusztig93}).

There are two equivalent definitions for totally non-negative matrices in the classical case of $GL_n\left( \C \right)$.
\begin{thm}[Whitney \cite{bib:Whitney52}, Loewner \cite{bib:Lo55} ]
\label{thm:total_positivity_gln}
An invertible $n \times n$ matrix is said to be totally non-negative if all its minors are $\geq 0$, or equivalently, if it has a decomposition
$$ y_{i_1}(t_1) \dots y_{i_m}(t_m) h x_{i_1}(t_1') \dots x_{i_m}(t_m')$$
where $h$ is diagonal with positive entries, $y_i(t) = e^{t E_{i+1,i}} = I_n + t E_{i+1,i}$, $x_i(t) = e^{t E_{i,i+1}} = I_n + t E_{i,i+1}$ (Chevalley matrices) for $t \geq 0$. Moreover, the space of totally non-negative matrices can be characterized as the semi-group generated by such elements.
\end{thm}

In \cite{bib:Lusztig94}, Lusztig generalized this definition to arbitrary complex reductive groups, by taking the semi-group property as a starting point. The following sets are called totally non-negative parts of $G$.
\begin{itemize}
 \item The semi-group generated by $a \in H$ such that $a^\gamma>0$ for every weight $\gamma$: $ H_{>0} = A = exp\left( \afrak \right) = H \cap G_{\geq 0}$
 \item The semi-group generated by the $x_\alpha(t), t>0, \alpha \in \Delta$: $U_{\geq 0}$
 \item The semi-group generated by the $y_\alpha(t), t>0, \alpha \in \Delta$: $N_{\geq 0}$
 \item Finally, the totally non-negative part of G is denoted $G_{\geq 0}$ and is formed by the semi-group generated by all of them.
\end{itemize}

Lusztig proved that totally non-negative elements admit a Gauss decomposition made of totally non-negative elements, and exhibited parametrizations as products of Chevalley elements.
\begin{thm}[\cite{bib:Lusztig94} lemma 2.3]
\label{thm:totally_positive_gauss_decomposition}
Any element $g \in G_{\geq 0}$ has a unique Gauss decomposition $g = n a u$ with $n \in N_{\geq 0}$, $a \in A$ and $u \in U_{\geq 0}$.
\end{thm}

\begin{thm}[\cite{bib:Lusztig94} proposition 2.7, \cite{bib:BZ97}, Proposition 1.1]
\label{thm:lusztig_variety_params}
For any $w \in W$ with $k=\ell(w)$, every reduced word $\mathbf{i} = (i_1, \dots, i_k)$ in $R(w)$ gives rise to a parametrization of $U^{w}_{>0} := U_{\geq 0} \cap B w B$ by:
$$
 \begin{array}{cccc}
x_{\mathbf{i}}: & \R_{>0}^k & \rightarrow & U^{w}_{>0}\\
                & (t_1, \dots, t_k) & \mapsto     & x_{i_1}(t_1) \dots x_{i_k}(t_k)
 \end{array}
$$
Moreover, for ${\bf i}, {\bf i'} \in R(w)$, the maps $R_{\bf i,i'} := x_{\mathbf{i'}}^{-1} \circ x_{\mathbf{i}}: \R_{>0}^m \rightarrow \R_{>0}^m$ are rational and substraction free.
\end{thm}
Hence the name of totally non-negative varieties for the sets $U^{w}_{>0}, w \in W$. The Bruhat decomposition tells us then that the previous maps have disjoint images and cover the entire non-negative part $U_{\geq 0} = \bigsqcup_{w \in W} U^{w}_{>0}$. Of course, after transpose, one has analogous parametrizations for $N^{w}_{>0} := N_{\geq 0} \cap B^+ w B^+$.

Afterwards, Berenstein, Fomin and Zelevinsky in a series of papers (\cite{bib:BZ97, bib:FZ99, bib:BZ01}) completed the picture by defining generalized minors on semi-simple groups, allowing to define the totally positive varieties as the locus where appropriate minors are positive. This tool will be defined when needed.

\subsection{Coordinates in totally positive varieties}

\paragraph{Geometric lifting:}
Consider a rational expression in $k$ variables $a \in \Q\left( x_1, \dots, x_k \right)$ that has no minus sign. Tropicalizing $a$ to $\left[ a \right]_{trop}$ is tantamount to replacing the algebraic operations $\left( +, \times, / \right)$ by $\left( \min, +, - \right)$. A rational expression in the operations $\left( \min, +, - \right)$ is now commonly referred to as a tropical expression. Formally, if $a$ and $b$ are rational subtraction-free functions, then:
$$ \left[ a + b \right]_{trop} = \min( a, b) $$
$$ \left[ a   b \right]_{trop} = a + b $$
$$ \left[ a / b \right]_{trop} = a - b $$
$$ \left[ a \circ b \right]_{trop} = \left[ a \right]_{trop} \circ \left[ a \circ b \right]_{trop}$$

For example:
$$ \left[ \frac{t_2 t_3}{t_1 + t_3} \right]_{trop} = t_2 + t_3 - \min(t_1, t_3)$$

Geometric lifting is the general idea that computations in the tropical world using the semi-field $(\R, \min, +)$ have analogues in the  geometric world using $(\R_{>0}, +, .)$. Here, the reader will only need to have in mind that rational and substraction free maps preserve positivity and can tropicalized. Lusztig's interest in total positivity lies in the fact that totally positive varieties are geometric liftings (of parametrizations) of canonical bases.

\begin{definition}[Lusztig variety]
\label{def:geom_lusztig_variety}
Define the geometric Lusztig variety as:
$$U_{>0}^{w_0} := U \cap B w_0 B \cap G_{\geq 0}$$ 
\end{definition}

Such a name is legitimate because changes of parametrization in the $G^\vee$-canonical basis are given by tropicalizing $R_{\bf i,i'}$ (\cite{bib:BZ01}, theorem 5.2). In other words, using the notations from subsection \ref{subsection:Lusztig_canonical_basis}:
$$ \left[ R_{\bf i,i'} \right]_{trop} = \left( \left[x_{\bf i'}\right]^\vee \right)^{-1} \circ \left[x_{\bf i}\right]^\vee: \N^m \rightarrow \N^m$$
where the superscript $\vee$ indicates that the corresponding map has to be considered for the dual.

A similar construction holds if we are interested in the string parametrization of the $G^\vee$-canonical basis:
\begin{definition}[Kashiwara (or string) variety]
\label{def:geom_kashiwara_variety}
Define the geometric string variety $C_{>0}^{w_0}$ as:
$$ C_{>0}^{w_0} := U \bar{w_0} U \cap B \cap G_{\geq 0}$$ 
\end{definition}
Of course, this definition depends on $\bar{w}_0$, our choice of representative for the longest element $w_0$ in the Weyl group. The Kashiwara variety also has parametrizations indexed by reduced words:

\begin{thm}[\cite{bib:BZ01}]
\label{thm:kashiwara_variety_params}
Every reduced word ${\bf i} \in R(w_0)$ gives rise to a bijection:
$$
 \begin{array}{cccc}
x_{-\mathbf{i}}: & \R_{>0}^m & \rightarrow & C_{>0}^{w_0}\\
                 & (c_1, \dots, c_m) & \mapsto     & x_{-i_1}(c_1) \dots x_{-i_m}(c_m)
 \end{array}
$$
Moreover, for ${\bf i}, {\bf i'} \in R(w_0)$, the maps $R_{\bf -i,-i'} := x_{-\mathbf{i'}}^{-1} \circ x_{-\mathbf{i}}: \R_{>0}^m \rightarrow \R_{>0}^m$ are rational and substraction free.
\end{thm}

In the same fashion, changes of parametrizations in string coordinates for the $G^\vee$-canonical basis are the tropicalization of $R_{\bf -i,-i'}$ (\cite{bib:BZ01}, theorem 5.2), i.e:
$$ \left[ R_{\bf -i,-i'} \right]_{trop} = \left( \left[x_{\bf -i'}\right]^{-1} \right)^\vee \circ \left[x_{\bf -i}\right]^\vee: \Cc_{\bf i} \cap \N^m \rightarrow \Cc_{\bf i'} \cap \N^m$$

A useful relationship between the maps $x_{\bf i}$ and $x_{\bf -i}$ is the following:
\begin{lemma}[ \cite{bib:BZ01} Lemma 6.1 and Remark 6.2]
\label{lemma:change_of_coordinates_UC}
Let ${\bf i } = \left( i_1, \dots, i_j \right) \in R(w)$ a reduced expression and $\left( \beta^\vee_1, \dots, \beta^\vee_j \right)$ an associated positive coroots enumeration. Then the following statements are equivalent:
$$(i)   \left( x_{-i_1}(c_1) \dots x_{-i_j}(c_j) \right)^T = c_1^{-\alpha_{i_1}^\vee} \dots c_j^{-\alpha_{i_j}^\vee} x_{i_j}(t_j) \dots x_{i_1}(t_1) $$
$$(ii)  \forall 1 \leq k \leq j, t_k = c_k \prod_{l<k} c_l^{ \alpha_{i_k}(\alpha_{i_l}^\vee) }$$
$$(iii) \forall 1 \leq k \leq j, c_k = t_k \prod_{l<k} t_l^{ \beta_{k}(\beta_{l}^\vee) }$$
Moreover:
$$ \prod_{k=1}^j c_k^{\alpha_{i_k}^\vee} = \prod_{k=1}^j t_k^{-w^{-1} \beta_k^\vee}$$
\end{lemma}
\begin{proof}
The equivalence between the two first statements is immediate using commutation relations:
\begin{align*}
  & \left( x_{-i_1}(c_1) \dots x_{-i_j}(c_j) \right)^T\\
= &  c_j^{-\alpha_{i_j}^\vee} x_{i_j}(c_j) \dots c_1^{-\alpha_{i_1}^\vee} x_{i_1}(c_1)\\
= &  c_1^{-\alpha_{i_1}^\vee} \dots c_j^{-\alpha_{i_j}^\vee} \prod_{k=0}^{j-1} x_{i_{j-k}}(c_{j-k} \prod_{l<j-k} c_l^{ \alpha_{i_k}(\alpha_{i_l}^\vee) } )
\end{align*}
The equivalence between  the two last statements can be proved by induction over $j$. For $j=1$, it is immediate. Then, for $j \geq 1$, by induction hypothesis:
$$c_j = t_j \prod_{k=1}^{j-1} c_k^{-\alpha_{i_j}(\alpha_{i_k}^\vee)}
      = t_j \prod_{k=1}^{j-1} \left( t_k \prod_{l=1}^{k-1} t_l^{\beta_k(\beta_l^\vee)} \right)^{-\alpha_{i_j}(\alpha_{i_k}^\vee)} $$
Rearranging the double product gives:
\begin{align*}
c_j = & t_j \left( \prod_{k=1}^{j-1} t_k^{-\alpha_{i_j}(\alpha_{i_k}^\vee) } \right) \prod_{l=1}^{j-1} t_l^{-\sum_{k=l}^{j-1} \beta_k(\beta_l^\vee) \alpha_{i_j}(\alpha_{i_k}^\vee) }\\
= & t_j \prod_{k=1}^{j-1} t_k^{ -\alpha_{i_j}\left( \alpha_{i_k}^\vee + \sum_{l=k}^{j-1} \beta_l(\beta_k^\vee) \alpha_{i_l}^\vee \right) }
\end{align*}
Then using the second identity in lemma \ref{lbl:kumar} with $\lambda = \beta^\vee_k$:
\begin{align*}
  & \alpha_{i_k}^\vee + \sum_{l=k}^{j-1} \beta_l(\beta_k^\vee) \alpha_{i_l}^\vee\\
= & \alpha_{i_k}^\vee + \left( s_{i_1} \dots s_{i_k} \right)^{-1} \beta_k^\vee - \left( s_{i_1} \dots s_{i_{j-1}} \right)^{-1} \beta_k^\vee\\
= & - \left( s_{i_1} \dots s_{i_{j-1}} \right)^{-1} \beta_k^\vee
\end{align*}
In the end, as announced:
$$ c_j = t_j \prod_{k=1}^{j-1} t_k^{ \alpha_{i_j}\left( \left( s_{i_1} \dots s_{i_{j-1}} \right)^{-1} \beta_k^\vee \right) } 
       = t_j \prod_{k=1}^{j-1} t_k^{ \beta_j( \beta_k^\vee ) } $$
The last equality is a straightforward calculation:
$$
\prod_{k=1}^j c_k^{\alpha_{i_k}^\vee} 
= \prod_{k=1}^j \left( t_k^{\alpha_{i_k}^\vee} \prod_{l < k} t_l^{\beta_k(\beta_l) \alpha_{i_k}^\vee}\right)
= \prod_{k=1}^j t_k^{\alpha_{i_k}^\vee + \sum_{l=k+1}^j \beta_l(\beta_k) \alpha_{i_l}^\vee}
$$
Using again the lemma \ref{lbl:kumar}, we have:
$$\alpha_{i_k}^\vee + \sum_{l=k+1}^j \beta_l(\beta_k) \alpha_{i_l}^\vee = -w^{-1} \beta_k^\vee$$
\end{proof}

\paragraph{From the Lusztig variety to the Kashiwara variety:}
One can map $U_{>0}^{w_0}$ to $C_{>0}^{w_0}$ in many ways. But only the following one is the correct one. Define 
\begin{align}
\label{eqn:geom_lusztig_2_kashiwara}
\forall u \in U \cap B w_0 B, \eta^{e, w_0}\left( u \right) & := [\bar{w}_0^{-1} u^T]_{-0}^{-1} = [\bar{w}_0^{-1} u^T]_{+} \bar{w}_0 S\left( u \right)^\iota
\end{align}
\begin{align}
\label{eqn:geom_kashiwara_2_lusztig}
\forall v \in B \cap U \bar{w}_0 U, \eta^{w_0, e}\left( v \right)& := [\left( \bar{w}_0 v^T\right)^{-1}]_+ 
\end{align}

\begin{thm}[ \cite{bib:BZ01}, corollary 5.6 ]
\label{thm:geom_from_lusztig_to_kashiwara}
The map $\eta^{e, w_0}$ is a bijection from $U \cap B w_0 B$ to $B \cap U \bar{w_0} U$ and restricts to a bijection from $U_{>0}^{w_0}$ to $C_{>0}^{w_0}$. The inverse map is $\eta^{w_0, e}$.
\end{thm}

The tropicalization of this correspondence is a very interesting map:
\begin{thm}[ \cite{bib:BZ01}, theorem 5.7 ]
\label{thm:tropical_from_lusztig_to_kashiwara}
Changes of parametrization for the $G^\vee$ canonical basis are obtained by tropicalizing the following rational subtraction-free expressions. Going from ${\bf i}$-Lusztig parameters to ${\bf i'}$-string parameters is achieved by tropicalizing:
$$ x_{\bf-i'} \circ \eta^{e, w_0} \circ x_{\bf i}$$
Conversely, in order to obtain ${\bf i}$-string parameters from ${\bf i'}$-Lusztig parameters, tropicalize:
$$ x_{\bf i'} \circ \eta^{w_0, e} \circ x_{\bf-i}$$
\end{thm}

\subsection{Geometric crystal elements}

We now define geometric crystals with a notation similar to Kashiwara crystals.

\begin{definition}[Geometric crystals]
\label{def:geom_crystal}
Define the geometric crystal of highest weight $\lambda \in \afrak$ as the set:
$$\Bc\left(\lambda\right) := C_{>0}^{w_0} e^{\lambda}$$
The union of all highest weight crystals will be denoted by $\Bc$, which is nothing but the set of totally positive elements in $B$:
$$ \Bc := \bigsqcup_{ \lambda \in \afrak} \Bc(\lambda) = B_{\geq 0}$$
\end{definition}
Later, we will see that $\Bc$ is an abstract crystal in the sense of definition \ref{def:crystal}, thanks to the work of Berenstein and Kazhdan. Now, the goal of this subsection consists in exhibiting the correct way of embedding the Lusztig and Kashiwara varieties in $\Bc$ (figure \ref{fig:geom_parametrizations}).

The highest weight can easily be recovered from any element $x \in \Bc$ using the highest weight map:
\begin{definition}[Highest and lowest weight, \cite{bib:BK06} relation 1.6]
\label{def:highest_lowest_weight}
Define the highest weight map $\hw: \Bc \rightarrow \mathfrak{a}$ by:
$$ \forall x \in \Bc, \hw(x) := \log[ \bar{w}_0^{-1} x ]_0$$
The lowest weight is given by:
$$ \forall x \in \Bc, \lw(x) := \log[ \bar{w}_0^{-1} x^\iota ]_0^\iota = w_0 \hw(x)$$
\end{definition}

Highest weight crystals are disjoint in $\Bc$ and $\hw^{-1}\left( \{ \lambda \} \right) = \Bc\left( \lambda \right)$. Moreover:
\begin{properties}
\phantomsection
\label{properties:hw}
\begin{itemize}
 \item[(i)] $\hw$ can be extended to $B^+ w_0 B^+$ as:
$$\forall \left( z,u \right) \in U \times U, t \in H, \hw\left( z \bar{w}_0 t u \right) = \log(t) $$
 \item[(ii)] $\hw$ is an $U \times U$-invariant function.
 \item[(iii)] $$\forall \left(x, y \right) \in \mathfrak{a}^2, \forall g \in B^+ w_0 B^+, \hw\left( e^x g e^y \right) = w_0 x + \hw\left(g\right) + y$$ 
\end{itemize}
\end{properties}
\begin{proof}
\begin{itemize}
 \item[(i)] If $g = z \bar{w}_0 t u \in B^+ w_0 B^+$, then:
$$ \hw\left( g \right) = \log [\bar{w}^{-1}_0  g]_0
= \log [\bar{w}^{-1}_0  z \bar{w}_0 t u]_0\\
= \log(t)$$
 \item[(ii)] Immediate from (i)
 \item[(iii)] 
\begin{align*}
\hw\left( e^x g e^y \right) & = \log [\bar{w}^{-1}_0  e^x g e^y]_0\\
= & \log [\bar{w}^{-1}_0  e^x \bar{w}_0 \bar{w}^{-1}_0 g e^y]_0\\
= & \log [e^{w_0 x} \bar{w}^{-1}_0 g e^y]_0\\
= & w_0 x + \hw\left(g\right) + y
\end{align*}
\end{itemize}
\end{proof}

Every element $x \in \Bc\left(\lambda\right)$ can be written using a certain associated parameter. The letters $u$ and $z$ will usually refer to an element in $U_{>0}^{w_0}$ and the letter $v$ will usually refer to an element in $C_{>0}^{w_0}$. An expression we will often use is:
\begin{thm}
For $x \in \Bc\left(\lambda\right)$, one can write uniquely:
$$ x = z \bar{w_0} e^{\lambda} u, z \in U^{w_0}_{>0}, u \in U^{w_0}_{>0}$$
Mapping $x$ to $z$ (resp. $u$) is a bijection from $\Bc(\lambda)$ to $U^{w_0}_{>0}$. The former will be refered to as the Lusztig parameter associated to $x$, and the latter the twisted Lusztig parameter.
\end{thm}
\begin{proof}
The existence is a consequence of the definition of $\Bc(\lambda)$. The uniqueness comes, as we will see, from exhibiting inverse maps that preserve total positivity.
\end{proof}

There is also the possibility of using a parameter $v \in C_{>0}^{w_0}$ that we will call the string or Kashiwara parameter associated to $x$. Such names are justified by the fact that these choices give a geometric lifting of the parametrizations for crystal bases. In all the following formulas, the group elements considered belong to the double Bruhat cell $B \cap B^+ w_0 B^+$ and thus, every Gauss decomposition that we use is allowed.

\begin{definition}[Parameters associated to a crystal element]
\label{def:crystal_parameter}
Define the following maps on $\Bc$:
$$
\begin{array}{cccc}
\varrho^L : &  \Bc                           & \longrightarrow & U^{w_0}_{>0}\\
	    &  x = z \bar{w}_0 e^{\lambda} u &   \mapsto       & z = [\bar{w}_0^{-1} x^\iota]_+^\iota
\end{array}$$

$$
\begin{array}{cccc}
\varrho^K : &  \Bc                           & \longrightarrow & C^{w_0}_{>0}\\
	    &  x                             &   \mapsto       & v = [\bar{w}_0^{-1} [x]_-]_{0+}^T
\end{array}$$

$$
\begin{array}{cccc}
\varrho^T : &  \Bc                           & \longrightarrow & U^{w_0}_{>0}\\
	    &  x = z \bar{w}_0 e^{\lambda} u &   \mapsto       & u = [\bar{w}_0^{-1} x]_+
\end{array}$$

For $x \in \Bc$, the group elements $z = \varrho^L(x)$, $v = \varrho^K(x)$ and $z = \varrho^T(x)$ will be referred to as the Lusztig, Kashiwara and twisted Lusztig parameters associated to $x$.
\end{definition}

The following property shows that all highest weight crystals share the same parametrizations, hinting to the compatibility properties of the canonical basis. Recall that $\eta^{w_0, e}$ is given in equation \eqref{eqn:geom_kashiwara_2_lusztig}.

\begin{proposition}
\label{proposition:crystal_param_maps}
Once restricted to $\Bc(\lambda)$ the maps $\varrho^L$, $\varrho^K$ and $\varrho^T$ are invertible with inverses:
$$
\begin{array}{cccc}
b_\lambda^L: &  U^{w_0}_{>0} & \longrightarrow & \Bc(\lambda)\\
	     &  z            &   \mapsto       & x = [z \bar{w}_0]_{-0} e^\lambda = z \bar{w}_0 e^\lambda \left( e^{-\lambda} [z \bar{w}_0]_+^{-1} e^{\lambda}\right)
\end{array}$$

$$
\begin{array}{cccc}
b_\lambda^K: &  C^{w_0}_{>0} & \longrightarrow & \Bc(\lambda)\\
	     &  v            &   \mapsto       & x = [\eta^{w_0, e}\left( v \right) \bar{w}_0]_{-0} e^\lambda = [(\bar{w}_0 v^T)^{-1} ]_+ \bar{w}_0 v^T [v^T]_0^{-1}e^\lambda
\end{array}$$

$$
\begin{array}{cccc}
b_\lambda^T: &  U^{w_0}_{>0} & \longrightarrow & \Bc(\lambda)\\
	     &  u            &   \mapsto       & x = S \circ \iota \left( e^{-\lambda} [ \bar{w}_0^{-1} u^T ]_{+} e^{\lambda} \right) \bar{w}_0 e^\lambda u 
\end{array}$$
\end{proposition}
\begin{rmk}
It is easy to see that with such definitions, $\eta^{e, w_0}\left( z \right) = v$.
\end{rmk}

In the sequel, we will try to stick to the letters $z$, $v$ and $u$ when dealing with each choice of parameter. The figure \ref{fig:geom_parametrizations} shows the different charts for $\Bc(\lambda)$, together with the inverse maps $b^L_\lambda$, $b^K_\lambda$ and $b^T_\lambda$. 

\setcounter{figure}{ \value{equation} }
\addtocounter{equation}{1}
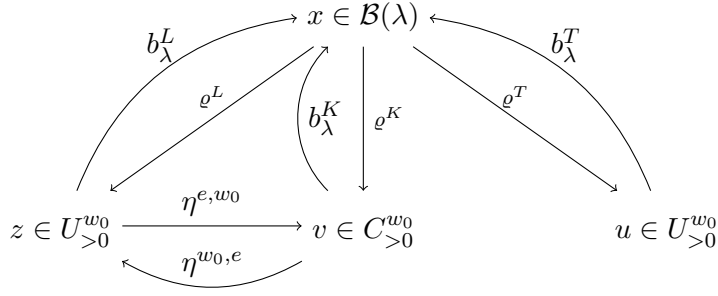
\begin{figure}[htp!]
\centering
\begin{tikzpicture}[baseline=(current bounding box.center)]
\matrix(m)[matrix of math nodes, row sep=5em, column sep=6em, text height=3ex, text depth=1ex, scale=2]
{
                    & x \in \Bc(\lambda) &                    \\
 z \in U^{w_0}_{>0} & v \in C^{w_0}_{>0} & u \in U^{w_0}_{>0} \\
};
\path[->, font=\scriptsize] (m-1-2) edge node[above]{$\varrho^L$} (m-2-1);
\path[->, font=\scriptsize] (m-1-2) edge node[right]{$\varrho^K$} (m-2-2);
\path[->, font=\scriptsize] (m-1-2) edge node[above]{$\varrho^T$} (m-2-3);

\draw [->] (m-2-1) to [bend left=30]  node[above, swap]{$b^L_\lambda$} (m-1-2);
\draw [->] (m-2-2) to [bend left=45]  node[auto, swap]{$b^K_\lambda$} (m-1-2);
\draw [->] (m-2-3) to [bend right=30] node[above, swap]{$b^T_\lambda$} (m-1-2);

\draw [->] (m-2-1) to [bend left=0 ]  node[above, swap]{$\eta^{e, w_0}$} (m-2-2);
\draw [->] (m-2-2) to [bend left=30]  node[auto , swap]{$\eta^{w_0, e}$} (m-2-1);
\end{tikzpicture}
\caption{Charts for the highest weight geometric crystal $\Bc(\lambda)$}
\label{fig:geom_parametrizations}
\end{figure}

\begin{proof}
The fact that these charts preserve total positivity is dealt with later in theorem \ref{thm:charts_positivity}. Let us start by writing:
$$ x = z \bar{w}_0 e^{\lambda} u $$

Let us deal with the Lusztig parameter $z$. It is easy to see that since $x \in B$, we have $x = [x]_{-0} = [z \bar{w}_0]_{-0} e^{\lambda}$. The second expression is obtained directly by the identity $[z \bar{w}_0]_{-0} = z \bar{w}_0 [z \bar{w}_0]_+^{-1}$.  In order to obtain $z$ from $x$, making use of the anti-automorphism $\iota$ gives $ x^\iota = u^\iota e^{-\lambda} \bar{w}_0 z^\iota$. Hence:
$$ \left[ \bar{w}_0^{-1} x^\iota \right]_+^\iota 
 = \left[ \bar{w}_0^{-1} u^\iota e^{-\lambda} \bar{w}_0 z^\iota \right]_+^\iota
 = \left( z^\iota \right)^\iota
 = z $$

The twisted Lusztig parameter $u$ is treated in a similar way. Write:
\begin{align*}
x = & [x^\iota]_{-0}^\iota\\
= & [ u^\iota e^{-\lambda} \bar{w}_0 ]_{-0}^\iota\\
= & \left([ u^\iota e^{-\lambda} \bar{w}_0 ]_{+}^\iota\right)^{-1} \bar{w}_0 e^\lambda u \\
= & S \circ \iota (y) \bar{w}_0 e^\lambda u
\end{align*}
where we have used the involutive automorphism $S \circ \iota$ and:
\begin{align*}
y & = S \circ \iota \left( \left([ u^\iota e^{-\lambda} \bar{w}_0 ]_{+}^\iota\right)^{-1} \right)\\
& = \bar{w}_0^{-1} [ \bar{w}_0^{-1} e^{-\lambda} \left(u^\iota\right)^T ]_{-}^\iota \bar{w}_0\\
& = [ \bar{w}_0^{-1} e^{-\lambda} u^T e^{\lambda} ]_{+}\\
& = e^{-\lambda} [ \bar{w}_0^{-1} u^T ]_{+} e^{\lambda}
\end{align*}
And in order to obtain $u$ from $x$, write $\left[\bar{w}_0^{-1} x\right]_+ = \left[\bar{w}_0^{-1} z \bar{w}_0 e^\lambda u \right]_+ = u$.

Finally, for the Kashiwara parameter, if $x \in \Bc(\lambda)$ and $\varrho^K(x) = v$, then:
\begin{align*}
\eta^{w_0, e}(v) = & [ \left( \bar{w}_0 v^T \right)^{-1} ]_+\\
= & [ [\bar{w}_0^{-1} [x]_- ]_{0+}^{-1} \bar{w}_0^{-1} ]_+\\
= & [ [x]_-^{-1} \bar{w}_0 [\bar{w}_0^{-1} [x]_- ]_{-} \bar{w}_0^{-1} ]_+\\
= & \bar{w}_0 [\bar{w}_0^{-1} [x]_- ]_{-} \bar{w}_0^{-1}
\end{align*}
Therefore:
\begin{align*}
x = & [ \bar{w}_0 [ \bar{w}_0^{-1} x ]_{-0} ]_{-0}\\
= & [ \bar{w}_0 [ \bar{w}_0^{-1} [x]_- ]_{-} ]_{-0} e^{\lambda}\\
= & [ \eta^{w_0, e}(v) \bar{w}_0 ]_{-0} e^\lambda
\end{align*}
Another possible expression is indeed:
\begin{align*}
x = & \eta^{w_0, e}(v) \bar{w}_0 [ \eta^{w_0, e}(v) \bar{w}_0 ]_{+}^{-1} e^\lambda\\
= & \eta^{w_0, e}(v) \bar{w}_0 [ \left( \bar{w}_0 v^T \right)^{-1} \bar{w}_0 ]_{+}^{-1} e^\lambda\\
= & \eta^{w_0, e}(v) \bar{w}_0 v^T [v^T]_0^{-1} e^\lambda
\end{align*}
\end{proof}

And as announced, the following theorem shows that the maps we considered preserve total positivity:
\begin{thm}
\label{thm:charts_positivity}
All maps $\varrho^L$, $\varrho^{K}$ and $\varrho^T$ (and their inverses) are rational and subtraction-free once written in coordinates.
\end{thm}
\begin{proof}
Notice that:
$$ \eta^{e, w_0} \circ \varrho^L = \varrho^K $$
$$ \iota \circ \varrho^L \circ \iota = \varrho^T$$
We already know that $\eta^{e, w_0}$ and its inverse are rational subtraction-free once written in the appropriate charts (theorems \ref{thm:geom_from_lusztig_to_kashiwara} and \ref{thm:tropical_from_lusztig_to_kashiwara}). The same goes for $\iota$ as $x_{\bf i^{op}}^{-1} \circ \iota \circ x_{\bf i} = id$ and $x_{\bf i'}^{-1} \circ \iota \circ x_{\bf i}$. Therefore, the theorem will be proved by dealing only with the mappings $\varrho^L$ and $b_\lambda^L$.

In order to further reduce the problem, introduce the twist map studied in \cite{bib:BZ97}:
$$
\begin{array}{cccc}
\eta_{w_0} : & \left( U \cap B w_0 B \right) & \longrightarrow & \left( U \cap B w_0 B \right)\\
	    &  z                             &   \mapsto       & [ \bar{w}_0^{-1} z^T]_+
\end{array}$$
It is easy to see that $\eta_{w_0}^{-1} = \iota \circ \eta_{w_0} \circ \iota$. Moreover, one can show that $x_{\bf i'}^{-1} \circ \eta_{w_0} \circ x_{\bf i}$ is rational and subtraction-free, for every reduced words ${\bf i'}$ and ${\bf i}$.

Technically, in proposition \ref{proposition:crystal_param_maps}, we only proved that the following correspondence for the Lusztig parametrization is bijective:
$$
\begin{array}{ccc}
  \left( B \cap U \bar{w_0} U \right) e^\lambda & \longrightarrow & U \cap B w_0 B\\
  x = [z \bar{w}_0]_{-0} e^{\lambda}            &   \mapsto       & z = [\bar{w}_0^{-1} x^\iota]_+^\iota
\end{array}$$
Therefore, after getting rid of the dependence in $\lambda$, we will consider:
$$
\begin{array}{cccc}
\varphi: & B \cap U \bar{w_0} U   & \longrightarrow & U \cap B w_0 B\\
         & x = [z \bar{w}_0]_{-0} &   \mapsto       & z = [\bar{w}_0^{-1} x^\iota]_+^\iota
\end{array}$$
and prove that $x_{\bf i}^{-1} \circ \varphi \circ x_{\bf -i}(c_1, \dots, c_m)$ is rational and subtraction-free in the variables $(c_1, \dots, c_m)$, hence preserving total positivity. Applying the monomial change of variable in lemma \ref{lemma:change_of_coordinates_UC}:
$$   \left( x_{\bf-i}(c_1, \dots, c_m) \right)^T = c_1^{-\alpha_{i_1}^\vee} \dots c_m^{-\alpha_{i_m}^\vee} x_{\bf i^{op}}(t_m, \dots, t_1) $$
we obtain $x_{\bf i}^{-1} \circ \eta_{w_0}^{-1} \circ x_{\bf i^{op}}(t_m, \dots, t_1)$, which we know is rational and subtraction-free, as well as its inverse. All intermediate rearrangements were also rational and subtraction-free, hence the result.
\end{proof}

\subsection{The weight map}
\label{subsection:geom_weight_map}

\begin{definition}
\label{def:geom_weight_map}
Define the weight map $\gamma: \Bc\left( \lambda \right) \rightarrow \afrak$ by:
$$ e^{\gamma(x)} = [x]_0$$ 
\end{definition}

This weight map is the geometric analogue of the classical weight map for crystal bases. The similarity is particularly obvious when comparing the following result with equations \eqref{eqn:weight_lusztig} and \eqref{eqn:weight_string}. It uses the dual root system as the geometric crystal on $G$ encodes the combinatorics of the  canonical basis for $G^\vee$.

\begin{thm}
\label{thm:geom_weight_map}
Let $x \in \Bc(\lambda)$, $z = \varrho^L(x)$, $v = \varrho^{K}(x)$ and $u = \varrho^T(x)$. Write for ${\bf i} \in R(w_0)$
$$ z = x_{ \bf i}\left( t_1, \dots, t_m \right)$$
$$ v = x_{-\bf i}\left( c_1, \dots, c_m \right)$$
$$ u = x_{ \bf i}\left( t_1', \dots, t_m' \right)$$
And $e^{-\tilde{t}_k} = t_k$, $e^{-\tilde{t}_k'} = t_k'$ and $e^{-\tilde{c}_k} = c_k$. Then, in terms of coordinates, the weight map is given by:
\begin{align*}
\gamma(x) & = \lambda - \sum_{k=1}^m \tilde{t}_k \beta^\vee_k\\
& = w_0 \left( \lambda - \sum_{k=1}^m \tilde{t}'_k \beta_{k}^\vee \right)\\
& = \lambda - \sum_{k=1}^m \tilde{c}_k \alpha^\vee_{i_k}
\end{align*}
\end{thm}
\begin{proof}
For $x = [z \bar{w}_0]_{-0} e^\lambda$ with $z = x_{\bf i}\left( t_1, \dots, t_m\right)$. We apply lemma \ref{lemma:change_of_coordinates_UC} to the opposite reduced word ${\bf i^{op}} = \left( i_m, \dots, i_1 \right)$. As such, the positive roots enumeration is reversed as well as the order of the parameters $t_1, \dots, t_m$. There is a $b = x_{-\bf i^{op}}\left(c_1, \dots, c_m \right)^T \in B^+ \cap N \bar{w}_0^{-1} N$ such that:
$$z = \left( \prod_{j=1}^m t_j^{-w_0 \beta_{{\bf i^{op}},m-j+1}^\vee} \right) b$$
The exponents can be simplified as:
\begin{align*}
  & -w_0 \beta_{ {\bf i^{op}} ,m-j+1 }^\vee\\
= & -w_0 s_{i_m} \dots s_{i_{j+1}} \alpha_{i_j}^\vee\\
= & s_{i_1} \dots s_{i_{j-1}} \alpha_{i_j}^\vee\\
= & \beta_j^\vee
\end{align*}
Hence:
$$u = \left( \prod_{j=1}^m t_j^{\beta_j^\vee} \right) b$$
Because $b \bar{w}_0 \in N U$, we have:
$$ [x]_0
= [u \bar{w}_0]_0 e^{\lambda }
= \left[\left( \prod_{j=1}^m t_j^{\beta_j^\vee} \right) b \bar{w}_0\right]_0 e^{\lambda}
= e^{\lambda} \left( \prod_{j=1}^m t_j^{\beta_j^\vee} \right) 
$$

We can deduce the weight map expression in terms of $\left( t'_1, t_2', \dots, t_m'\right)$ easily from the above proof. Notice that applying $\iota$ to $x$, changes $u$ to $u'^\iota$, $\lambda$ to $-w_0 \lambda$ and $\gamma(x)$ to $-\gamma(x)$. As such, using the expression found for the weight map in Lusztig coordinates, while considering the opposite word ${\bf i^{op}}$:
$$ -\gamma(x) = -w_0 \lambda - \sum_{k=1}^m \tilde{t}'_{m-k+1} \beta_{{\bf i^{op}},k}^\vee $$
Hence:
\begin{align*}
 \gamma(x) & = w_0 \lambda + \sum_{k=1}^m \tilde{t}'_{m-k+1} \beta_{{\bf i^op},k}^\vee\\
& = w_0 \lambda + \sum_{k=1}^m \tilde{t}'_{m-k+1} (-w_0 s_{i_1} \dots s_{i_{m-k-1}}) \alpha_{i_{m-k}}^\vee \\
& = w_0 \left( \lambda - \sum_{k=1}^m \tilde{t}'_k \beta_{k}^\vee \right)
\end{align*}

In the string parametrization $v = \varrho^S(x) = x_{-\bf i}\left( c_1, \dots, c_m\right)$. By definition \ref{def:crystal_parameter}:
$$ v = [ \bar{w}_0^{-1} [x]_- ]_{0+}^T$$
Hence:
$$ \prod_{j=1}^m c_j^{-\alpha_{i_j}^\vee} 
= [v]_0 
= \left[\bar{w}_0^{-1} [x]_- \right]_0
= [\bar{w}_0^{-1} x]_0 [x]_0^{-1}
= e^{\lambda} [x]_0^{-1}
$$
Rearranging the equation yields the result.
\end{proof}

\subsection{Examples}

We illustrate the previous coordinate systems and maps by a few examples for different semi-simple groups. We will take $x \in \Bc$ and write in coordinates:
$$ z = \varrho^L(x) \in U^{w_0}_{>0}$$
$$ v = \varrho^K(x) \in C^{w_0}_{>0}$$

\paragraph{$A_1$-type:}
$$G = SL_2 = \left\{ x  = \begin{pmatrix} a & c  \\ b & d \end{pmatrix} \ | ad - bc = 1 \right\}$$
$$\gfrak = \mathfrak{sl}_2 = \left\{ x \in M_2(\C) \ | tr(x) = 0 \right\}$$
$$H = \left\{ x  = \begin{pmatrix} a & 0  \\ 0 & a^{-1} \end{pmatrix}, a \in \C^* \right\}$$
$$\hfrak = \C \alpha^\vee $$
where $\alpha^\vee = \begin{pmatrix} 1 & 0  \\ 0 & -1 \end{pmatrix}$.

The disjoint union of all highest weight crystals is $\Bc$:
$$\Bc = \left\{ \begin{pmatrix} a     & 0  \\ b & a^{-1} \end{pmatrix} \ | \ a>0, b>0 \right\}$$
For $x = \begin{pmatrix} a     & 0  \\ b & a^{-1} \end{pmatrix} \in \Bc$, if:
$$ \lambda = \hw( x ) $$
$$ z = \begin{pmatrix} 1     & t  \\ 0 & 1      \end{pmatrix}$$
$$ v = \begin{pmatrix} c^{-1} & 0 \\ 1 & c     \end{pmatrix}$$
then, in terms of the matrix $x$, we have:
$$ \lambda = \log(b) \alpha^\vee$$
$$ t = c = \frac{a}{b}$$

\paragraph{$A_2$-type:}
$$G = SL_3(\C) $$
$$\gfrak = \mathfrak{sl}_2 = \left\{ x \in M_3(\C) \ | tr(x) = 0 \right\}$$
$$H = \left\{ x  = \begin{pmatrix} a & 0 & 0 \\ 0 & b & 0 \\ 0 & 0 & c \end{pmatrix}, abc = 1, (a,b,c) \in (\C^*)^3 \right\}$$
$$\hfrak = \C \alpha_1^\vee \oplus \C \alpha_2^\vee$$
where $\alpha_1^\vee = \begin{pmatrix} 1 & 0 & 0 \\ 0 & -1 & 0 \\ 0 & 0 &  0 \end{pmatrix}$ and $\alpha_2^\vee = \begin{pmatrix} 0 & 0 & 0 \\ 0 &  1 & 0 \\ 0 & 0 & -1  \end{pmatrix}$.

The disjoint union of all highest weight crystals is given by lower triangular totally positive matrices:
$$\Bc = \left\{ \begin{pmatrix} a  & 0 & 0 \\ b & c & 0 \\ d & e & f \end{pmatrix} \ | \  acf=1; a,b,c,d,e,f>0; be-dc>0  \right\}$$
For a crystal element $x = \begin{pmatrix} a  & 0 & 0 \\ b & c & 0 \\ d & e & f \end{pmatrix} \in \Bc$, if:
$$ \lambda = \hw(x)$$
$$ z = x_{\bf 121}(t_1, t_2, t_3)
     = \begin{pmatrix} 1  & t_1 + t_3 & t_1 t_2 \\ 0 & 1 & t_2 \\ 0 & 0 & 1 \end{pmatrix}$$
$$ v = x_{\bf-121}(c_1, c_2, c_3)
     = \begin{pmatrix} \frac{1}{c_1 c_3}  & 0 & 0 \\ c_3^{-1}+\frac{c_1}{c_2} & \frac{c_1 c_3}{c_2} & 0 \\ 1 & c_3 & c_2 \end{pmatrix}$$
then the correspondence $\eta^{e, w_0}\left( z \right) = v$ gives:
$$ \left( t_1, t_2, t_3 \right) = \left( c_1, c_3, c_2 c_3^{-1}\right)$$
Moreover, we have:
$$ e^\lambda = \begin{pmatrix} d & 0 & 0 \\ 0 & \frac{be-dc}{d} & 0 \\ 0 & 0 & \frac{1}{be-dc} \end{pmatrix} $$
\begin{align*}
x & = \begin{pmatrix} t_1 t_2  & 0 & 0 \\ t_2 & t_3 t_1^{-1} & 0 \\ 1 & \frac{t_1+t_3}{t_1 t_2} & \frac{1}{t_2 t_3} \end{pmatrix} e^\lambda\\ 
& = \begin{pmatrix} c_1 c_3 & 0 & 0 \\ c_3 & c_1^{-1} c_2 c_3^{-1} & 0 \\ 1 & \frac{c_1 c_3+c_2}{c_1 c_3^2} & \frac{1}{c_2} \end{pmatrix} e^\lambda
\end{align*}

\subsection{Geometric crystals in the sense of Berenstein and Kazhdan}

Now, we will explain why $\Bc$ is a positive geometric crystal in the sense of Berenstein and Kazhdan (\cite{bib:BK00}, \cite{bib:BK06}) using their framework. Their construction starts with the notion of unipotent bicrystal. In our case, the unipotent bicrystal is simply the cell $B \cap B^+ w_0 B^+$. Then it can be decorated with structural maps and endowed with a positive structure. This tantamounts to restricting the structural maps to $\Bc$, the totally positive part. The structural maps we inherit satisfy the axioms in definition \ref{def:crystal} and more.

Define the fundamental additive $N$-character $e^{\chi_\alpha^-}: N \rightarrow \C$ by:
$$ \forall (\alpha,\beta) \in \Delta^2, \forall t \in \R, \chi_\alpha^-(e^{t f_\beta})) := t \delta_{\alpha, \beta}$$
where $\delta_{.,.}$ is the Kronecker delta. It is naturally extended to $B$ by setting $\forall x \in B, \chi_\alpha^-(x) = \chi_\alpha^-([x]_-)$

\begin{thm}
\label{thm:geom_crystal_is_crystal}
The set $\Bc$ is an abstract geometric crystal once endowed with the structural maps:
\begin{itemize}
 \item  $\begin{array}{cccc}
         \gamma : &  \Bc  & \rightarrow & \mathfrak{a}     \\
		  &  g    & \mapsto     & \log\left([g]_0\right)
        \end{array}$
 \item For $x \in \Bc$:
       \begin{align*}
          \varepsilon_\alpha(x) & := \chi_\alpha^-(x)\\
          \varphi_\alpha(x) & := \chi_\alpha^- \left( x^\iota \right) = \alpha\left( \gamma(x) \right) + \varepsilon_\alpha(x)
       \end{align*}
 \item $e_\alpha^c \cdot x = \left[x_\alpha\left( \frac{e^c-1}{e^{\varepsilon_\alpha(x)}} \right) x \right]_{-0}
                      = x_\alpha\left( \frac{e^c-1}{e^{\varepsilon_\alpha(x)}} \right) x \ x_\alpha\left( \frac{e^{-c}-1}{e^{\varphi_\alpha(x)}} \right)$
\end{itemize}
\end{thm}
An important fact to keep in mind is that the previous group product uses a $U \times U$ action, and the right action is there to exactly balance the left action. As such, the resulting group element is still in $\Bc \subset B$.

So far, we made the choice of working directly with the totally positive elements. Only in this subsection, we will work outside of the totally positive varieties, in order to present Berenstein and Kazhdan's construction, from which theorem \ref{thm:geom_crystal_is_crystal} follows immediately.

\paragraph{ The unipotent bicrystal $({ \bf X }, p)$: }
A unipotent bicrystal is a couple $({ \bf X }, p)$ such that ${ \bf X }$ is a $U \times U$ variety, meaning a set with a right and left action of $U$, and $p: { \bf X } \rightarrow G$ a $U \times U$-equivariant application, meaning it is an application such that the action of $U \times U$ on ${ \bf X }$ and $G$ commute.

Here pick ${\bf X} := B^+ w_0 B^+$ with the natural left and right group action of $U$. And $p: B^+ w_0 B^+ \hookrightarrow G$ is the inclusion map.

\paragraph{ The unipotent crystal $X^-$:} Following \cite{bib:BK06} section 2, we define $X^-$ to be the unipotent crystal associated to $({ \bf X }, p)$ by:
$$ X^- := p^{-1}(B) = B \cap B^+ w_0 B^+$$
It is nothing but the largest double Bruhat cell inside of the Borel subgroup $B$. Here we are dealing with a unipotent bicrystal of type $w_0$ (\cite{bib:BK06}, claim 2.6).

\paragraph{ The positive structure $\Theta_{\bf X}$: } Now fix $\lambda \in \afrak$. For every ${\bf i}$ consider the charts:
$$
\begin{array}{cccc}
b_\lambda^L \circ x_{\bf i}: &  \R_{>0}^m & \rightarrow & \Bc(\lambda)\\
b_\lambda^K \circ x_{\bf i}: &  \R_{>0}^m & \rightarrow & \Bc(\lambda)\\
b_\lambda^T \circ x_{\bf i}: &  \R_{>0}^m & \rightarrow & \Bc(\lambda)\\
\end{array}$$
In the language of \cite{bib:BK06}, they are the restrictions to the positive octant $\R_{>0}^m$ of toric charts from $\C^m$ to $\left( B \cap U \bar{w}_0 U \right) e^{\lambda}$. Moreover, because of theorem \ref{thm:charts_positivity}, these toric charts are positively equivalent, defining the same positive structure $\Theta_{\bf X}$ on $\left( B \cap U \bar{w}_0 U \right) e^{\lambda}$. When looking only at the image of $\R_{>0}^m$ through those charts, one is dealing only with $\Bc(\lambda)$.

\paragraph{ The positive geometric crystals $\Fc( { \bf X }, p, \Theta_{\bf X} )$: }

By proposition 2.25 in \cite{bib:BK06}, the unipotent bicrystal $\left( { \bf X }, p\right)$ gives rise to a geometric crystal:
$$\Fc( { \bf X }, p) = \left( { \bf X }^-, \gamma, \varphi_\alpha, \varepsilon_\alpha, e_\alpha^. | \alpha \in \Delta \right) .$$
By lemma 3.30 in \cite{bib:BK06}, one gets a positive geometric crystal $\Fc( { \bf X }, p, \Theta_{\bf X} )$, meaning that these structural maps respect the positive structure. Therefore, we can  restrict them to $\Bc$, which proves theorem \ref{thm:geom_crystal_is_crystal}. Notice that the notation for $\varepsilon_\alpha$ and $\varphi_\alpha$ are reversed compared to \cite{bib:BK06}.

\paragraph{ Tensor product of geometric crystals:} Given two geometric crystals $X$ and $Y$, each one endowed with maps $\left( \gamma, \varphi_\alpha, \varepsilon_\alpha, e_\alpha^. | \alpha \in \Delta \right)$, Berenstein and Kazhdan define the tensor product $X \otimes Y$ as the set $X \times Y$ endowed with the following maps:
\begin{align*}
\gamma( x \otimes y ) & = \gamma(x) + \gamma(y)\\
\varepsilon_\alpha( x \otimes y ) & =  \varepsilon_\alpha(x) + \log\left( 1 + e^{ \varepsilon_\alpha(y) - \varphi_\alpha(x) } \right)\\
\varphi_\alpha( x \otimes y ) & = \varphi_\alpha(y) + \log\left( 1 + e^{\varphi_\alpha(x) - \varepsilon_\alpha(y) } \right)\\
e^c_\alpha \cdot \left(x \otimes y \right) & = e^{c_1}_\alpha \cdot x \otimes e^{c_2}_\alpha \cdot y\\
 & \textrm{ where } \\
 & c_1 = \log\left(e^c + e^{ \varepsilon_\alpha(y) - \varphi_\alpha(x) } \right) - \log\left( 1 + e^{ \varepsilon_\alpha(y) - \varphi_\alpha(x) }\right)\\
 & c_2 = -\log\left(e^{-c} + e^{ \varepsilon_\alpha(y) - \varphi_\alpha(x) } \right) + \log\left( 1 + e^{ \varepsilon_\alpha(y) - \varphi_\alpha(x) }\right)
\end{align*}
Claim 2.16 in \cite{bib:BK06} asserts that $X \otimes Y$ is a geometric crystal. Notice that this definition is the same as our $h$-tensor product of crystals when $h=1$, and in proposition \ref{proposition:tensor_product_is_crystal}, we checked that claim.

\subsection{Additional structure}
\label{subsection:additional_structure}

\paragraph{Invariant under crystal action:}
At this level, it is easy to see that the highest weight is invariant under the crystal actions $e^._\alpha$:

\begin{lemma}
\label{lemma:hw_is_invariant}
$$ \forall x \in \Bc, \forall \alpha \in \Delta, \forall c \in \R, \hw\left( e^c_\alpha \cdot x \right) = \hw(x)$$
\end{lemma}
\begin{proof}
Notice that $\Bc = \left( B \cap B^+ w_0 B^+ \right)_{\geq 0} = \left( B \cap U \bar{w}_0 U \right)_{\geq 0} \cdot A$. Also, the crystal actions $e^._\alpha, \alpha \in \Delta$ are given by an action of $U \times U$, leaving the $A$ factor invariant. This factor is nothing but $e^{\hw(.)}$, hence the result.
\end{proof}

\paragraph{Verma relations:}
Following (\cite{bib:BK00}), for an abstract crystal $L$ and any word ${\bf i} = \left(i_1, \dots, i_k \right) \in I^k$ (not necessarily reduced), define the map:
$$ \begin{array}{cccc}
e_{\bf i}^.:  & \afrak \times L    & \rightarrow & L\\
              & \left(t, x \right) & \mapsto     & e_{\bf i}^t = e_{\alpha_{i_1}}^{ \beta^{(1)}(t) } \cdot e_{\alpha_{i_2}}^{ \beta^{(2)}(t) } \dots e_{\alpha_{i_k}}^{ \beta^{(k)}(t) } \cdot x
   \end{array}
$$
where $\beta^{(j)} = s_{i_k} \dots s_{i_{j+1}}(\alpha_{i_j})$.

The relations appearing in the next lemma are called Verma relations. If they hold, one can define unambiguously $e_w = e_{\bf i}$ for ${\bf i} \in I^k$ if $w = s_{i_1} \dots s_{i_k}$.
\begin{lemma}[lemma 2.1 \cite{bib:BK00}]
The following proposition are equivalent:
\begin{itemize}
 \item[(i)] For any ${\bf i} \in I^k$ and ${\bf i'} \in I^{k'}$, if:
$$ w = s_{i_1} \dots s_{i_k} = s_{i_1'} \dots s_{i_{k'}'}$$
Then $e_{\bf i} = e_{\bf i'}$.
 \item[(ii)] The following relations hold for every $c_1, c_2 \in \R$:
$$ e^{c_1}_\alpha \cdot e^{c_2}_\beta = e^{c_2}_\beta \cdot e^{c_1}_\alpha$$
if $\alpha(\beta) = \beta(\alpha) = 0$;
$$ e^{c_1 }_\alpha \cdot e^{2c_1 + c_2 }_\beta \cdot e^{c_1 + c_2 }_\alpha \cdot e^{c_2 }_\beta 
 = e^{c_2 }_\beta \cdot e^{c_1 + c_2 }_\alpha \cdot e^{2c_1 + c_2 }_\beta \cdot e^{c_1 }_\alpha $$
if $\alpha(\beta^\vee) = -1$, $\beta(\alpha^\vee) = -2$;
$$ e^{c_1 }_\alpha \cdot e^{3c_1 + c_2 }_\beta \cdot e^{2c_1 + c_2 }_\alpha \cdot e^{3c_1 + 2c_2 }_\beta \cdot e^{c_1 + c_2 }_\alpha \cdot e^{c_2 }_\beta 
 = e^{c_2 }_\beta \cdot e^{c_1 + c_2 }_\alpha \cdot e^{3c_1 + 2c_2 }_\beta \cdot e^{2c_1 + c_2 }_\alpha \cdot e^{3c_1 + c_2 }_\beta \cdot e^{c_1 }_\alpha $$
if $\alpha(\beta^\vee) = -1$, $\beta(\alpha^\vee) = -3$.
\end{itemize}
\end{lemma}
\begin{proof}
$(i) \Rightarrow (ii)$: 
If ${\bf i}$ and ${\bf i'}$ are reduced expressions, then by Tits lemma (theorem \ref{lemma:tits_lemma}), one can obtain ${\bf i'}$ from ${\bf i}$ using braid moves. If $\alpha$ and $\beta$ are simple roots in $\Delta$ and satisfy a $d$-term braid relationship $s_\alpha s_\beta s_\alpha \dots  = s_\beta s_\alpha s_\beta \dots $, then we obtain:
$$ \forall t \in \afrak, 
e^{ \alpha(t) }_\alpha \cdot e^{ (s_\alpha \beta)(t) }_\beta \cdot e^{ (s_\alpha s_\beta \alpha)(t) }_\alpha \dots = e^{ \beta(t) }_\beta \cdot e^{ (s_\beta \alpha)(t) }_\alpha \cdot e^{ (s_\alpha s_\beta \alpha)(t) }_\alpha \dots $$
In particular, writing this equation for $t$ in the span of the coweights $\omega_\alpha^\vee$ and $\omega_\beta^\vee$, we find the Verma relations. This is the classical rank $2$ reduction. The list of relations corresponds to the root systems $A_1 \times A_1$, $A_2$, $BC_2$, $G_2$.

$(ii) \Rightarrow (i)$
Conversely, the Verma relations imply $e_{\bf i} = e_{\bf i'}$ for reduced words. If ${\bf i}$ and ${\bf i'}$ are not reduced, it is well known that one can reduce them by using braid moves and by deleting equal successive indices, as they correspond to a product of the form $s_\alpha^2 = id$. Therefore, we only need to notice that if ${\bf i}$ contains two equal successive indices and ${\bf i'}$ is the word obtained by deleting them, then $e_{\bf i} = e_{\bf i'}$. Indeed, if ${\bf i} \in I^k$, $k \in N$ and $i_j = i_{j+1}$ for a certain $j$ then for all $t \in \afrak$:
$$ \forall t \in \afrak, 
   e^{\left( s_{i_k} \dots s_{i_{j+1}} (\alpha_{i_{j  }}) \right)(t)}_{\alpha_{i_{j  }}}
   e^{\left( s_{i_k} \dots s_{i_{j+2}} (\alpha_{i_{j+1}}) \right)(t)}_{\alpha_{i_{j+1}}} = id $$

\end{proof}

\begin{proposition}
\label{proposition:group_verma}
For the geometric crystal $\Bc$, Verma relations hold.
\end{proposition}
\begin{proof}
See \cite{bib:BK00}. The proof is carried by direct computations in the group.
\end{proof}

\paragraph{W-action on the crystal:} As soon as the Verma relations hold on any abstract crystal $L$, one can define a $W$ action on $L$. If $w = s_{i_1} \dots s_{i_k}$, define:
$$ \forall x \in L, w \cdot x = e^{-\gamma(x)}_w \cdot x$$
\begin{proposition}[ \cite{bib:BK00} ]
The $W$ action on a crystal is well defined, and the weight map is equivariant with respect to this action.
\end{proposition}
\begin{proof}
Equivariance is easily checked on simple reflections, as for $\alpha \in \Delta$ and $x \in L$:
$$ \gamma\left( s_\alpha \cdot x \right)
= \gamma\left( e^{-\alpha(\gamma(x))}_\alpha \cdot x \right)
= \gamma\left( x \right) - \alpha\left(\gamma(x)\right) \alpha^\vee
= s_\alpha( \gamma(x) )
$$
Then it carries on to all elements in the Weyl group by writing them as products of simple reflections, once we know we have defined an action.

Now, in order to check we have an action, consider $w = u v \in W$. By induction on the length, one can suppose that equivariance for $v$ holds ($\gamma( v \cdot x ) = v \gamma(x)$). If ${\bf i} \in I^k$ (resp. ${\bf i'} \in I^l$) is a word giving $u$ (resp. $v$), then their concatenation gives $w$. Moreover, with $\beta^{(j)}, 1 \leq j \leq k+l$ the corresponding roots and $t \in \afrak$:
\begin{align*}
e^{t}_w & = e_{\alpha_{i_1 }}^{\beta^{(1  )}(t) } \dots e_{\alpha_{i_k }}^{\beta^{(k  )}(t) } \cdot
            e_{\alpha_{i_1'}}^{\beta^{(k+1)}(t) } \dots e_{\alpha_{i_l'}}^{\beta^{(k+l)}(t) }\\
& = e_{\alpha_{i_1 }}^{v\beta^{(1  )}(vt) } \dots e_{\alpha_{i_k }}^{v\beta^{(k  )}(vt) } \cdot
    e_{\alpha_{i_1'}}^{ \beta^{(k+1)}( t) } \dots e_{\alpha_{i_l'}}^{ \beta^{(k+l)}( t) }\\
& = e^{vt}_u \cdot e^{t}_v
\end{align*}
Finally, it is easy to check that:
\begin{align*}
  & u \cdot (v \cdot x)\\
= & e^{-\gamma(v\cdot x)}_u \cdot \left( e^{-\gamma(x)}_v \cdot x \right)\\
= & e^{-v \gamma(x)}_u \cdot e^{-\gamma(x)}_v \cdot x\\
= & e^{-\gamma(x)}_w \cdot x\\
= & w \cdot x
\end{align*}
\end{proof}

\section{Paths and groups}
\label{section:path_transforms}
Let $\Cc\left( \R^+, \afrak \right)$ be the set of continuous paths valued in the real Cartan subalgebra $\afrak$. In this section, we introduce the left-invariant flow $\left( B_t(X) \right)_{ t \in \R^+ }$ driven by a path $X \in \Cc\left( \R^+, \afrak \right)$. And we explain why it is totally positive. If $g \in G_{\geq 0}$, the group acts on the left and induces a transformation on paths:
$$ T_g: \Cc\left( \R^+, \afrak \right) \longrightarrow \Cc\left( \R^+, \afrak \right)$$
However, if $g \notin G_{\geq 0}$, the resulting paths can blow up in finite time. We provide a careful analysis which allows to extend the path transform $T_g$ to certain $g \notin G_{\geq 0}$. This takes us to very close to the root operators in the geometric path model.

In general, the transform $T_g$ will be defined thanks to the Gauss decomposition of $g B_.(X)$, when it exists. Let us now introduce the tool of generalized determinantal calculus in order to test if a group element has a Gauss decomposition and if it is totally positive.

\subsection{On generalized determinantal calculus}
We mean by determinantal calculus, the computations involving minors and relations among them.

In the classical case of $GL_n\left( \C \right)$, the minor $\Delta_{I, J}(x)$ of a matrix $x \in GL_n$ is obtained as the determinant of the submatrix with rows $I$ and columns $J$. $I$ and $J$ are subsets of $\{ 1, \dots, n\}$. It is well known that matrices having a Gauss decomposition are those having non-zero principal minors. Moreover, recall from theorem \ref{thm:total_positivity_gln}, that total positivity also means the positivity of minors.

For our purposes, we need a generalization of these two facts to all complex semi-simple groups. In a series of papers \cite{bib:BZ97, bib:FZ99, bib:BZ01}, Berenstein, Fomin and Zelevinsky constructed generalized minors and used them to study total positivity. For $x = n a u \in N H U$ dense subset of $G$, define the generalized principal minors indexed by the fundamental weights as $\Delta^{\omega_i}(x) := a^{\omega_i}$, which is a regular function on all of $G$. This can be easily seen from the following definition as matrix coefficients in fundamental representations:
\begin{lemma}[\cite{bib:BBO} section 3]
Let $v_{\omega_i}$ be a highest weight vector for the representation $V(\omega_i)$ and $\langle \cdot , \cdot \rangle$ an invariant scalar product. With $v_{\omega_i}$ normalised, we have:
$$ \forall x \in G, \Delta^{\omega_i}(x) = \langle x v_{\omega_i}, v_{\omega_i} \rangle $$
\end{lemma}
\begin{proof}
On the dense subset $N H U \subset G$, write the Gauss decomposition $x = n a u$. Since the $U$ action fixes $v_{\omega_i}$ in the representation $V(\omega_i)$:
$$ \langle x v_{\omega_i}, v_{\omega_i} \rangle = \langle a u v_{\omega_i}, n^T v_{\omega_i} \rangle = \langle a v_{\omega_i},  v_{\omega_i} \rangle$$
The result holds using the fact that the torus $H$ acts multiplicatively on highest weight vectors:
$$a v_{\omega_i} = a^{\omega_i} v_{\omega_i}$$
\end{proof}

The first useful result is:
\begin{proposition}[\cite{bib:FZ99} Corollary 2.5]
 An element $x \in G$ admits a Gauss decomposition if and only if
$$ \forall \alpha \in \Delta, \Delta^{\omega_\alpha}(x) \neq 0$$
\end{proposition}

Arbitrary minors are indexed by the fundamental weights and couples of Weyl group elements:
$$ \forall (u, v) \in W \times W, \Delta_{u \omega_i, v \omega_i}(x) := \Delta^{\omega_i}\left( \overline{u}^{-1} x \bar{v} \right)$$
They can also be given a representation-theoretic definition:
$$ \forall (u, v) \in W \times W, \Delta_{u \omega_i, v \omega_i}(x) := \langle x \bar{v} v_{\omega_i}, \bar{u} v_{\omega_i} \rangle$$

This allows to state the second useful result which is a total positivity criterion for the lower unipotent group $N$.
\begin{thm}[Total positivity criterion - Particular case of Theorem 1.5 in \cite{bib:BZ97} or Theorem 1.11 in \cite{bib:FZ99}]
\label{thm:total_positivity_criterion}
The group element $x \in N$ is totally positive:
$$ x \in N^{w_0}_{>0}$$
if and only if:
$$ \forall w \in W, \forall \alpha \in \Delta, \Delta_{w \omega_\alpha, \omega_\alpha}(x) > 0$$
\end{thm}

\subsection{Paths on the solvable group B}
Let $\left( B_t(X) \right)_{ t \in \R^+ }$ be the $B$-valued path, driven by $X$ and solution of the following ordinary differential equation:
\begin{align}
\label{lbl:process_B_ode}
\left\{ \begin{array}{ll}
dB_t(X) = B_t(X) \left( \sum_{\alpha \in \Delta} f_\alpha dt + dX_t\right) \\
B_0(X) = \exp( X_0 )
\end{array} \right.
\end{align}

The link to Kostant's Whittaker model is explained in appendix \ref{appendix:whittaker_model}. This equation can be understood as being formal if $X$ fails to be regular enough so that the differential equation has a meaning. The following expression is easy to check (\cite{bib:BBO} after transpose) and can be taken as a definition when discarding the smoothness assumption on $X$:
\begin{thm}
\begin{align}
\label{lbl:process_B_explicit}
B_t(X) & = \left( \sum_{k \geq 0} \sum_{ i_1, \dots, i_k } \int_{ t \geq t_k \geq \dots \geq t_1 \geq 0} e^{ -\alpha_{i_1}(X_{t_1}) \dots -\alpha_{i_k}(X_{t_k}) } dt_1 \dots dt_k f_{i_1} \cdot f_{i_2} \dots f_{i_k} \right) e^{X_t}
\end{align} 
\end{thm}
By convention, the term for $k=0$ is the identity element. For probabilistic purposes, $X$ can be taken as a semi-martingale. Then, equation \eqref{lbl:process_B_ode} has to be viewed as a stochastic differential equation (SDE) written in the Stratonovich convention.

One can also note that all the algebraic operations on group elements can be interpreted as matrix operations in any finite dimensional representation of the group $G$. When $X$ is differentiable, equation \eqref{lbl:process_B_ode} has to be understood the following way. In any finite dimensional group representation $V$, $B_.(X)$ is viewed as $GL(V)$-valued function of the time parameter:
$$ B_.(X): \R_+ \longrightarrow GL(V)$$
It is the solution of the system of ordinary differential equations written in matrix form as: 
\begin{align}
\left\{ \begin{array}{ll}
\frac{dB(X)}{dt}(t) = B_t(X) \left( \sum_{\alpha \in \Delta} f_\alpha + \frac{dX}{dt}(t)\right) \\
B_0(X) = \exp( X_0 )
\end{array} \right.
\end{align}

\begin{example}[$A_1$-type]
In the case of $SL_2$, in the canonical representation:
$$dB_t(X) = B_t(X) \begin{pmatrix} dX_t & 0\\ dt & -dX_t \end{pmatrix}  $$
Solving the differential equation leads to:
$$ B_t(X) = \begin{pmatrix} e^{X_t} & 0 \\ e^{X_t} \int_0^t e^{-2X_s}ds & e^{-X_t} \end{pmatrix} $$
\end{example}
\begin{example}[$A_2$-type]
For the canonical representation of $SL_3$, $\afrak = \{ x \in \R^3 | x_1 + x_2 + x_3 = 0 \}$:
$$dB_t(X) = B_t(X) \begin{pmatrix} dX^1_t & 0 & 0\\ dt & dX^2_t & 0 \\ 0 & dt & dX^3_t \end{pmatrix} $$
Solving the differential equation leads to:
$$ B_t(X) = \begin{pmatrix} e^{X^1_t}                                                               & 0                                        & 0
                         \\ e^{X^1_t} \int_0^t e^{-\alpha_1(X_s)}                                   & e^{X^2_t}                                & 0
                         \\ e^{X^1_t} \int_0^t e^{-\alpha_1(X_s) }ds \int_0^s e^{-\alpha_2(X_u) }du & e^{X^2_t} \int_0^t e^{-\alpha_2(X_s) }ds & e^{X^3_t} \end{pmatrix} $$
where $\left( \alpha_1 = (1,-1, 0), \alpha_2 = (0, 1, -1) \right) $ are the simple roots.
\end{example}

Now define $\left( A_t(X) \right)_{ t \in \R^+ }$ and $\left( N_t(X) \right)_{ t \in \R^+ }$ via the $NA$ decomposition of $B_.(X) = N_.(X) A_.(X)$:
\begin{align}
\label{lbl:process_A_explicit}
A_t(X) & = e^{X_t}
\end{align}
\begin{align}
\label{lbl:process_N_explicit}
N_t(X) & = \sum_{k \geq 0} \sum_{ i_1, \dots, i_k } \int_{ t \geq t_k \geq \dots \geq t_1 \geq 0 } e^{ -\alpha_{i_1}(X_{t_1}) \dots -\alpha_{i_k}(X_{t_k}) } f_{i_1} \cdot f_{i_2} \dots f_{i_k} dt_1 \dots dt_k
\end{align}

\begin{lemma}
$A_.(X)$ and $N_.(X)$ are solution of the following equations:
\begin{align}
\label{lbl:process_A_ode}
dA_t(X) = A_t(X) dX_t, & \quad A_0(X) = \exp( X_0 )
\end{align}
\begin{align}
\label{lbl:process_N_ode}
dN_t(X) = N_t(X) \left( \sum_{\alpha \in \Delta} e^{ -\alpha\left(X_t\right) }f_\alpha dt \right), & \quad N_0(X) = id
\end{align}
\end{lemma}
\begin{proof}
Obvious for the $A$-part. Then, since $B_t(X) = N_t(X) A_t(X)$, we have by differentiation (Stratonovich differentiation rule in the stochastic case):
\begin{align*}
   & \ dN_t(X) A_t(X) + N_t(X) dA_t(X) = B_t(X) \left(\sum_{\alpha \in \Delta} f_\alpha dt + dX_t\right) \\
\Leftrightarrow & \ dN_t(X) A_t(X) + N_t(X) A_t(X) dX_t = N_t(X) A_t(X) \left(\sum_{\alpha \in \Delta} f_\alpha dt + dX_t\right) \\
\Leftrightarrow & \ dN_t(X) A_t(X) = N_t(X) A_t(X) \left(\sum_{\alpha \in \Delta} f_\alpha dt\right) \\
\Leftrightarrow & \ dN_t(X) = N_t(X) A_t(X) \left(\sum_{\alpha \in \Delta} f_\alpha dt\right) A_t(X)^{-1}\\
\Leftrightarrow & \ dN_t(X) = N_t(X) \left(\sum_{\alpha \in \Delta} e^{-\alpha(X_t)} f_\alpha dt\right)
\end{align*}
The last step uses the $\Ad$ action of the torus on the Chevalley generators.
\end{proof}

\subsection{Total positivity of the flow}
\label{subsection:flow_total_positivity}
Morally speaking, $B_.(X)$ is obtained by infinitesimal increments that are totally non-negative. Therefore, as totally non-negative matrices form a semigroup, the following theorem is no surprise:
\begin{thm}[\cite{bib:BBO}, lemma 3.4 - Total positivity of the flow $B_.$]
\label{thm:flow_B_total_positivity}
Let $X \in \Cc\left( \R^+, \afrak \right)$. Then for all $t \geq 0$, $B_t(X)$ is totally non-negative. More precisely:
$$ \forall t>0, B_t(X) \in N^{w_0}_{>0} A$$
\end{thm}
\begin{proof}
For $t=0$, $B_0(X) = e^{X_0} \in A$ which is totally non-negative.

For $t>0$, clearly we need to prove that $N_t(X) \in N^{w_0}_{>0}$ or equivalently, thanks to theorem \ref{thm:total_positivity_criterion} that all minors $\Delta_{ w \omega_i, \omega_i }\left( N_t(X) \right), 1 \leq i \leq n, w \in W$ are positive:
\begin{align*}
  & \Delta_{ w \omega_i, \omega_i }\left( N_t(X) \right)\\
= & \langle N_t(X) v_{\omega_i}, \bar{w} v_{\omega_i} \rangle\\
= & \sum_{k \geq 0} \sum_{ i_1, \dots, i_k } \int_{ t \geq t_k \geq \dots \geq t_1 \geq 0} e^{ -\alpha_{i_1}(X_{t_1}) \dots -\alpha_{i_k}(X_{t_k}) } dt_1 \dots dt_k \\
  & \ \ \langle f_{i_1} \cdot f_{i_2} \dots f_{i_k} v_{\omega_i}, \bar{w} v_{\omega_i} \rangle
\end{align*}
Because of lemma 7.4 in \cite{bib:BZ01}, we have that:
$$\forall k \in \N, \forall w \in W, \langle f_{i_1} \cdot f_{i_2} \dots f_{i_k} v_{\omega_i}, \bar{w} v_{\omega_i} \rangle \geq 0$$
and therefore, we have a sum of non-negative terms. In order to see that $\Delta_{ w \omega_i, \omega_i }\left( N_t(X) \right)$ is strictly positive, only one of them needs to be non-zero.

As $\bar{w} v_{\omega_i}$ generates the one dimensional weight space $V(\omega_i)_{w \omega_i}$, there is some sequence $i_1, \dots, i_k$ such that $\alpha_{i_1} + \dots + \alpha_{i_k} = \omega_i - w \omega_i$ and
$\bar{w} v_{\omega_i}$ is proportional to $f_{i_1} \cdot f_{i_2} \dots f_{i_k} v_{\omega_i}$. Hence a non-zero scalar product.
\end{proof}

\subsection{Path transforms}

The following path transform will play a fundamental role in the sequel.
\begin{definition}
 When it exists, for $g \in G$ and $X$ a continuous path in $\afrak$, define:
$$ T_g X(t) := \log \left[ g B_t(X)\right]_0$$
The previous expression makes sense when $g B_t(X)$ has a Gauss decomposition and $\left[ g B_t(X)\right]_0 \in A$, in order to be able to consider its logarithm.
\end{definition}

This path transform has the property:
\begin{thm}[\cite{bib:BBO2}, proposition 6.4]
Let $X$ be a continuous path in $\afrak$ and $g \in G$. Assume that $g B_t(X)$ has a Gauss decomposition on an open time interval $J$. Then $[ g B_t(X) ]_{-0}$ solves for $t \in J$:
$$d [ g B_t(X) ]_{-0} = [g B_t(X)]_{-0} \left( \sum_{\alpha} f_\alpha dt + d\left( T_g X \right)_t \right)$$
\end{thm}

There are certain sets $D \subset G$ such that for $g \in D$, $g B_t(X), t \geq 0$ has always a Gauss decomposition. The following will play an important role:
$$ D := N \cdot A \cdot U_{\geq 0} $$
\begin{proposition}
For $g \in D$, we have a well-defined path transform:
$$ T_g: \Cc\left( \R^+, \afrak \right) \rightarrow \Cc\left( \R^+, \afrak \right)$$
such that for $X \in \Cc\left( \R^+, \afrak \right)$, $T_g X$ is the unique path in $\afrak$ such that:
$$ [g B_t(X)]_{-0} = [g]_- B_t(T_g X) $$
\end{proposition}
\begin{proof}
We only need to prove that the path transform is well defined. As for all $t \geq 0$, $B_t(X) \in G_{\geq 0}$, we have that $( A U_{\geq 0}) B_t(X) \subset G_{\geq 0}$, since the totally non-negative matrices form a semigroup. Hence $D B_t(X) \subset N G_{\geq 0}$, which is a set whose elements admit a Gauss decomposition (see theorem \ref{thm:totally_positive_gauss_decomposition}).

In order to prove that the equation driving $[g B_t(X)]_{-0}$ is of the required form, use the previous theorem.
\end{proof}

\begin{properties}
\label{lbl:path_transform_properties}
  Let $X$ be a continuous path. Then:
\begin{itemize}
 \item[(i)]$\forall g_1, g_2 \in G$ such that $g_1 \in D, g_1 g_2 \in D$, we have:
$$ T_{g_1 g_2} = T_{g_1 [g_2]_{-}} \circ T_{g_2} $$ 
In particular, $\forall u_1, u_2 \in U_{\geq 0}, T_{u_1 u_2} = T_{u_1} \circ T_{u_2}$.
 \item[(ii)]  $\forall g \in D, \forall n \in N, T_{ng} = T_g$
 \item[(iii)] $\forall g \in D, \forall a \in A, T_{a g} X = T_g X + \log a$
 \item[(iv)]  $\forall g \in D, x \in \afrak, T_g( X + x) = T_{ g e^x } (X)$
 \item[(v)]   If $\alpha \in \Delta$ and $g=x_\alpha(\xi), \xi>0$, then:
$$ \left(T_g X \right)_t = X_t + \log\left( 1 + \xi \int_0^t e^{-\alpha( X_s)} ds \right) \alpha^{\vee}$$
\end{itemize}
\end{properties}
\begin{proof}
The proof uses the standard properties of the Gauss decomposition given in equations \ref{lbl:gauss1} and \ref{lbl:gauss2}.

\begin{itemize}
 \item[(i)] \begin{eqnarray*}
& & [ g_1 g_2 B_t(X) ]_{-0} = [ g_1 [g_2 B_t(X)]_{-0} ]_{-0}\\
\Rightarrow & & [ g_1 g_2]_- B_t(T_{g_1 g_2} X) = [  g_1 [g_2]_- B_t(T_{g_2} X)]_{-0} = [  g_1 [g_2]_- ]_- B_t( T_{ g_1 [g_2]_{-} } \circ T_{g_2} X)\\
\Rightarrow & & B_t(T_{g_1 g_2} X) = B_t( T_{ g_1 [g_2]_{-} } \circ T_{g_2} X)
\end{eqnarray*}

 \item[(ii)] $[n g B_t(X)]_{-0} = [n g]_{-} B(T_{gn}X)$ by definition. And on the other hand, is it also equal to $n [g B_t(X)]_{-0} = n [g]_- B_t(T_gX)$

 \item[(iii)] $[agB_t(X)]_{-0} = [ag]_{-} B(T_{a g}X) = a [g]_{-} a^{-1} B(T_{a g}X) $ by definition. And on the other hand, is it also equal to $a[gB_t(X)g]_{-0} = a [g]_- B_t(T_gX)$. 
Then, we have $a^{-1} B_t(T_{ag}(X) = B_t(T_g X)$.

 \item[(iv)] One can check that $B_t(X+x) = \exp(x) B_t(X)$

 \item[(v)] Direct computation, using the embedding from $SL_2$ into the closed subgroup of $G$ whose Lie algebra is generated by the $\mathfrak{sl}_2$-triplets $(e_\alpha, f_\alpha, h_\alpha)$. One can also use the lemma in the next subsection.
\end{itemize}
\end{proof}

\subsection{Extension of the path transform}
Now, looking at property $(v)$, it is natural to expect the path transform $T_{x_\alpha(\xi)}$ to be extended to negative values of $\xi$, although this will depend on the path taken as input. Let us examine first when a Gauss decomposition exists for $x_\alpha(\xi) B_t(X)$, or equivalently $x_\alpha(\xi) N_t(X)$.

\begin{lemma}
For different $\alpha$ and $\beta$ in $\Delta$:
$$ \forall t \geq 0, \Delta^{\omega_\beta}\left( x_\alpha(\xi) N_t(X) \right) = 1$$
$$ \forall t \geq 0, \Delta^{\omega_\alpha}\left( x_\alpha(\xi) N_t(X) \right) = 1 + \xi \int_0^t e^{-\alpha(X_s)} ds$$
\end{lemma}
\begin{proof}
The first identity is a consequence of proposition 2.2 in \cite{bib:FZ99}. For the second, we start by using equation \eqref{lbl:process_N_explicit} and work with the highest weight representation $V(\omega_\alpha)$. We have:
\begin{align*}
  & \Delta^{\omega_\alpha}\left( x_\alpha(\xi) N_t(X) \right) \\
= & \langle \exp(\xi e_\alpha) N_t(X) v_{\omega_\alpha}, v_{\omega_\alpha} \rangle\\
= & \sum_{k \geq 0} \sum_{i_1, i_2, \dots, i_k} \int_{t\geq t_1 \geq \dots \geq t_k \geq 0} e^{ -\alpha_{i_1}(X_{t_1}) - \dots \alpha_{i_k}(X_{t_k})} dt_1 \dots dt_k \\
  & \ \ \langle \exp( \xi e_\alpha) f_{i_1} \dots f_{i_k}  v_{\omega_\alpha}, v_{\omega_\alpha} \rangle
\end{align*}
Hence, as we will write the expansion $e^{\xi e_\alpha} = \sum_{n \in \N} \frac{\xi^n}{n!} e_\alpha^n$, we need to consider vectors of the form:
$$e^n_\alpha f_{i_1} \dots f_{i_k} v_{\omega_\alpha}$$
Now notice that:
$$ \forall n \in \N, \forall k \in \N, e^n_\alpha f_{i_1} \dots f_{i_k} v_{\omega_\alpha} \in V(\omega_\alpha)_\mu$$
where $V(\omega_\alpha)_\mu$ is the weight space in $V(\omega_\alpha)$ corresponding to the weight
$$ \mu = \omega_\alpha + n \alpha - \sum_{j=1}^k \alpha_{i_j}$$
Weight spaces in $V(\omega_\alpha)$ corresponding to different weights are orthogonal under the invariant scalar product $\langle., . \rangle$. Therefore, if $k \neq n$ or there is a $j$ such that $\alpha_{i_j} \neq \alpha$, we have:
$$ \langle e^n_\alpha f_{i_1} \dots f_{i_k} v_{\omega_\alpha}, v_{\omega_\alpha} \rangle = 0$$
Moreover, in the representation  $V(\omega_\alpha)$, we have:
$$ \forall n\geq 2, f_\alpha^n v_{\omega_\alpha} = 0$$
Therefore, most terms are zero:
\begin{align*}
  & \Delta^{\omega_\alpha}\left( x_\alpha(\xi) N_t(X) \right) \\
= & \langle v_{\omega_\alpha}, v_{\omega_\alpha} \rangle + \xi \langle e_\alpha f_\alpha v_{\omega_\alpha}, v_{\omega_\alpha} \rangle \int_0^t e^{-\alpha(X_s)}ds\\
= & 1 + \xi \int_0^t e^{-\alpha(X_s)}ds
\end{align*}
The last equality is due to the fact that:
\begin{align*}
  & e_\alpha f_\alpha v_{\omega_\alpha}\\
= & [e_\alpha, f_\alpha] v_{\omega_\alpha} + f_\alpha e_\alpha v_{\omega_\alpha}\\
= & h_\alpha v_{\omega_\alpha}\\
= & v_{\omega_\alpha}
\end{align*}
\end{proof}

This lemma encourages us to consider paths only up to a certain horizon $T$ and when applying $T_{x_\alpha(\xi)}$ to $X \in \Cc\left([0; T], \afrak\right)$, one can only take $\xi \in (-\frac{1}{\int_0^T e^{-\alpha(X)}}; +\infty)$. This can be summarized in the following proposition.
\begin{proposition}
 For every $\alpha \in \Delta$ and $c \in \R$, there is a path transform:
$$ e^c_\alpha: \Cc\left([0; T], \afrak\right) \longrightarrow \Cc\left([0; T], \afrak\right)$$
such that for $X \in \Cc\left([0; T], \afrak\right)$:
$$ e^c_\alpha \cdot X = T_{x_\alpha\left(\frac{e^c-1}{\int_0^T e^{-\alpha(X)}}\right) }\left( X \right) $$
\end{proposition}
This path transform is in fact the corner stone of the geometric path model we will now present. For instance, as we will see, we have:
$$ e^{c+c'}_\alpha = e^{c}_\alpha e^{c'}_\alpha$$
It is in fact the geometric lifting of the Littelmann operators.

\section{Geometric crystals: The path model}
\label{section:path_model}

In this section, we define a continuous family of Littelmann path models depending on a parameter $q$, as well as $q$-tensor product. For $h \rightarrow 0$, we recover the continuous 'frozen' setting presented at the beginning of \cite{bib:BBO2}. Since all $q$-Littelmann models for $q>0$ are equivalent in certain sense, our study will focus on the $h=1$ case and prove that tensor product of crystals is given by the concatenation of their elements.

The path transforms described in the previous section naturally appear as the building blocks for the Littelmann operators $e^c_\alpha$. We will also benefit from the construction of Berenstein and Kazhdan (\cite{bib:BK00, bib:BK04, bib:BK06}) while exhibiting Verma relations and finally we show that a simple projection exists between the path model and the group picture $\Bc$.

\subsection{Path models}

Let $\Cc\left( [0; T], \afrak \right)$ be set of $\afrak$-valued continuous functions on $[0, T]$. Its elements are loosely referred to as paths in $\afrak$. We call a 'model' a candidate for becoming a crystal. Hence a path model will be a set of paths endowed with structure maps. The subscript $0$ will indicate that they are starting at zero.

\begin{definition}
\label{def:path_model}
A path crystal $L$ is a subset $\Cc\left( [0; T], \afrak \right)$, where $\afrak$ is the real Cartan subalgebra, endowed with maps $\gamma$, $\varepsilon_\alpha$, $\varphi_\alpha$ and actions $\left(e^._\alpha\right)_{ \alpha \in \Delta }$ such that
\begin{itemize}
 \item The weight map $\gamma: L \rightarrow \afrak$ gives the endpoint:
       $$ \gamma(\pi) = \pi(T) $$
 \item $L$ is an abstract geometric crystal as in definition \ref{def:crystal}.
\end{itemize}
\end{definition}

\paragraph{Time reversal duality:}
Define the duality map $\iota: \Cc_0\left( [0; T], \afrak \right) \rightarrow \Cc_0\left( [0; T], \afrak \right)$ that associates to each path $\pi$ its dual $\pi^\iota$. It is defined as the time reversal:
$$\pi^\iota(t) = \pi(T-t) - \pi(T)$$
A crystal structure is said to behave well with respect to duality if:
$$e^{-c}_\alpha = \iota \circ e^c_\alpha \circ \iota$$
$$\varphi_\alpha = \varepsilon_\alpha \circ \iota $$

\begin{rmk}
We use the same symbol as the Kashiwara involution, because the two correspond as we explained in section \ref{section:involutions}.
\end{rmk}

\paragraph{Continuous h-Littelmann model:}
Here we define a family of path models indexed by $h>0$, a parameter that can be understood as temperature (\cite{bib:OConnell}). When $h=0$, we recover the continuous path model introduced and studied in \cite{bib:BBO2} as the continuous counterpart of Littelmann's path model.

When $h>0$, a continuous $h$-Littelmann model is a subset $L$ of $\Cc_0\left( [0; T], \afrak \right)$ endowed with the structure $L_h = \left( \gamma, \left( \varepsilon_\alpha, \varphi_\alpha, e^._\alpha \right)_{\alpha \in \Delta} \right)$: 
\begin{itemize}
 \item The weight map $\gamma: L \rightarrow \afrak$ is the endpoint:
       $$ \gamma(\pi) = \pi(T) $$
 \item $\varepsilon_\alpha, \varphi_\alpha: L \rightarrow \R$ defined for every $\alpha \in \Delta$ as
$$\varepsilon_\alpha( \pi ) := h \log\left( \int_0^T e^{ -h^{-1} \alpha(\pi(s)) }ds \right)$$
$$\varphi_\alpha( \pi ) := \alpha\left( \pi(T) \right) + h \log\left( \int_0^T e^{-h^{-1} \alpha(\pi(s))}ds \right)$$
 \item $e^c_\alpha: L \rightarrow L$, $c \in \R$, $\alpha \in \Delta$ defined as
$$ e^c_\alpha \cdot \pi(t) := \pi(t) + h \log\left( 1 + \frac{ e^{h^{-1} c} - 1 }{e^{h^{-1}\varepsilon_\alpha(\pi)}}\int_0^t e^{ -h^{-1} \alpha(\pi(s)) }ds \right) \alpha^\vee $$
\end{itemize}
When $h=0$, we take as defining axioms the limit $h \rightarrow 0$:
$$\varepsilon_\alpha\left( \pi \right) = -\inf_{0 \leq s \leq T} \alpha\left( \pi(s) \right) $$
$$\varphi_\alpha\left( \pi \right) = \alpha\left( \pi(T) \right)-\inf_{0 \leq s \leq T} \alpha\left( \pi(s) \right) $$
\begin{align}
\label{eq:root_operator_trop}
\forall \ 0<t<T, \ e^c_\alpha\left( \pi \right)(t) & = \ \pi(t) + \inf_{0 \leq s \leq T} \alpha\left( \pi(s) \right) \alpha^\vee \\
 & \quad - \min\left( \inf_{0 \leq s \leq t} \alpha\left( \pi(s) \right) - c,
                 \inf_{t \leq s \leq T} \alpha\left( \pi(s) \right) \right) \alpha^\vee 
\end{align}
$$ e^c_\alpha\left( \pi \right)(0) = \pi(0) = 0$$
$$ e^c_\alpha\left( \pi \right)(T) = \pi(T) + c \alpha^\vee$$
Indeed, the limits for $\varepsilon_\alpha$ and $\varphi_\alpha$ are an immediate application of the Laplace method. The limit for the root operator comes from re-arranging the geometric expression before using the Laplace method as well:
\begin{align*}
e^c_\alpha \cdot \pi(t) & = \pi(t) + h \log\left( 1 + \frac{ e^{h^{-1} c} - 1 }{e^{h^{-1}\varepsilon_\alpha(\pi)}}\int_0^t e^{ -h^{-1} \alpha(\pi(s)) }ds \right) \alpha^\vee\\
& = \pi(t) + h \log\left( e^{h^{-1} c} \int_0^t e^{ -h^{-1} \alpha(\pi(s)) }ds + \int_t^T e^{ -h^{-1} \alpha(\pi(s)) }ds \right) \alpha^\vee \\
& \ \ \ - h \log\left( \int_0^T e^{ -h^{-1} \alpha(\pi(s)) }ds \right) \alpha^\vee\\
& \stackrel{ h \rightarrow 0 }{ \rightarrow }
    \pi(t) + \inf_{0 \leq s \leq T} \alpha\left( \pi(s) \right) \alpha^\vee 
    - \alpha^\vee \min\left( \inf_{0 \leq s \leq t} \alpha\left( \pi(s) \right) - c,
                             \inf_{t \leq s \leq T} \alpha\left( \pi(s) \right) \right)
\end{align*}
Notice that the condition $0<t<T$ is essential in the Laplace method, as certain integral terms disappear at $t=0$ and $t=T$. We claim that this expression is exactly the same as the one defining the generalized Littelmann operators defined in \cite{bib:BBO2} section 3. We discuss that point in the next subsection '$h=0$ limit'.

A $h$-Littelmann model that satisfies the crystal axioms is called a $h$-Littelmann crystal. An important fact is that for $h>0$, all the continuous $h$-Littelmann structures $\left( L_h \right)_{h>0}$ on $\Cc_0\left( [0; T], \afrak \right)$ are equivalent. That is why we can restrict our attention to the case $h=1$. We will use the term ``Geometric Littelmann crystal'' to refer to the $h=1$ model as it is the path model for the geometric crystals introduced by Berenstein and Kazhdan.

The fact that $h$-Littelmann crystal structures on $\Cc_0\left( [0; T], \afrak \right)$ are equivalent for $h>0$ can easily be checked by using the rescaling on reals and on paths.
$$ \forall x \in \R, \psi_{h,h'}\left( x \right) = \frac{h'}{h} x$$
$$ \forall \pi \in \Cc_0\left( [0; T], \afrak \right), \psi_{h,h'}\left( \pi \right) = \frac{h'}{h} \pi$$
$\psi_{h,h'}$ intertwines structural maps. Consider two continuous Littelmann structures on $\Cc_0\left( [0; T], \afrak \right)$ associated to $h$ and $h'$:
$$ L_h    = \left( \gamma, \left( \varepsilon_\alpha, \varphi_\alpha, e^._\alpha \right)_{\alpha \in \Delta} \right) \quad  
   L_{h'} = \left( \gamma', \left( \varepsilon^{'}_{\alpha}, \varphi^{'}_{\alpha}, e^{'.}_{\alpha} \right)_{\alpha \in \Delta} \right)$$
We have:
\begin{align*}
\gamma^{'}\left( \psi_{h,h'}\left( \pi \right) \right) & = \psi_{h,h'}\left( \gamma\left( \pi \right) \right)\\
\varepsilon^{'}_{\alpha}\left( \psi_{h,h'}\left( \pi \right) \right) & = \psi_{h,h'}\left( \varepsilon_\alpha\left( \pi \right) \right)\\
\varphi^{'}_{\alpha}\left( \psi_{h,h'}\left( \pi \right) \right) & = \psi_{h,h'}\left( \varphi_\alpha\left( \pi \right) \right)\\
e^{'\psi_{h,h'}(c)}_{\alpha} \cdot \psi_{h,h'}\left( \pi \right) & = \psi_{h,h'}\left(  e^{c}_\alpha \cdot \pi \right)
\end{align*}

\begin{rmk}
 Our choice of describing this relationship as 'equivalence' and not 'isomorphism' in the strict sense is because the real parameter $c$ in the actions $e^c_\alpha, \alpha \in \Delta$ is rescaled. The structural maps $\varepsilon_\alpha, \varphi_\alpha$ are also rescaled. Clearly, this transformation is more easily seen as a change of underlying semifields.
\end{rmk}

Another fact worth mentioning is the commutation with respect to tensor product: If $B_h$ and $B_h^{'}$ are $h$-Littelmann crystals then $$ \psi_{h, h'}\left( B_h \otimes_h B_h^{'} \right) = \psi_{h, h'}\left( B_h \right) \otimes_{h'} \psi_{h, h'}\left( B_h^{'} \right) $$

\subsection{Classical Littelmann model as a limit}
\label{subsection:classical_littemann_limit}

In \cite{bib:BBO2} definition 3.3, the continuous path model described, which in fact coincides with Littelmann's original definition, has the following structural maps:
$$\varepsilon_\alpha\left( \pi \right) = -\inf_{0 \leq s \leq T} \alpha\left( \pi(s) \right) $$
$$\varphi_\alpha\left( \pi \right) = \alpha\left( \pi(T) \right)-\inf_{0 \leq s \leq T} \alpha\left( \pi(s) \right) $$
$$\textrm{If } 0 \leq c \leq \varepsilon_\alpha\left( \pi \right), \Ec^c_\alpha\left( \pi \right)(t) = \pi(t)
  - \min\left(0, -c - \inf_{0 \leq s \leq T} \alpha\left( \pi(s) \right) - \inf_{0 \leq s \leq t} \alpha\left( \pi(s) \right) \right) \alpha^\vee
 $$
$$\textrm{If } - \varphi_\alpha\left( \pi \right) \leq c \leq 0, \Ec^c_\alpha\left( \pi \right)(t) = \pi(t)
  - \min\left(-c, \inf_{t \leq s \leq T} \alpha\left( \pi(s) \right) - \inf_{0 \leq s \leq T} \alpha\left( \pi(s) \right) \right) \alpha^\vee
 $$
In this subsection, we give explicit indications on why these are exactly the same maps as our $h=0$ limit. The identification is immediate except when it comes to recognizing our tropical actions $e_\alpha^.$ (equation \eqref{eq:root_operator_trop}). It shows that our description has at least an advantage at $h=0$: Only one formula for $e_\alpha^c$ independently of the sign of $c$. The following computation proves they are equivalent. For paths starting from zero, if $c \geq 0$:
\begin{align*}
  & \min\left( \inf_{0 \leq s \leq t} \alpha(\pi(s)) - c, \inf_{t \leq s \leq T} \alpha(\pi(s)) \right)\\
= & \min\left( \inf_{0 \leq s \leq t} \alpha(\pi(s)) - c, \inf_{0 \leq s \leq t} \alpha(\pi(s)), \inf_{t \leq s \leq T} \alpha(\pi(s)) \right)\\
= & \min\left( \inf_{0 \leq s \leq t} \alpha(\pi(s)) - c, \inf_{0 \leq s \leq T} \alpha(\pi(s)) \right)
\end{align*}
While if $c \leq 0$:
\begin{align*}
  & \min\left( \inf_{0 \leq s \leq t} \alpha(\pi(s)) - c, \inf_{t \leq s \leq T} \alpha(\pi(s)) \right)\\
= & -c + \min\left( \inf_{0 \leq s \leq t} \alpha(\pi(s)), c + \inf_{t \leq s \leq T} \alpha(\pi(s)) \right)\\
= & -c + \min\left( \inf_{0 \leq s \leq t} \alpha(\pi(s)), \inf_{t \leq s \leq T} \alpha(\pi(s)), c + \inf_{t \leq s \leq T} \alpha(\pi(s)) \right)\\
= & -c + \min\left( \inf_{0 \leq s \leq T} \alpha(\pi(s)), c + \inf_{t \leq s \leq T} \alpha(\pi(s)) \right)\\
= & \min\left( -c + \inf_{0 \leq s \leq T} \alpha(\pi(s)), \inf_{t \leq s \leq T} \alpha(\pi(s)) \right)
\end{align*}
Replacing in each case, $\min\left( \inf_{0 \leq s \leq t} \alpha(\pi(s)) - c, \inf_{t \leq s \leq T} \alpha(\pi(s)) \right)$ in our expression for $e^c_\alpha$ recovers $\Ec^c, c \geq 0$ and $\Ec^c, c \leq 0$.

\begin{rmk}
The usual cutting conditions $-\varphi_\alpha(\pi) \leq c \leq \varepsilon_\alpha(\pi)$ appear naturally only when tropicalizing the canonical measure on geometric crystals. We refrain from saying more as it is the subject of \cite{bib:chh14c}. For now, we can notice that in order for $e^c_\alpha$ to preserve continuity at $t=0$, one needs $c \leq \varepsilon_\alpha(\pi)$. In order to preserve continuity at $t=T$, we need $-\varphi_\alpha(\pi) \leq c$.
\end{rmk}

\subsection{A rank 1 example}
In rank $1$, crystal actions on paths in $\afrak$ are in fact one dimensional, and via projection $\afrak$ can be considered as $\R$.

\paragraph{Connected crystal at $\bf h=1$:}
Let $\pi \in \Cc_0\left( [0; T], \afrak \right)$ be a path and $\langle \pi \rangle$ be the connected crystal generated by $\pi$:
$$\langle \pi \rangle = \left\{ \pi_c = e^c_\alpha \cdot \pi, c \in \R \right\} = \left\{ t \mapsto \pi(t) + \log\left( 1 + (e^c - 1)\frac{\int_0^t e^{-\alpha\left( \pi \right)}}{\int_0^T e^{-\alpha\left( \pi \right)}} \right) \alpha^\vee \right\}$$ 
Notice that there is an extremal element $\eta = e^{-\infty}_\alpha \cdot \pi$ that does not belong to the crystal, as it diverges at its endpoint ($t=T$):
$$\eta\left(t\right) = \pi\left(t\right) + \log\left(1 - \frac{\int_0^t e^{-\alpha\left( \pi \right)}}{\int_0^T e^{-\alpha\left( \pi \right)}} \right)\alpha^\vee $$

The transform $e^{-\infty}_\alpha$ is a projection as it gives $\eta$ when applied to any element of the crystal, as a consequence of $e^{-\infty}_\alpha \cdot e^{c}_\alpha = e^{-\infty}_\alpha$. As such, it is clearly not an injective map. However there is only one real number that is lost in this process, and it is in fact $\int_0^T e^{-\alpha\left( \pi \right)}$. This basic remark will be a key element in parametrizing path crystals.

\subsection{Geometric Littelmann model} 

As announced, we will now restrict our attention to the $h=1$ case, which we call the geometric case. In the next subsection we prove there is a projection morphism to $\Bc$ the typical crystal in the sense of Berenstein and Kazhdan.

\subsubsection{Geometric Littelmann Crystal}
 A geometric Littelmann crystal $L$ is a subset of $\Cc_0\left( [0; T], \mathfrak{a} \right)$ endowed with
\begin{itemize}
 \item A weight map $\gamma: L \rightarrow \mathfrak{a}$ defined as
$$ \gamma(\pi) = \pi(T) $$
 \item For every $\alpha \in \Delta$, maps $\varepsilon_\alpha, \varphi_\alpha$ defined as:
$$\varepsilon_\alpha( \pi ) := \log\left( \int_0^T e^{ -\alpha(\pi(s)) }ds \right)$$
$$\varphi_\alpha( \pi ) := \varepsilon_\alpha \circ \iota ( \pi )
                        = \alpha\left( \pi(T) \right) + \log\left( \int_0^T e^{-\alpha(\pi(s))}ds \right)$$
 \item The Littelmann operators $\left(e^._\alpha\right)_{ \alpha \in \Delta }$ are defined as:
$$ e^c_\alpha \cdot \pi(t) := \pi(t) + \log\left( 1 + \frac{ e^c - 1 }{e^{\varepsilon_\alpha(\pi)}}\int_0^t e^{ -\alpha(\pi(s)) }ds \right) \alpha^\vee $$
\end{itemize}

\begin{example}
 The whole set $\Cc_0\left( [0; T], \mathfrak{a} \right)$ is a geometric Littelmann crystal.
\end{example}

A little lemma shows how the Littelmann operators are linked to the transform $T_g$:
\begin{lemma}
 For $g = x_\alpha\left( \frac{e^c-1}{e^{\varepsilon_\alpha(\pi)}} \right)$, one has:
$$\forall \pi \in \Cc_0\left( [0, T], \afrak \right), e^c_\alpha \cdot \pi = T_g \pi$$
\end{lemma}
Such an expression for the group element $x_\alpha\left( \frac{e^c-1}{e^{\varepsilon_\alpha(\pi)}} \right)$ is no coincidence, as it is exactly the left action on the crystal $\Bc$. In fact, we derived it starting from the path model, which is in a sense independent from the work of Berenstein and Kazhdan. Then realized we were looking at geometric crystals.

Let $L \subset \Cc\left( [0; T], \afrak \right)$ be a geometric Littelmann crystal and $\pi \in L$. The following properties show that it is indeed a crystal in the usual sense:
\begin{properties}
\begin{enumerate}
 \item[(i)]   $\varphi_\alpha(\pi) = \varepsilon_\alpha(\pi) + \alpha\left( \gamma(\pi) \right)$
 \item[(ii)]  $\gamma\left( e^c_\alpha \cdot \pi \right) = \gamma\left( \pi \right) + c \alpha^\vee$
 \item[(iii)] $\varepsilon_\alpha\left( e^c_\alpha \cdot \pi \right) = \varepsilon_\alpha\left( \pi \right) - c$
 \item[(iv)]  $\varphi_\alpha\left( e^c_\alpha \cdot \pi \right) = \varphi_\alpha\left( \pi \right) + c$
 \item[(v)]   $e^._\alpha$ are indeed actions as $e^{c}_\alpha \cdot e^{c'}_\alpha = e^{c+c'}_\alpha $
 \item[(vi)]  The Littelmann action behaves well with respect to time-reversal:
              $$e^{-c}_\alpha = \iota \circ e^c_\alpha \circ \iota$$
              $$\varphi_\alpha = \varepsilon_\alpha \circ \iota $$
 \end{enumerate}
\end{properties}
\begin{proof}
\begin{enumerate}
 \item[(i)] Obvious.
 \item[(ii)]  \begin{align*}
		 & \gamma\left( e^c_\alpha \cdot \pi \right)\\
               = & \pi(T) + \log\left( 1 + \frac{ e^c - 1 }{e^{\varepsilon_\alpha(\pi)}}\int_0^T e^{ -\alpha(\pi(s)) }ds \right) \alpha^\vee\\
               = & \pi(T) + \log\left( 1 + ( e^c - 1 ) \right) \alpha^\vee\\
               = & \gamma\left( \pi \right) + c \alpha^\vee
              \end{align*}
 \item[(iii)] \begin{align*}
                 & \varepsilon_\alpha\left( e^c_\alpha \cdot \pi \right)\\
               = & \log\left( \int_0^T \frac{ e^{ -\alpha(\pi(s)) } }{\left( 1 + \frac{ e^c - 1 }{e^{\varepsilon_\alpha(\pi)}}\int_0^s e^{ -\alpha(\pi(u)) }du\right)^2 }ds \right)\\
               = & \log\left( \frac{e^{\varepsilon_\alpha(\pi)}}{ e^c - 1 } \left( \int_0^T -\frac{d}{ds}\left( \frac{ 1 }{ 1 + \frac{ e^c - 1 }{e^{\varepsilon_\alpha(\pi)}}\int_0^s e^{ -\alpha(\pi(u)) }du } \right)ds \right) \right)\\
               = & \log\left( \frac{e^{\varepsilon_\alpha(\pi)}}{ e^c - 1 } \left( 1 - e^{-c} \right) \right)\\
               = & \varepsilon_\alpha(\pi) - c\\
              \end{align*}
 \item[(iv)]  Obvious using (i), (ii) and (iii).
 \item[(v)]   \begin{align*}
                & \left( e^{c}_\alpha \cdot e^{c'}_\alpha \cdot \pi \right) (t)\\
              = & \left( e^{c'}_\alpha \cdot \pi \right)(t) + \log\left( 1 + \frac{e^c - 1}{e^{ \varepsilon\left( e^{c'}_\alpha \cdot \pi \right)}}\int_0^t e^{ -\alpha\left( e^{c'}_\alpha \cdot \pi (s) \right) } ds \right) \alpha^{\vee}\\
              = & \left( e^{c'}_\alpha \cdot \pi \right)(t) + \log\left( 1 + \frac{e^c - 1}{e^{ \varepsilon\left( \pi \right) - c'}}\int_0^t \frac{ e^{ -\alpha\left( \pi (s) \right) } }{\left( 1 + \frac{ e^{c'} - 1 }{e^{\varepsilon_\alpha(\pi)}}\int_0^s e^{ -\alpha(\pi(u)) }du \right)^2 } ds \right) \alpha^{\vee}\\
              = & \left( e^{c'}_\alpha \cdot \pi \right)(t) + \log\left( 1 + \frac{e^c - 1}{e^{ \varepsilon\left( \pi \right) - c'}} \frac{ e^{\varepsilon\left(\pi\right)} }{e^{c'} - 1}\int_0^t -\frac{d}{ds}\left( \frac{ 1 }{ 1 + \frac{ e^{c'} - 1 }{e^{\varepsilon_\alpha(\pi)}}\int_0^s e^{ -\alpha(\pi(u)) }du } \right) ds \right) \alpha^{\vee}\\
              = & \left( e^{c'}_\alpha \cdot \pi \right)(t) + \log\left( 1 + e^{c'}\frac{e^c - 1}{e^{c'} - 1}\left( 1 - \frac{ 1 }{ 1 + \frac{ e^{c'} - 1 }{e^{\varepsilon_\alpha(\pi)}}\int_0^t e^{ -\alpha(\pi(s)) }ds } \right) \right) \alpha^{\vee}\\
              = & \left( e^{c'}_\alpha \cdot \pi \right)(t) + \log\left( 1 + e^{c'} \frac{ \frac{ e^{c} - 1 }{e^{\varepsilon_\alpha(\pi)}}\int_0^t e^{ -\alpha(\pi(s)) }ds }{ 1 + \frac{ e^{c'} - 1 }{e^{\varepsilon_\alpha(\pi)}}\int_0^t e^{ -\alpha(\pi(s)) }ds } \right) \alpha^{\vee}\\
              = & \left( e^{c'}_\alpha \cdot \pi \right)(t) + \log\left( \frac{ 1 + \frac{ e^{c+c'} - 1 }{e^{\varepsilon_\alpha(\pi)}}\int_0^t e^{ -\alpha(\pi(s)) }ds }{ 1 + \frac{ e^{c'} - 1 }{e^{\varepsilon_\alpha(\pi)}}\int_0^t e^{ -\alpha(\pi(s)) }ds } \right) \alpha^{\vee}\\
              = & \left( e^{c+c'}_\alpha \cdot \pi \right)(t)\\
               \end{align*}
 \item[(vi)]   \begin{align*}
                & \left( \iota \circ e^c_\alpha \circ \iota \right)(\pi) (t)\\
              = & \iota\left( t \mapsto \pi^\iota(t) + \log\left( 1 + \frac{ e^c - 1 }{e^{\varepsilon_\alpha(\pi^\iota)}}\int_0^t e^{ -\alpha(\pi^\iota(s)) }ds \right) \alpha^\vee \right)(t)\\
              = & \pi(t) + \iota\left( t \mapsto \log\left( 1 + \frac{ e^c - 1 }{e^{\varphi_\alpha(\pi)}}\int_0^t e^{ -\alpha(\pi^\iota(s)) }ds \right) \alpha^\vee \right)(t)\\
              = & \pi(t) + \iota\left( t \mapsto \log\left( 1 + \frac{ e^c - 1 }{e^{\varphi_\alpha(\pi)}} e^{\alpha(\pi(T))} \int_{T-t}^T e^{ -\alpha(\pi(s)) }ds \right) \alpha^\vee \right)(t)\\
              = & \pi(t) + \iota\left( t \mapsto \log\left( 1 + \frac{ e^c - 1 }{e^{\varepsilon_\alpha(\pi)}} \int_{T-t}^T e^{ -\alpha(\pi(s)) }ds \right) \alpha^\vee \right)(t)\\
              = & \pi(t) + \log\left( 1 + \frac{ e^c - 1 }{e^{\varepsilon_\alpha(\pi)}} \int_t^T e^{ -\alpha(\pi(s)) }ds \right) \alpha^\vee
                         - \log\left( 1 + \frac{ e^c - 1 }{e^{\varepsilon_\alpha(\pi)}} \int_0^T e^{ -\alpha(\pi(s)) }ds \right) \alpha^\vee \\
              = & \pi(t) + \log\left( 1 + \frac{ e^c - 1 }{e^{\varepsilon_\alpha(\pi)}} \left( \int_0^T e^{ -\alpha(\pi(s)) }ds - \int_0^t e^{ -\alpha(\pi(s)) }ds \right) \right) \alpha^\vee - c \alpha^\vee \\
              = & \pi(t) + \log\left( e^c - \frac{ e^c - 1 }{e^{\varepsilon_\alpha(\pi)}} \int_0^t e^{ -\alpha(\pi(s)) }ds \right) \alpha^\vee - c \alpha^\vee \\
              = & \pi(t) + \log\left( 1 + \frac{ e^{-c} - 1 }{e^{\varepsilon_\alpha(\pi)}} \int_0^t e^{ -\alpha(\pi(s)) }ds \right) \alpha^\vee
              \end{align*}
              As for $\varphi_\alpha = \varepsilon_\alpha \circ \iota $, it is obvious.
\end{enumerate}
\end{proof}

\subsubsection{Tensor products of crystals and concatenation of paths}
In this subsection, we will see that the seemingly complicated definition for the tensor product of crystals is in fact easily coded within a path model using the concatenation of paths. Define the concatenation of two paths $\pi_1: [0,T] \rightarrow \afrak$ and $\pi_2: [0,S] \rightarrow \afrak$ as 
the path $\pi_1 \ast \pi_2: [0,T+S] \rightarrow \afrak$ given by:
$$ \pi_1 \ast \pi_2 \left(t\right)
   = \left\{\begin{array}{cc}
             \pi_1(t)                    & \textrm{ if } 0 \leq t \leq T \\
             \pi_1(T) + \pi_2\left( t - T\right) & \textrm{ otherwise }
            \end{array}
     \right.
$$

\begin{thm}
\label{thm:concatenation_is_isomorphism}
$$ \begin{array}{cccc}
   \theta: & \Cc_0\left( [0; T], \afrak \right) \otimes \Cc_0\left( [0; S], \afrak \right) & \rightarrow & \Cc_0\left( [0; T+S], \afrak \right)\\
       & \pi_1 \otimes \pi_2                                                 & \mapsto     & \pi_1 \ast \pi_2
   \end{array}
$$
is a crystal isomorphism. In fact, the following properties are true:
\begin{itemize}
	\item[(i)] $\gamma\left( \pi_1 \ast \pi_2 \right) = \gamma\left( \pi_1 \otimes \pi_2 \right)$
	\item[(ii)] $\varepsilon_\alpha\left( \pi_1 \ast \pi_2 \right) = \varepsilon_\alpha\left( \pi_1 \otimes \pi_2 \right)$ or equivalently $\varphi_\alpha\left( \pi_1 \ast \pi_2 \right) = \varphi_\alpha\left( \pi_1 \otimes \pi_2 \right)$
	\item[(iii)] $e^c_\alpha\left( \pi_1 \ast \pi_2 \right) = \theta\left( e^c_\alpha\left( \pi_1 \otimes \pi_2 \right) \right)$
\end{itemize}
\end{thm}
\begin{proof}
Given those properties, $\theta$ clearly transports the crystal structure. The fact that it is invertible with a morphism as inverse map is obvious. Let us show these relations:
\begin{itemize}
	\item[(i)] \begin{align*}
	             & \gamma\left( \pi_1 \ast \pi_2 \right)\\
	           = & \pi_1 \ast \pi_2\left( T + S\right)\\
	           = & \pi_1(T) + \pi_2(S)\\
	           = & \gamma\left( \pi_1 \otimes \pi_2 \right)
	           \end{align*}
	\item[(ii)] \begin{align*}
	             & \varepsilon_\alpha\left( \pi_1 \ast \pi_2 \right)\\
	           = & \log\left( \int_0^{T+S} e^{-\alpha\left( \pi_1 \ast \pi_2(s) \right) } ds \right)\\
	           = & \log\left( \int_0^{T} e^{-\alpha\left( \pi_1 (s) \right) } ds + e^{ -\alpha\left( \pi_1(T)\right) } \int_0^{S} e^{-\alpha\left( \pi_2(s) \right) } ds \right)\\
	           = & \varepsilon_\alpha\left( \pi_1 \right) + \log\left( 1 + e^{\varepsilon_\alpha\left(\pi_2\right) - \alpha\left( \gamma(\pi_1) \right) - \varepsilon_\alpha\left(\pi_2\right)} \right)\\
	           = & \varepsilon_\alpha\left( \pi_1 \right) + \log\left( 1 + e^{\varepsilon_\alpha\left(\pi_2\right) - \varphi_\alpha\left(\pi_2\right)} \right)\\
	           = & \varepsilon_\alpha\left( \pi_1 \otimes \pi_2 \right)	           
	            \end{align*}
	\item[(iii)] One first needs to remember that
	             $$  e^c_\alpha\left( \pi_1 \otimes \pi_2 \right) = e^{c_1}_\alpha \cdot \pi_1 \otimes e^{c_2}_\alpha \cdot \pi_2 $$
	             with
	             $$ e^{c_1} = \frac{ e^{c+\varphi_\alpha\left( \pi_1 \right) } + e^{ \varepsilon_\alpha\left(\pi_2\right) } }
	                               { e^{  \varphi_\alpha\left( \pi_1 \right) } + e^{ \varepsilon_\alpha\left(\pi_2\right) } }  $$
	             $$ e^{c_2} = \frac{ e^{  \varphi_\alpha\left( \pi_1 \right) } + e^{ \varepsilon_\alpha\left(\pi_2\right) } }
	                               { e^{  \varphi_\alpha\left( \pi_1 \right) } + e^{-c+ \varepsilon_\alpha\left(\pi_2\right) } }  $$
	            We will compute both parts of the concatenated path $\theta\left( e^c_\alpha\left( \pi_1 \otimes \pi_2 \right) \right)$ separately. The first half is, using $0 \leq t \leq T$:
	            $$ \theta\left( e^c_\alpha\left( \pi_1 \otimes \pi_2 \right) \right)(t) = e^{c_1}_\alpha \cdot \pi_1(t) = \pi_1(t) + \alpha^\vee \log\left( 1 + \frac{e^{c_1}-1}{e^{\varepsilon_\alpha\left(\pi_1\right)}} \int_0^t e^{-\alpha\left( \pi_1(s)\right)}ds \right)$$
	            But as:
	            \begin{align*}
	              & \frac{e^{c_1}-1}{e^{\varepsilon_\alpha\left(\pi_1\right)}}\\
	            = & \frac{\frac{ e^{c+\varphi_\alpha\left( \pi_1 \right) } + e^{ \varepsilon_\alpha\left( \pi_2 \right) } }{ e^{  \varphi_\alpha\left( \pi_1 \right) } + e^{ \varepsilon_\alpha\left( \pi_2 \right) }  }-1}{ e^{\varepsilon_\alpha\left(\pi_1\right)} }\\
	            = & \frac{ e^{\varphi_\alpha\left( \pi_1 \right) } \left( e^c - 1 \right) }
	                     { e^{\varepsilon_\alpha\left(\pi_1\right)}\left( e^{  \varphi_\alpha\left( \pi_1 \right) } + e^{ \varepsilon_\alpha\left( \pi_2 \right) } \right) } \\
	            = & \frac{ e^c - 1  }
	                     { e^{\varepsilon_\alpha\left(\pi_1\right)} + e^{ \varepsilon_\alpha\left( \pi_2 \right) - \alpha\left( \gamma\left(\pi_1\right) \right) } } \\
	            = & \frac{ e^c - 1  }{ e^{\varepsilon\left( \pi_1 \ast \pi_2 \right)} }
	            \end{align*}
	            We get:
	            $$ \forall 0 \leq t \leq T, \theta\left( e^c_\alpha\left( \pi_1 \otimes \pi_2 \right) \right)(t) =  e^c_\alpha \cdot \left( \pi_1 \ast \pi_2 \right)(t)$$
	            
	            Moving on to the second half:
              \begin{align*}
                & \theta\left( e^c_\alpha\left( \pi_1 \otimes \pi_2 \right) \right)(T + t)\\
              = & e^{c_1}_\alpha \cdot \pi_1(T) + e^{c_2}_\alpha \cdot \pi_2(t)\\
              = & \pi_1(T) + c_1 \alpha^\vee + \pi_2(t) + \log\left( 1 + \frac{e^{c_2}-1}{e^{\varepsilon_\alpha\left(\pi_2\right)}} \int_0^t e^{-\alpha\left( \pi_2(s)\right)}ds \right) \alpha^\vee \\
              = & \pi_1 \ast \pi_2 (T+t) + c_1 \alpha^\vee + \log\left( 1 + \frac{e^{c_2}-1}{e^{\varepsilon_\alpha\left(\pi_2\right) }} \int_0^t e^{-\alpha\left( \pi_2(s)\right)}ds \right)\alpha^\vee
              \end{align*}
              Then, as:
              \begin{align*}
	              & \frac{e^{c_2}-1}{e^{\varepsilon_\alpha\left(\pi_2\right)}}\\
	            = & \frac{\frac{ e^{\varphi_\alpha\left( \pi_1 \right) } + e^{ \varepsilon_\alpha\left( \pi_2 \right) } }{ e^{  \varphi_\alpha\left( \pi_1 \right) } + e^{ -c + \varepsilon_\alpha\left( \pi_2 \right) }  }-1}{ e^{\varepsilon_\alpha\left(\pi_2\right)} }\\
	            = & \frac{ 1 - e^{-c}  }
	                     { e^{\varphi_\alpha\left(\pi_1\right)} + e^{ -c +\varepsilon_\alpha\left( \pi_2 \right)}} \\
	            \end{align*}
	            We get:
	            \begin{align*}
                & \theta\left( e^c_\alpha\left( \pi_1 \otimes \pi_2 \right) \right)(T + t)\\
              = & \pi_1 \ast \pi_2 (T+t) + c_1 \alpha^\vee + \log\left( 1 + \frac{ e^{c}-1 }{ e^{ c + \varphi_\alpha\left(\pi_1\right) } + e^{\varepsilon_\alpha\left(\pi_2\right) }} \int_0^t e^{-\alpha\left( \pi_2(s)\right)}ds \right) \alpha^\vee \\
              = & \pi_1 \ast \pi_2 (T+t) + c_1 \alpha^\vee + \\
                & \quad \log\left( 1 + \frac{ e^{c}-1 }{ e^{ c + \varphi_\alpha\left(\pi_1\right) } + e^{\varepsilon_\alpha\left(\pi_2\right) }} e^{\alpha\left( \gamma(\pi_1) \right) } \left( \int_0^{T+t} e^{-\alpha\left( \pi_1 \ast \pi_2(s)\right)}ds - e^{\varepsilon_\alpha\left( \pi_1 \right) }\right) \right) \alpha^\vee \\
              \end{align*}
              But 
              \begin{align*}
                & 1- \frac{ e^{c}-1 }{ e^{ c + \varphi_\alpha\left(\pi_1\right) } +  e^{ \varepsilon_\alpha\left( \pi_2 \right) }} e^{\alpha\left( \gamma(\pi_1) \right) + \varepsilon_\alpha\left( \pi_1 \right) }\\
              = & 1- \frac{ e^{c}-1 }{ e^{ c + \varphi_\alpha\left(\pi_1\right) } +  e^{ \varepsilon_\alpha\left( \pi_2 \right) }} e^{ \varphi_\alpha\left( \pi_1 \right) }\\
              = & \frac{ e^{ \varphi_\alpha\left( \pi_1 \right) } + e^{ \varepsilon_\alpha\left( \pi_2 \right) } }{ e^{ c + \varphi_\alpha\left(\pi_1\right) } +  e^{ \varepsilon_\alpha\left( \pi_2 \right) }}\\
              = & e^{-c_1}
              \end{align*}
              So that:
              \begin{align*}
                & \theta\left( e^c_\alpha\left( \pi_1 \otimes \pi_2 \right) \right)(T + t)\\
              = & \pi_1 \ast \pi_2 (T+t) + c_1 \alpha^\vee + \log\left( e^{-c_1} + \frac{ e^{c}-1 }{ e^{ c + \varphi_\alpha\left(\pi_1\right) } + e^{\varepsilon_\alpha\left(\pi_2\right) }} e^{\alpha\left( \gamma(\pi_1) \right) } \int_0^{T+t} e^{-\alpha\left( \pi_1 \ast \pi_2(s)\right)}ds \right) \alpha^\vee \\
              = & \pi_1 \ast \pi_2 (T+t) + \log\left( 1 + e^{c_1} \frac{ e^{c}-1 }{ e^{ c + \varphi_\alpha\left(\pi_1\right) } + e^{\varepsilon_\alpha\left(\pi_2\right) }} e^{\alpha\left( \gamma(\pi_1) \right) } \int_0^{T+t} e^{-\alpha\left( \pi_1 \ast \pi_2(s)\right)}ds \right) \alpha^\vee \\
              = & \pi_1 \ast \pi_2 (T+t) + \log\left( 1 + \frac{ \frac{ e^{c}-1 }{ e^{ c + \varphi_\alpha\left(\pi_1\right) } + e^{\varepsilon_\alpha\left(\pi_2\right) }} e^{\alpha\left( \gamma(\pi_1) \right) } }{ 1 - \frac{ e^{c}-1 }{ e^{ c + \varphi_\alpha\left(\pi_1\right) } + e^{\varepsilon_\alpha\left(\pi_2\right) }} e^{\alpha\left( \gamma(\pi_1) \right) + \varepsilon_\alpha\left( \pi_1 \right) }} \int_0^{T+t} e^{-\alpha\left( \pi_1 \ast \pi_2(s)\right)}ds \right) \alpha^\vee \\
              = & \pi_1 \ast \pi_2 (T+t) + \log\left( 1 + \frac{ \left( e^{c}-1 \right) e^{\alpha\left( \gamma(\pi_1) \right) } }{  e^{ c + \varphi_\alpha\left(\pi_1\right) } + e^{\varepsilon_\alpha\left(\pi_2\right) } - \left( e^{c}-1 \right) e^{\varphi_\alpha\left( \pi_1 \right) }} \int_0^{T+t} e^{-\alpha\left( \pi_1 \ast \pi_2(s)\right)}ds \right) \alpha^\vee \\
              = & \pi_1 \ast \pi_2 (T+t) + \log\left( 1 + \frac{ e^{c}-1 }{ e^{ \varphi_\alpha\left(\pi_1\right) } + e^{-\alpha\left( \gamma(\pi_1) \right)+\varepsilon_\alpha\left(\pi_2\right) } } \int_0^{T+t} e^{-\alpha\left( \pi_1 \ast \pi_2(s)\right)}ds \right) \alpha^\vee \\
              = & \pi_1 \ast \pi_2 (T+t) + \log\left( 1 + \frac{ e^{c}-1 }{e^{\varepsilon_\alpha\left(\pi_1 \ast \pi_2(s) \right)}} \int_0^{T+t} e^{-\alpha\left( \pi_1 \ast \pi_2(s)\right)}ds \right) \alpha^\vee \\
              = & e^c_\alpha\left( \pi_1 \ast \pi_2 \right)( T + t )
              \end{align*}
\end{itemize}
\end{proof}

\subsection{Projection onto the group picture}

We define the projection map as:
$$ \begin{array}{cccc}
    p: & L   & \rightarrow & \Bc = \left( B \cap B^+ w_0 B^+ \right)_{\geq 0}\\
       & \pi & \mapsto     & B_T\left(\pi\right)
   \end{array}
$$
We claim that $p$ is injective if $L$ is connected, and its image contains a $\Bc\left( \lambda \right)$ as soon as the intersection is non-empty. This will be a consequence of being an isomorphism of crystals if restricted to a connected component (section \ref{section:isom_results}). For now, let us just show that $p$ transports structures:

\begin{thm}
\label{thm:p_is_morphism}
$p$ is a morphism of abstract crystals, as the following properties hold:
\begin{itemize}
 \item[(i)]   $\gamma\left( \pi \right) = \gamma\left( p\left( \pi \right) \right)$
 \item[(ii)]  $p \circ \iota = \iota \circ p$ where $\iota$ stands for path duality on the left-hand side and for the Kashiwara involution on the right-hand side.
 \item[(iii)] $\varepsilon_\alpha = \varepsilon_\alpha \circ p$ or equivalently $\varphi_\alpha = \varphi_\alpha \circ p$
 \item[(iv)] $e^c_\alpha \cdot p \left( \pi \right) = p\left( e^c_\alpha \cdot \pi \right)$
\end{itemize}
\end{thm}
\begin{proof}
\begin{itemize}
 \item[(i)]   $$ \gamma\left( p\left( \pi \right) \right) = \gamma\left( B_T\left(\pi\right) \right) = \pi(T) = \gamma\left( \pi \right) $$
 \item[(ii)]  \begin{align*}
                & p \circ \iota \left( \pi \right)\\
              = & p\left( \pi^\iota \right)\\
              = & \left(\sum_{k \geq 1} \sum_{ i_1, \dots, i_k } \int_{ T \geq t_k \geq \dots \geq t_1 \geq 0} e^{ -\alpha_{i_1}(\pi^\iota(t_1)) \dots -\alpha_{i_k}(\pi^\iota(t_k)) } f_{i_1} f_{i_2} \dots f_{i_k} dt_1 \dots dt_k\right)e^{\pi^\iota(T)} + e^{ \pi^\iota(T) }\\
              \end{align*}
              Since:
              \begin{align*}
                & e^{ -\alpha_{i_1}(\pi^\iota(t_1)) \dots -\alpha_{i_k}(\pi^\iota(t_k)) } f_{i_1} f_{i_2} \dots f_{i_k}\\
              = & e^{ -\alpha_{i_1}(\pi(T-t_1)) \dots -\alpha_{i_k}(\pi(T-t_k)) + \alpha_{i_1}(\pi(T)) \dots + \alpha_{i_k}(\pi(T)) } f_{i_1} f_{i_2} \dots f_{i_k}\\
              = & e^{ -\alpha_{i_1}(\pi(T-t_1)) \dots -\alpha_{i_k}(\pi(T-t_k)) } e^{ -\pi(T) } f_{i_1} f_{i_2} \dots f_{i_k} e^{ \pi(T) }\\
              \end{align*}
              We obtain:
              \begin{align*}
                & p \circ \iota \left( \pi \right)\\
              = & \left( \sum_{k \geq 1} \sum_{ i_1, \dots, i_k } \int_{ T \geq t_k \geq \dots \geq t_1 \geq 0} e^{ -\alpha_{i_1}(\pi(T-t_1)) \dots -\alpha_{i_k}(\pi(T-t_k)) } e^{-\pi(T)} f_{i_1} f_{i_2} \dots f_{i_k} e^{ \pi(T) } dt_1 \dots dt_k\right) \\
                & \ \ \ e^{-\pi(T)} + e^{-\pi(T)} \\
              = &  e^{-\pi(T)}\left( id + \sum_{k \geq 1} \sum_{ i_1, \dots, i_k } \int_{ T \geq t_k \geq \dots \geq t_1 \geq 0} e^{ -\alpha_{i_1}(\pi(T-t_1)) \dots -\alpha_{i_k}(\pi(T-t_k)) }f_{i_1} f_{i_2} \dots f_{i_k} dt_1 \dots dt_k\right) \\
              = & \left( \left(\sum_{k \geq 1} \sum_{ i_1, \dots, i_k } \int_{ T \geq t_k \geq \dots \geq t_1 \geq 0} e^{ -\alpha_{i_1}(\pi(T-t_1)) \dots -\alpha_{i_k}(\pi(T-t_k)) } f_{i_k} \dots f_{i_2} f_{i_1} dt_1 \dots dt_k + id\right) e^{ \pi(T) } \right)^{ \iota}\\
              = & \left( \left(\sum_{k \geq 1} \sum_{ i_1, \dots, i_k } \int_{ T \geq t_1 \geq \dots \geq t_k \geq 0} e^{ -\alpha_{i_1}(\pi(t_1)) \dots -\alpha_{i_k}(\pi(t_k)) } f_{i_k} \dots f_{i_2} f_{i_1} dt_1 \dots dt_k + id\right) e^{ \pi(T) } \right)^{ \iota}\\
              = & \left( \left(\sum_{k \geq 1} \sum_{ i_1, \dots, i_k } \int_{ T \geq t_k \geq \dots \geq t_1 \geq 0} e^{ -\alpha_{i_1}(\pi(t_1)) \dots -\alpha_{i_k}(\pi(t_k)) } f_{i_1} f_{i_2} \dots f_{i_k} dt_1 \dots dt_k + id\right) e^{ \pi(T) } \right)^{ \iota}\\
              = & \left( B_T\left(\pi\right) \right)^\iota\\
              = & \iota \circ p \left( \pi \right)
              \end{align*}
 \item[(iii)] The two propositions are equivalent as in both the path model and Berenstein and Kazhdan's model, $\varepsilon_\alpha = \varphi_\alpha \circ \iota$.
              $$ \varepsilon_\alpha \circ p (\pi) = \chi_\alpha^-\left( B_T\left(\pi\right) \right) = \chi_\alpha^-\left( N_T\left(\pi\right) \right)$$
             All that remains to be proven is $e^{ \chi_\alpha^-\left( N_T\left(\pi\right) \right) }= \int_0^T e^{-\alpha\left(\pi(s)\right) } ds $. Both expressions coincide for $T=0$, and have the same derivatives with respect to $T$.
 \item[(iv)] \begin{align*}
               & e^c_\alpha \cdot p \left( \pi \right)\\
             = & e^c_\alpha \cdot B_T\left(\pi\right)\\
             = & x_\alpha\left( \frac{e^c-1}{e^{\varepsilon_\alpha\left( B_T\left(\pi\right) \right)}} \right) \cdot B_T\left(\pi\right) \cdot x_\alpha\left( \frac{e^{-c}-1}{e^{\varphi_\alpha\left( B_T\left(\pi\right) \right)}} \right)\\
             = & B_T\left( T_g \pi \right)\\
               & \quad \textrm{ where } g = x_\alpha\left( \frac{e^c-1}{e^{\varepsilon_\alpha\left( \pi \right)}} \right)\\
             = & B_T\left( e^c_\alpha \cdot \pi \right)\\
             = & p\left( e^c_\alpha \cdot \pi \right)
             \end{align*}
\end{itemize} 
\end{proof}

\subsection{Verma relations in the path model}
Thanks to the previous subsection, we know that the geometric crystals given by the path model, in a certain sense, sit above the group picture $\Bc$. The Verma relations are also valid at the path level: Given a geometric Littelmann crystal $L$, for any ${\bf i} \in I^k$, $k \in \N$ consider the map $e_{\bf i}^.$ as in subsection \ref{subsection:additional_structure}. The analogue of proposition \ref{proposition:group_verma} holds:
\begin{proposition}
 In the Littelmann geometric path model, $e_{\bf i}$ depends only on:
 $$ w = s_{i_1} \dots s_{i_k} \in W$$
\end{proposition}
\begin{proof}
For $\pi \in \Cc_0\left( [0, T], \afrak \right)$, $t \in \afrak$ and ${\bf i}, {\bf i'}$ words defining the same Weyl group element consider:
$$ \eta  = e^t_{\bf i } \cdot \pi$$
$$ \eta' = e^t_{\bf i'} \cdot \pi$$
Now let us prove that $\eta = \eta'$. Because the Littelmann path operators can be expressed thanks to the operator $T_.$, there are two elements $u, u' \in U$ such that:
$$ \eta = T_u \pi, \eta' = T_{u'} \pi$$
Furthermore, after applying the crystal morphism $p = B_T(.)$:
$$ B_T(\eta) = [u B_T(\pi)]_{-0}, B_T(\eta') = [u' B_T(\pi)]_{-0}$$
But since the Verma relations hold for the group picture (proposition \ref{proposition:group_verma}):
$$ B_T(\eta) = B_T(\eta')$$
Now, using the fact that $[u B_T(\pi)]_{-0} = u B_T(\pi) [u B_T(\pi)]_{+}^{-1}$ write:
\begin{align*}
g & = [\bar{w}_0^{-1} B_T(\eta)^\iota]_+\\
  & = [\bar{w}_0^{-1} \left( [u B_T(\pi)]_+^{-1} \right)^\iota B_T(\pi)^\iota u^\iota]_+\\
  & = [\bar{w}_0^{-1} B_T(\pi)^\iota ]_+ u^\iota
\end{align*}
Symetrically, $g=[\bar{w}_0^{-1} B_T(\pi)^\iota ]_+ u^\iota=[\bar{w}_0^{-1} B_T(\pi)^\iota ]_+ (u')^\iota$. Hence $u=u'$ and $\eta=\eta'$.
\end{proof}

Let us remark that one can also prove the Verma relations for the path model by direct computation on paths, though it is more complicated. Instead of group operations, one has to do multiple integrations by parts. We only give a partial sketch in the simply-laced case.

\paragraph{ $\mathbf{A_1}$ case:} In the case that $\alpha\left( \beta^\vee \right) = \beta\left( \alpha^\vee \right) = 0$, the actions $e_\alpha^.$ and $e_\beta^.$ commute, which proves the required Verma relation for type $A_1$:
$$ e^{c_1}_\alpha \cdot e^{c_2}_\beta = e^{c_2}_\beta \cdot e^{c_1}_\alpha$$

\paragraph{ $\mathbf{A_2}$ case:} By writing $t = c_1 \omega_1 + c_2 \omega_2$, the Verma relationship becomes
$$ e^{c_1 }_\alpha \cdot e^{c_1 + c_2 }_\beta \cdot e^{c_2 }_\alpha 
 = e^{c_2 }_\beta \cdot e^{c_1 + c_2 }_\alpha \cdot e^{c_1 }_\beta $$

A tedious computation gives the following lemma, that we give without proof:
\begin{lemma}
\label{label:lemma_verma_A2}
If $\alpha\left( \beta^\vee \right) = \beta\left( \alpha^\vee \right) = -1$, then
\begin{align*}
        & e^{c_1 }_\alpha \cdot e^{c_1 + c_2 }_\beta \cdot e^{c_2 }_\alpha \cdot \pi (t) = \pi(t)\\
        & + \alpha^\vee \log\left( 1 + \left( e^{c_1 + c_2} - 1 \right)
\frac
{\int_0^t e^{-\alpha(\pi(s))} \left( 1 + \left( e^{c_1} - 1 \right) \frac{\int_0^s e^{ -\beta(\pi(u)) }du }{\int_0^T e^{ -\beta(\pi(u)) }du } \right)}
{\int_0^T e^{-\alpha(\pi(s))} \left( 1 + \left( e^{c_1} - 1 \right) \frac{\int_0^s e^{ -\beta(\pi(u)) }du }{\int_0^T e^{ -\beta(\pi(u)) }du } \right)} \right) \\
        & + \beta^\vee  \log\left( 1 + \left( e^{c_1 + c_2} - 1 \right)
\frac
{\int_0^t e^{-\beta(\pi(s))} \left( 1 + \left( e^{c_2} - 1 \right) \frac{\int_0^s e^{ -\alpha(\pi(u)) }du }{\int_0^T e^{ -\alpha(\pi(u)) }du } \right)}
{\int_0^T e^{-\beta(\pi(s))} \left( 1 + \left( e^{c_2} - 1 \right) \frac{\int_0^s e^{ -\alpha(\pi(u)) }du }{\int_0^T e^{ -\alpha(\pi(u)) }du } \right)} \right)
\end{align*}
\end{lemma}
By inspecting the formula, one realizes that is symmetric in $\alpha$ and $c_1$ on the one hand, and $\beta$ and $c_2$ on the other hand. As such, by swapping those variables in the left hand-side term, one gets the Verma relation for type $A_2$:
$$ e^{c_1 }_\alpha \cdot e^{c_1 + c_2 }_\beta \cdot e^{c_2 }_\alpha 
 = e^{c_2 }_\beta \cdot e^{c_1 + c_2 }_\alpha \cdot e^{c_1 }_\beta $$

\paragraph{ $\mathbf{ADE}$ case:} Root systems from the ADE classifications have Dynkin diagrams with single edges as $\alpha\left( \beta^\vee \right) = -1$ in all cases. Hence, the only Verma relations needed are of type $A_1$ and $A_2$ and $e_w$ is unambiguously defined for any element $w$ in a Weyl group of $ADE$ type.

\section{Paths on the edge and the geometric Pitman transform}
\label{section:paths_on_the_edge}
A novelty in the path model approach to geometric crystals is the appearance of extended paths. Indeed, in order to parametrize geometric path crystals, we will need to consider paths whose endpoint or starting point is not defined anymore, as we move to the 'edges' of the crystal. This allows a simple compactification that does not involve the geometry of Bruhat cells. A visual sketch is given in figure \ref{figure:ExtremalPaths}.

The highest weight path of a connected crystal $\langle \pi_0 \rangle$ will be given by the geometric Pitman transform $\Tc_{w_0}$ applied to any $\pi \subset \langle \pi_0 \rangle$.

\setcounter{figure}{ \value{equation} }
\addtocounter{equation}{1}
\begin{figure}[ht!]
\centering
\caption{Sketch of extremal paths corresponding to a geometric Littelmann path $\pi$ with Lusztig parameter $g \in U^{w_0}_{>0}$ }
\label{figure:ExtremalPaths}

\begin{tikzpicture}[auto, scale=0.7]
\node at ( -0.3, 0 ) {$0$};
\node at ( 9.3, -0.3 ) {$T$};
\path [draw, very thick, color = black, ->] (0,0) -- (10 , 0 );
\path [draw, very thick, color = black, ->] ( 0, -7 ) -- ( 0, 7 );

\path [draw, dashed, color = black] (9,-7) -- (9 , 7 );

\definecolor{red  }{rgb}{ 0.5,0.0,0.0 }
\definecolor{blue }{rgb}{ 0.0,0.0,0.5 }
\definecolor{black}{rgb}{ 0.0,0.0,0.0 }

\draw (0,0) .. controls (3,2) and (6,2)  .. (9,3) node (pi) [right]{$\pi$};

\draw[dashed, color=blue] (0.3,-6) -- (1.05,-4);
\draw[color=blue] (1.05, -4) .. controls (1.8,-2) and (6, 5) .. (9, 6) node (high) [right]{$\eta^{high} = \Tc_{w_0} \pi$};

\draw[color=red] (0,0) .. controls (3,1) and (6.3,-2) .. (7.5, -4);
\draw[dashed, color=red] (7.5, -4) -- (8.7, -6) node (low) [right]{$\eta^{low} = e^{-\infty}_{w_0} \pi$};

\Edge[lw=0.1cm,style={post, bend right,color=black,}, label=$T_g$](low)(pi)
\Edge[lw=0.1cm,style={post, bend right,color=black,}, label=$\Tc_{w_0}$](pi)(high)
\Edge[lw=0.3cm,style={post, bend right,color=black,}, label=$e^{-\infty}_{w_0}$](pi)(low)
\end{tikzpicture}

\end{figure}
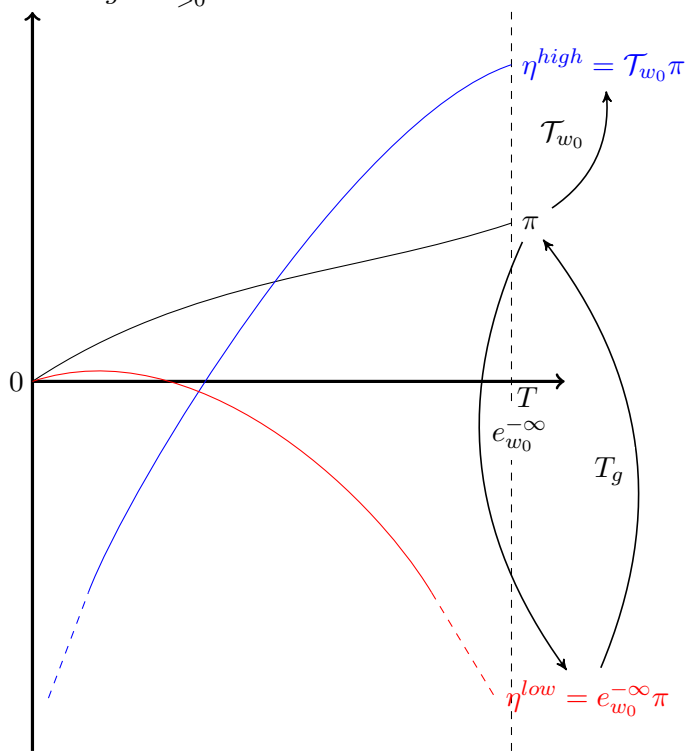

\subsection{High path transforms}
A first example giving extended paths was introduced in \cite{bib:BBO}: for every simple root $\alpha$, define the following transform of a continuous path $\pi \in C(\R_+^*, \afrak )$, such that $e^{-\alpha(\pi)}$ is integrable at the neighborhood of $0$:
$$ \forall t > 0, \Tc_\alpha(\pi)(t) := \pi(t) + \log\left(\int_0^t e^{-\alpha(\pi(s))} ds\right) \alpha^{\vee} $$
These degenerate to the tropical Pitman operators $\Pc_{\alpha^\vee} := \lim_{\varepsilon \rightarrow 0} \varepsilon \Tc_{\alpha} \varepsilon^{-1}$ where:
$$ \forall t>0, \Pc_{\alpha^\vee}(\pi)(t) = \pi(t) - \inf_{0\leq s \leq t} \alpha\left(\pi(s)\right) \alpha^{\vee}$$
Notice that $\Tc_\alpha(\pi)$ is not defined at zero. Therefore, contrary to the tropical case, composing the transforms $\left(\Tc_\alpha\right)_{\alpha \in \Delta}$ requires some care. It was also proven that the $\left(\Tc_\alpha\right)_{\alpha \in \Delta}$ satisfy the braid relationships. We give a simpler proof shortly.
\begin{thm}[\cite{bib:BBO}]
\label{thm:braid_relations_Tc}
If $\bar{w} = \bar{s}_{i_1} \dots \bar{s}_{i_k}$ is a representative in $G$ of $w\in W$ written in a reduced fashion and $\pi$ is a continuous path, then:
$$ e^{\Tc_w(\pi)_t} := [\bar{w}^{-1} B_t(\pi)]_0$$ 
is well defined for $t>0$ and
$$ \Tc_w =  \Tc_{\alpha_{i_k}} \circ \dots \circ \Tc_{\alpha_{i_2}} \circ \Tc_{\alpha_{i_1}}$$
Moreover, the operators $(\Tc_\alpha)_{\alpha \in \Delta}$ satisfy the braid relationships.
\end{thm}

The operator $\Tc_{w_0}$ arises then in a very natural way as the highest weight path transform for the geometric Littelmann model. Indeed, considering a connected geometric crystal, because crystal actions are free, there is no such thing as a dominant path that could be preferred, unlike the $h=0$ case considered in \cite{bib:BBO2}, or the original setting considered by Littelmann. Hence the idea of finding an invariant under crystal actions that will play that role. In the group picture, we have already introduced a notion of highest weight in definition \ref{def:highest_lowest_weight} which fullfills that purpose as invariant (lemma \ref{lemma:hw_is_invariant}). Now that we have at our disposal the projection map $p$, it is natural to transport the definition of highest weight from the group picture. 

\begin{definition}
If $\pi \in C( [0, T], \afrak )$ the associated highest weight is given by:
$$ \hw(\pi) := \hw( B_T(\pi) )$$
\end{definition}

And evidently:
\begin{proposition}
If $\pi \in C( [0, T], \afrak )$ then:
$$  \hw(\pi) = \Tc_{w_0}(\pi)(T)$$
\end{proposition}

Now, going back to discussing the braid relations, we make a simple remark:
\begin{align*}
\forall t>0, T_{\xi^{-\alpha^\vee} x_\alpha\left( \xi \right) }( \pi )(t)
& = \pi(t) + \log\left( \frac{1}{\xi} + \int_0^t e^{-\alpha(\pi(s))} ds\right) \alpha^{\vee}\\
& \stackrel{ \xi \rightarrow +\infty }{\longrightarrow}  \Tc_\alpha(\pi)(t)
\end{align*} 
This will allow us to give a simpler proof of the braid relations for $(\Tc_\alpha)_{\alpha \in \Delta}$ using the path transform properties \ref{lbl:path_transform_properties} as suggested at the end of \cite{bib:BBO2} (section 6.6).

\begin{proof}[Proof of theorem \ref{thm:braid_relations_Tc}]
We will show that the operators $(\Tc_\alpha)_{\alpha \in \Delta}$ satisfy the braid relationships as a consequence of the fact that the representatives $(\bar{s}_\alpha)_{\alpha \in \Delta}$ also satisfy them.

Let $\pi$ a continuous path in $\afrak$ and fix $g \in H_{>0} U^w_{>0}$ such that:
$$ g = \xi_k^{h_{i_k}} x_{i_k} \left( \frac{1}{\xi_k} \right) \dots \xi_1^{h_{i_1}} x_{i_1}\left( \frac{1}{\xi_1} \right) $$
for parameters $\xi_i > 0$. We will make use of the approximants:
\begin{align} 
\label{lbl:weyl_group_approximants}
\bar{s}_i(t) & := \phi_i\left(  \left( \begin{array}{cc} t & -1 \\ 1 & 0 \end{array}\right) \right) = y_i\left( -\frac{1}{t} \right) t^{h_i} x_i\left( \frac{1}{t} \right)\\
\bar{\bar{s}}_i(t) & := \phi_i\left(  \left( \begin{array}{cc} t & 1 \\ -1 & 0 \end{array}\right) \right) = y_i\left( \frac{1}{t} \right) t^{h_i} x_i\left( -\frac{1}{t} \right)
\end{align}
which converges respectively to $\bar{s}_i$ and $\bar{\bar{s}}_i = \bar{s}_{i_1}^{-1}$ as the parameter goes to zero. Let us start by writing:
\begin{align*}
\forall t>0, e^{T_g\left( \pi \right)(t)}
& = \left[g B_t(\pi)\right]_0\\
& = \left[\xi_k^{h_{i_k}} x_{i_k}\left( \frac{1}{\xi_k} \right)
    \dots \xi_1^{h_{i_1}} x_{i_1}\left( \frac{1}{\xi_1} \right)
          B_t(\pi)\right]_0\\
& = \left[ y_{i_k}\left( \frac{1}{\xi_k} \right) \bar{\bar{s}}_{i_k}( \xi_k )
     \dots y_{i_1}\left( \frac{1}{\xi_1} \right) \bar{\bar{s}}_{i_1}( \xi_1 )
           B_t(\pi) \right]_0\\ 
\end{align*}
Here we need to make $\xi_j$ successively go to zero in the decreasing order $j=k, \dots, 1$. Let us prove by induction that at the step $k \geq j > 1$, we get the quantity:
\begin{align}
 \label{eqn:quantity_j}
 \left[ \overline{ \overline{s_{i_k} \dots s_{i_{j+1}} }}
                y_{i_j}\left( \frac{1}{\xi_j} \right) \bar{\bar{s}}_{i_j}( \xi_j ) 
          \dots y_{i_1}\left( \frac{1}{\xi_1} \right) \bar{\bar{s}}_{i_1}( \xi_1 )
           B_t(\pi) \right]_0
\end{align}
Since $B_t(\pi)$ is totally positive inside $B$ (theorem \ref{thm:flow_B_total_positivity}), then for all $j=k, \dots, 1$, it is also the case for 
$$ g_j = \xi_j^{h_{i_j}} x_{i_j}\left( \frac{1}{\xi_j} \right)
    \dots \xi_1^{h_{i_1}} x_{i_1}\left( \frac{1}{\xi_1} \right) B_t(\pi)$$
Therefore, the minors $\Delta^{\omega_\alpha}\left( \overline{s_{i_{j+1}} \dots s_{i_k}}^{-1} g_j\right)$ are non zero, a topologically open property that stays valid for the $\xi_j$ in a neighborhood of zero. Hence, taking those limits and considering those Gauss decompositions is allowed.

First, at step $j=k$, we can get rid of $y_{i_k} \left( \frac{1}{\xi_k} \right) \in N$:
\begin{align*}
e^{T_g\left( \pi \right)(t)}
& = \left[ y_{i_k}\left( \frac{1}{\xi_k} \right) \bar{\bar{s}}_{i_k}( \xi_k )
     \dots y_{i_1}\left( \frac{1}{\xi_1} \right) \bar{\bar{s}}_{i_1}( \xi_1 )
           B_t(\pi) \right]_0\\ 
& \stackrel{ \xi_k \rightarrow 0 }{\longrightarrow}
    \left[ \bar{\bar{s}}_{i_k}
     \dots y_{i_1}\left( \frac{1}{\xi_1} \right) \bar{\bar{s}}_{i_1}( \xi_1 )
           B_t(\pi) \right]_0
\end{align*}
Now, assume that equation \ref{eqn:quantity_j} is proven for step $j$:
$$ \left[ \overline{ \overline{s_{i_k} \dots s_{i_{j+1}} }}
                y_{i_j}\left( \frac{1}{\xi_j} \right) \bar{\bar{s}}_{i_j}( \xi_j ) 
          \dots y_{i_1}\left( \frac{1}{\xi_1} \right) \bar{\bar{s}}_{i_1}( \xi_1 )
           B_t(\pi) \right]_0$$
Here, write:
$$ \overline{ \overline{s_{i_k} \dots s_{i_{j+1}} }} y_{i_j}\left( \frac{1}{\xi_j} \right)
 = \exp\left( \frac{1}{\xi_j} \Ad(\overline{ \overline{s_{i_k} \dots s_{i_{j+1}} }}) f_{i_j} \right) 
   \overline{ \overline{s_{i_k} \dots s_{i_{j+1}} }}$$
And since $\Ad(\overline{ \overline{s_{i_k} \dots s_{i_{j+1}} }}) f_{i_j} \in \mathfrak{g}_{ -s_{i_k} \dots s_{i_{j+1}} \alpha_{i_j} } \subset \nfrak$, one has:
\begin{align*}
  & \left[ \overline{ \overline{s_{i_k} \dots s_{i_{j+1}} }}
                  y_{i_j}\left( \frac{1}{\xi_j} \right) \bar{\bar{s}}_{i_j}( \xi_j ) 
            \dots y_{i_1}\left( \frac{1}{\xi_1} \right) \bar{\bar{s}}_{i_1}( \xi_1 )
            B_t(\pi) \right]_0\\
= & \left[ \exp\left( \frac{1}{\xi_j} \Ad(\overline{ \overline{s_{i_k} \dots s_{i_{j+1}} }}) f_{i_j} \right) 
           \overline{ \overline{s_{i_k} \dots s_{i_{j+1}} }} \bar{\bar{s}}_{i_j}( \xi_j ) 
            \dots y_{i_1}\left( \frac{1}{\xi_1} \right) \bar{\bar{s}}_{i_1}( \xi_1 )
            B_t(\pi) \right]_0\\
= & \left[ \overline{ \overline{s_{i_k} \dots s_{i_{j+1}} }} \bar{\bar{s}}_{i_j}( \xi_j ) 
            \dots y_{i_1}\left( \frac{1}{\xi_1} \right) \bar{\bar{s}}_{i_1}( \xi_1 )
            B_t(\pi) \right]_0\\
& \stackrel{ \xi_j \rightarrow 0}{\longrightarrow}
    \left[ \overline{ \overline{s_{i_k} \dots s_{i_{j+1}} s_{i_j} }}
            \dots y_{i_1}\left( \frac{1}{\xi_1} \right) \bar{\bar{s}}_{i_1}( \xi_1 )
            B_t(\pi) \right]_0
\end{align*}
The previous limit gives step $j-1$.

At the end, we get:
$$\left[\overline{ \overline{s_{i_k} \dots s_{i_1} }} B_t(\pi) \right]_0 = \left[\bar{w}^{-1} B_t(\pi) \right]_0 = \exp\left( \Tc_w \pi(t) \right)$$ 
On the other hand, because the group elements belong to the appropriate sets, we can use the composition property among properties \ref{lbl:path_transform_properties}:
\begin{align*}
\forall t>0, T_g\left( \pi \right)(t)
& = T_{ \xi_1^{h_{i_k}} x_{i_k} \left( \frac{1}{\xi_k} \right) } \circ \dots \circ T_{ \xi_1^{h_{i_1}} x_{i_1} \left( \frac{1}{\xi_1} \right) }\left( \pi \right)(t)\\
& \longrightarrow  \Tc_{\alpha_k} \circ \dots \circ \Tc_{\alpha_1}(\pi)(t)
\end{align*}
The previous limit makes sense if and only if for every $j=1, \dots, k$, $e^{-\alpha_{i_j}( \Tc_{s_{i_1} \dots s_{i_{j-1}}}\pi)}$ is integrable. Later, we have a much precise description of this integrability property, but for now, we already know thanks to the previous computation that the Gauss decompositions exist at every level. Therefore, the highest path transforms $(\Tc_\alpha)_{\alpha \in \Delta}$ must have been applied to paths with the appropriate integrability property. Identifying both limits, the braid relationships are proven:
$$\forall t>0,  \Tc_{\alpha_k} \circ \dots \circ \Tc_{\alpha_1}(\pi)(t) = \Tc_w \pi(t)$$
\end{proof}

\subsection{Low path transforms}
Define $e^{-\infty}_\alpha: \Cc_0\left( [0; T), \afrak \right) \rightarrow \Cc_0\left( [0; T), \afrak \right)$ as
$$ \forall 0 \leq t < T, e^{-\infty}_\alpha \cdot \pi(t) := \pi(t) + \log\left( 1 - \frac{ \int_0^t e^{-\alpha(\pi)} }{ \int_0^T e^{-\alpha(\pi)} } \right) \alpha^\vee$$
Notice that $T$ is excluded and that this path transform makes sense even if $\int_0^T e^{\alpha(\pi)} = \infty$. The notation obviously comes from the fact that $e^{-\infty}_\alpha = \lim_{c \rightarrow -\infty} e^c_\alpha$, hence the name of 'low' path transforms.

Clearly, $e^{-\infty}_\alpha$ is a projection in the sense that $e^c_\alpha \cdot e^{-\infty}_\alpha = e^{-\infty}_\alpha \cdot e^c_\alpha = e^{-\infty}_\alpha$ and it stabilizes paths $\pi$ such that $\int_0^T e^{-\alpha(\pi)} = +\infty$. In fact, we can associate such transforms to each element of the Weyl group:
\begin{definition}
Given a reduced expression $w=s_{i_1} \dots s_{i_l}$ for $w \in W$ and $\ell(w) = l$, $e^{-\infty}_w$ is defined unambiguously as 
$$ e^{-\infty}_w = e^{-\infty}_{\alpha_{i_l}} \dots e^{-\infty}_{\alpha_{i_1}} $$
\end{definition}
\begin{proof}
A first proof uses lemmma \ref{lemma:high_low_duality} and the braid relations for $\left( \Tc_w \right)_{w \in W}$. A second proof consists of using the Verma relations in order to check the claim in the case of a braid move: $w = s_i s_j s_i \dots = s_j s_i s_j \dots \ $. Indeed let $\left( \beta_1, \beta_2, \beta_3, \dots \right)$ and  $\left( \beta_1', \beta_2', \beta_3', \dots \right)$ the two positive roots enumerations associated to each reduced word.
$$\forall t \in \afrak, e_w^t = e^{ \beta_1(t) }_{\alpha_i} \cdot e^{ \beta_2(t) }_{\alpha_j} \dots = e^{ \beta_1'(t) }_{\alpha_j} \cdot e^{ \beta_2'(t) }_{\alpha_i} \dots $$
Then take $t = - M \mu$ with $\mu$ in the open Weyl chamber, $M$ a real number, and have $M$ go to $+\infty$.
\end{proof}

The name of 'high' path transforms for $\left( \Tc_w \right)_{w \in W}$ is justified by the fact that they are dual to 'low' path transforms:
\begin{lemma}
\label{lemma:high_low_duality}
$$ e^{-\infty}_{\alpha} \circ \iota = \iota \circ \Tc_{\alpha}  $$
And for $w \in W$:
$$ e^{-\infty}_{w} \circ \iota = \iota \circ \Tc_{w} $$ 
\end{lemma}
\begin{proof}
The first identity is a quick computation. The second one is a consequence.
\end{proof}

\begin{rmk}
Notice that it does not make sense to apply the duality map $\iota$ after $e^{-\infty}_w$ for $w \in W$, since it produces a path lacking an endpoint. Though, an extended duality holds, as we will see.

Moreover, the transforms $\left( \Tc_w \right)_{w \in W}$ are not projections.
\end{rmk}

Of course $e^{-\infty}_{w_0}$ is special projection as:
$$\forall \alpha \in \Delta, \forall c \in \R, e^{-\infty}_{w_0} \cdot e^{c}_{\alpha} = e^{c}_{\alpha} \cdot e^{-\infty}_{w_0} = e^{-\infty}_{w_0}$$
The following proposition shows that it is constant on the crystal's components:
\begin{proposition}
\label{proposition:partial_connectedness_criterion}
 If $\pi_1$ and $\pi_2$ are connected then $e^{-\infty}_{w_0} \cdot \pi_1 = e^{-\infty}_{w_0} \cdot \pi_2$
\end{proposition}
\begin{proof}
Being connected means that there are real numbers $\left(c_1, \dots, c_l \right)$ and indices 
$\left( i_1, \dots, i_l \right)$ such that:
$$\pi_2 = e^{c_1}_{\alpha_{i_1}} \dots e^{c_l}_{\alpha_{i_l}} \pi_1$$
Then:
$$e^{-\infty}_{w_0} \cdot \pi_2 = e^{-\infty}_{w_0} \cdot e^{c_1}_{\alpha_{i_1}} \dots e^{c_l}_{\alpha_{i_l}} \pi_1 = e^{-\infty}_{w_0} \cdot \pi_1$$
\end{proof}

In theorem \ref{thm:connectedness_criterion}, we will see that the converse is true giving a connectedness criterion.

\subsection{A certain property of the Weyl co-vector}
While moving to the edges of geometric crystals, we will obtain paths that can blow-up in finite time. The direction taken to go to infinity will be of utmost importance, involving the Weyl co-vector:
$$ \rho^\vee := \sum_{\alpha \in \Delta} \omega_\alpha^\vee = \frac{1}{2} \sum_{ \beta \in \Phi^+ } \beta^\vee$$

We will need a little property linking $\rho^\vee$ and the weak Bruhat order.
\begin{lemma}
  \label{lemma:rho_property}
  Let $w = s_{i_1} s_{i_2} \dots s_{i_k} \in W$ with $\ell(w) = k$. It defines a positive roots enumeration 
  $\left( \beta_1, \beta_2, \dots, \beta_k \right)$. Then:
  \begin{itemize}
   \item $$\rho^\vee - w \rho^\vee = \beta_1^\vee + \beta_2^\vee + \dots + \beta_k^\vee$$ 
   \item $\ell(s_\alpha w) = \ell(w) + 1$ if and only if $-\alpha\left( \rho^\vee - w \rho^\vee \right) \geq 0$
  \end{itemize}
\end{lemma}
\begin{proof}
 The first statement comes as an application of lemma \ref{lbl:kumar}. Concerning the second, following Bourbaki ( \cite{bib:Bourbaki}, Ch. V, \S 3, Th. 1, (ii)), $\ell\left(s_\alpha w\right) = 1 + \ell\left(w\right)$ is equivalent to saying that $C$ and $w\left( C \right)$ are on the same side of the wall associated to $\alpha$. As the Weyl co-vector is inside $C$, it tantamounts to $\alpha\left( w \rho^\vee \right) > 0$. In the end:
$$ -\alpha\left( \rho^\vee - w \rho^\vee \right) > -1 $$
The proof is finished once we notice that the left-hand side is an integer.
\end{proof}
\begin{rmk}
If $w$ is taken as $w_0$ the longest element, we recover the identity we used to define the Weyl co-vector.
\end{rmk}

\subsection{Extended path types}
The paths transforms $\left( \Tc_w \right)_{w \in W}$ ( resp. $\left( e^{-\infty}_w \right)_{w \in W}$) give paths that lack a starting (resp. an ending) point. In order to examine the possible asymptotics, let us first start by a simple lemma:
\begin{lemma}
 \label{lemma:start_asymptotics}
 If $\pi \in \Cc\left( [0, T], \afrak \right)$ then for all $w \in W$:
$$ \Tc_w \pi(t) = w^{-1} \pi(0) + \log(t)\left( \rho^\vee - w^{-1} \rho^\vee \right) + c_w + o(1)$$
where $c_w$ is a constant depending only on $w$ and $o(1)$ goes to zero as $t \rightarrow 0$.
\end{lemma}
\begin{proof}
By induction on $\ell(w)$. If $\ell(w)=0$, then $w=e$ and the result is obvious ($c_e=0$). If $w = u s_\alpha$ with $u \in W$ and $\ell(w) = \ell(u) + 1$, then:
$$\Tc_w = \Tc_\alpha \circ \Tc_{u}$$
Using the induction hypothesis, for $s>0$:
$$e^{-\alpha(\Tc_u\pi(s))} = e^{-\alpha\left( u^{-1} \pi(0) + c_u\right) + o(1)} s^{ -\alpha( \rho^\vee - u \rho^\vee) }$$
Lemma \ref{lemma:rho_property} applied to $w^{-1}$ tells us that $-\alpha( \rho^\vee - u^{-1} \rho^\vee) \geq 0$, and $e^{-\alpha(\Tc_u\pi(s))}$ is integrable at the neighborhood of zero. Since integrating equivalents is allowed, and using that $\alpha\left( \rho^\vee \right) = 1$:
\begin{align*}
  & \int_0^t e^{-\alpha(\Tc_u \pi)}\\
= & e^{-\alpha\left( u^{-1} \pi(0) + c_u\right) + o(1)} \int_0^t s^{ -\alpha( \rho^\vee - u^{-1} \rho^\vee) } ds\\
= & e^{-\alpha\left( u^{-1} \pi(0) + c_u\right) + o(1)} \frac{t^{1-\alpha( \rho^\vee - u^{-1} \rho^\vee) }}{1-\alpha( \rho^\vee - u^{-1} \rho^\vee)}\\
= & e^{-\alpha\left( u^{-1} \pi(0) + c_u\right) + o(1)} \frac{t^{ (u\alpha)( \rho^\vee) }}{(u\alpha)( \rho^\vee)}
\end{align*}
Then:
\begin{align*}
\Tc_w \pi(t) & = \Tc_{u}\pi(t) + \alpha^\vee \log \int_0^t e^{-\alpha(\Tc_u \pi)}\\
& = u^{-1} \pi(0) + \log(t)\left( \rho^\vee - u^{-1}\rho^\vee \right) + c_u + \\
& \ \ -\alpha\left( u^{-1} \pi(0) + c_u\right)\alpha^\vee  + \log \frac{t^{ (u\alpha)( \rho^\vee) }}{(u\alpha)( \rho^\vee)}\alpha^\vee  + o(1)\\
& = w^{-1} \pi(0) + \log(t) \left( \rho^\vee - u^{-1} \rho^\vee + (u\alpha)(\rho^\vee) \alpha^\vee \right) + s_\alpha c_u - \alpha^\vee \log\left( (u\alpha)( \rho^\vee)\right) + o(1)\\
& = w^{-1} \pi(0) + \log(t) \left( \rho^\vee - w^{-1} \rho^\vee \right) + s_\alpha c_u - \alpha^\vee \log\left( (u\alpha)( \rho^\vee)\right) + o(1)
\end{align*}
Set $c_w = s_\alpha c_u - \alpha^\vee \log\left( (u\alpha)( \rho^\vee)\right)$.

Finally, in order to prove that $c_w$ depends only on $w$ and not the reduced expression used, we invoke the fact that the asymptotic development is unique and $\Tc_w$ depends only on $w$.
\end{proof}

For completeness, we give an explicit expression for the constants $c_w, w \in W$, without proof, as we will not use them:
\begin{lemma}
$$ \forall w \in W, c_w = w^{-1} \sum_{\beta \in \Inv(w)} \log\left( \beta(\rho^\vee) \right) \beta^\vee$$
where $\Inv(w)$ is the set of inversions of $W$.
\end{lemma}

Lemma \ref{lemma:start_asymptotics} suggests to allow paths with undefined starting point, and to distinguish between them depending on their asymptotic behaviour at $t=0$, hence a definition:

\begin{definition}
\label{def:high_path_types}
For $T > 0$, we say that a path $\pi \in \Cc\left( (0, T], \afrak \right)$ is a high path of type $w \in W$ when the following asymptotic development holds at $0$:
$$ \pi(t) = \log(t)\left( \rho^\vee - w^{-1} \rho^\vee \right) + c_w + o(1)$$
The set of all high paths of type $w \in W$ in $\Cc\left( (0, T], \afrak \right)$ is denoted $\Cc^{high}_w\left( (0, T], \afrak \right)$.
\end{definition}
\begin{rmk}
As $c_e = 0$, $\Cc^{high}_e\left( (0, T], \afrak \right) = \Cc_0\left( [0, T], \afrak \right)$.
\end{rmk}

By duality, having in mind remark \ref{lemma:high_low_duality}:
\begin{definition}
\label{def:low_path_types}
For $T > 0$, we say that a path $\pi \in \Cc_0\left( [0; T), \afrak \right)$ is a low path of type $w \in W$ when the following asymptotic development holds at $T$:
$$ \exists C_\pi, \pi(t) = C_\pi + \log(T-t)\left( \rho^\vee - w^{-1} \rho^\vee \right) + c_w + o(1)$$
The set of all low paths of type $w \in W$ in $\Cc\left( [0, T), \afrak \right)$ is denoted $\Cc^{low}_w\left( [0, T), \afrak \right)$.
\end{definition}
\begin{rmk}
A low path $\pi$ of type $e$ has a continuous extension at $t=T$ by letting $\pi(T)=C_\pi$, hence:
$$\Cc^{low}_e = \Cc_0\left( [0, T], \afrak \right)$$
\end{rmk}

Both high and low paths will be referred to as extended paths. Clearly, the extended paths of type $w_0$ deserve a special name. As such, high (resp. low) paths of type $w_0$ will be referred to as highest (resp. lowest) paths. In that sense, the geometric Pitman transform gives the highest path of a crystal.

\section{Parametrizing geometric path crystals}
\label{section:parametrizations_in_path_model}

In this section, we will show how connected components of a geometric Littelmann path crystal are parametrized by the totally positive group elements. In the same fashion as in the group picture (section \ref{section:geom_lifting}), we will associate to every path its Lusztig parameter, an element in $U^{w_0}_{>0}$ or equivalently a Kashiwara parameter in $C^{w_0}_{>0}$. More precisely, we fix a time horizon $T>0$, then consider a path $\pi \in \Cc_0\left( [0, T], \afrak \right)$ and the crystal $\langle \pi \rangle$ it generates. We will show how Lusztig parameters and Kashiwara (or string) parameters can be retrieved from a path. The expressions found for string coordinates are geometric liftings of the formulas used in the classical Littelmann path model. Basically, we construct maps for every reduced word ${\bf i} \in R(w_0)$, $m = \ell(w_0)$:
$$\varrho_{\bf i}^L : \langle \pi \rangle \rightarrow \left(\R_{>0} \right)^m$$
$$\varrho_{\bf i}^K : \langle \pi \rangle \rightarrow \left(\R_{>0} \right)^m$$
The maps $\varrho_{\bf i}^L$ (resp. $\varrho_{\bf i}^K$) give the Lusztig (resp. Kashiwara) parameters for a path, in the sense that the diagram \ref{fig:parametrizations_diagram} is commutative (corollaries \ref{corollary:commutative_diagram_kashiwara} and \ref{corollary:commutative_diagram_lusztig}). This completes the group picture from figure \ref{fig:geom_parametrizations} and shows that indeed parametrizations of both the group picture and the path model are parallel.

\setcounter{figure}{ \value{equation} }
\addtocounter{equation}{1}
\begin{figure}[htp!]
\centering
\begin{tikzpicture}[baseline=(current bounding box.center)]
\matrix(m)[matrix of math nodes, row sep=3em, column sep=5em, text height=3ex, text depth=1ex, scale=1.2]
{ 
U_{>0}^{w_0} & \left( \R_{>0} \right)^m & \langle \pi \rangle    & \left( \R_{>0} \right)^m & C_{>0}^{w_0}\\
             &                          & \Bc(\lambda)\\
};
\path[->, font=\scriptsize] (m-1-3) edge node[above, scale=1.5]{$\varrho_{\bf i}^K$} (m-1-4);
\path[->, font=\scriptsize] (m-1-4) edge node[above, scale=1.5]{$x_{-\bf i}$} (m-1-5);
\path[->, font=\scriptsize] (m-1-3) edge node[above, scale=1.5]{$\varrho_{\bf i}^L$} (m-1-2);
\path[->, font=\scriptsize] (m-1-2) edge node[above, scale=1.5]{$x_{ \bf i}$} (m-1-1);

\path[->, font=\scriptsize] (m-1-3) edge node[left,  scale=1.5]{$p$} (m-2-3);
\draw[->] (m-2-3) to [bend right=30] node[right]{$p^{-1}$} (m-1-3);

\path[->, font=\scriptsize] (m-1-1) edge node[below, scale=1.5]{$b^L_\lambda$} (m-2-3);
\draw[->] (m-2-3) to [bend left=30] node[below]{$\varrho^L$} (m-1-1);
\path[->, font=\scriptsize] (m-1-5) edge node[below, scale=1.5]{$b^K_\lambda$} (m-2-3);
\draw[->] (m-2-3) to [bend right=30] node[below]{$\varrho^K$} (m-1-5);
\end{tikzpicture} 
\caption{Parametrizations for a connected crystal $\langle \pi \rangle$, with $\pi \in C_0([0, T], \afrak)$ and $\lambda = \Tc_{w_0} \pi(T)$}
\label{fig:parametrizations_diagram}
\end{figure}

Finally, for the purpose of greater generality, we will also consider a possibly infinite time horizon. In such a case, paths will need to have a drift inside the Weyl chamber (subsection \ref{subsection:infinite_T}).

\subsection{String parameters for paths}
\label{subsection:string_params}

\subsubsection{Definition}
Let ${\bf i} \in R(w_0)$. For any path $\pi \in \Cc_0\left( [0, T], \afrak \right)$, define $\varrho_{\bf i}^K(\pi)$ as the sequence of numbers ${ \bf c} = \left( c_1, c_2, \dots, c_m \right)$ recursively as:
$$ c_k = \frac{1}{\int_0^T e^{-\alpha_{i_k}\left( \Tc_{s_{i_1} \dots s_{i_{k-1}}} \pi \right)}}$$
\begin{thm}
\label{thm:string_params_paths}
The map $\varrho_{\bf i}^K$ is well-defined on $\Cc_0\left( [0, T], \afrak \right)$ and takes values in $\R_{>0}^m$. Moreover, for $\pi \in \Cc_0\left( [0, T], \afrak \right)$, the $m$-tuple $\varrho_{\bf i}^K(\pi)$ allows to recover $\pi$ from the highest weight path $\Tc_{w_0} \pi$.
\end{thm}
The proof is given soon after a few discussions.

\subsubsection{Extracting string parameters}
Building up on lemma \ref{lemma:start_asymptotics}, we will see when it is possible to apply $\left( \Tc_\alpha \right)_{\alpha \in \Delta}$ depending on a path's type. The importance of the weak Bruhat order is quite remarkable. Also, the use of the geometric Pitman operator $\Tc_\alpha$ for $\alpha \in \Delta$ corresponds to the loss of exactly one real number:
\begin{proposition}
\label{proposition:extracting_string}
Let $w \in W$ and $\alpha \in \Delta$ such that $\ell(w s_\alpha) = \ell(w) + 1$.

(1) If $\pi$ is a high path of type $w$ then $c = \frac{1}{\int_0^T e^{-\alpha(\pi)}} > 0$ and $\eta = \Tc_\alpha \pi$ has type $w s_\alpha$.

(2) Reciprocally, given a high path $\eta$ with type $w s_\alpha$ and a positive $c>0$, there is a unique high path $\pi$ of type $w$ such that $c = \frac{1}{\int_0^T e^{-\alpha(\pi)}} > 0$ and $\eta = \Tc_\alpha \pi$. It is given by:
$$ \forall 0 < t \leq T, \pi(t) = \eta(t) + \log\left( c + \int_t^T e^{-\alpha(\eta)} \right) \alpha^\vee$$
\end{proposition}
\begin{proof}
(1) Using lemma \ref{lemma:rho_property}, if $\pi$ is a high path of type $w$ and $\ell(w s_\alpha) = \ell(w) + 1$, then $e^{-\alpha(\pi)}$ is integrable at the neighborhood of zero, hence $c>0$.

Thanks to the proof of \ref{lemma:start_asymptotics}, we have seen that $\eta = \Tc_\alpha \pi$ will have the required asymptotics at $t=0$, so that it will be of type $w s_\alpha$.

(2) By composing the equality $\eta = \Tc_\alpha \pi$ with $e^{-\alpha(.)}$:
$$ e^{-\alpha(\eta(t))} = \frac{ e^{-\alpha(\pi(t))} }{\left(\int_0^t e^{-\alpha(\pi)}\right)^2} = - \frac{d}{dt} \frac{1}{\int_0^t e^{-\alpha(\pi)}} $$
Then after integration between $t>0$ and $T$:
$$ \forall t>0, \int_t^T e^{-\alpha(\eta)} = \frac{1}{\int_0^t e^{-\alpha(\pi)}} - c$$
Reinjecting this relation in the definition of $\eta$, we have:
$$ \forall t>0, \pi(t) = \eta(t) + \alpha^\vee \log\left( c + \int_t^T e^{-\alpha(\eta)} \right)$$
Finally, all that is left is to check that $\pi$ has the type $w$. The asymptotic development at $0$ follows from a computation similar to the proof of lemma \ref{lemma:start_asymptotics}. Indeed, since $\eta$ is of type $w s_\alpha$, we have the following asymptotics for $e^{-\alpha(\eta)}$ at zero:
$$ e^{-\alpha(\eta(t))} = e^{-\alpha(c_{w s_\alpha}) + o(1)}t^{-\alpha(\rho^\vee - s_\alpha w^{-1}\rho^\vee) }$$
As $-\alpha(\rho^\vee - s_\alpha w^{-1}\rho^\vee) \leq -1$ (lemma \ref{lemma:rho_property}), integrating $e^{-\alpha(\eta)}$ gives a divergent integral at zero. Therefore, we have the following equivalent for $t \rightarrow 0$:
\begin{align*}
  & \int_t^T e^{-\alpha(\eta)}\\
\sim \ & e^{-\alpha(c_{w s_\alpha}) + o(1)} \int_t^{T} s^{-\alpha(\rho^\vee - s_\alpha w^{-1}\rho^\vee) } ds\\
= & e^{-\alpha(c_{w s_\alpha}) + o(1)}
    \frac{\left[ s^{1-\alpha(\rho^\vee - s_\alpha w^{-1}\rho^\vee) } \right]_t^{T}}
         { 1-\alpha(\rho^\vee - s_\alpha w^{-1}\rho^\vee) }\\
\sim \ & \frac{ e^{-\alpha(c_{w s_\alpha}) + o(1)} }
         { (w \alpha)\left( \rho^\vee \right) t^{(w \alpha)\left( \rho^\vee \right)} }
\end{align*}
Because $\ell(w s_\alpha) = \ell(w) +1$, $w \alpha$ is a positive root and $w \alpha(\rho^\vee)>0$. Thus, we can write:
\begin{align*}
  & \pi(t)\\
= & \log(t)\left( \rho^\vee - s_\alpha w^{-1} \rho^\vee \right) + c_{w s_\alpha} + \log\left( c +  \frac{ e^{-\alpha(c_{w s_\alpha}) + o(1)} }{ (w \alpha)\left( \rho^\vee \right) t^{(w \alpha)\left( \rho^\vee \right)} } \right) \alpha^\vee + o(1)\\
= & \log(t)\left( \rho^\vee - s_\alpha w^{-1} \rho^\vee - (w \alpha)(\rho^\vee) \alpha^\vee \right) + s_\alpha c_{w s_\alpha} - \log\left( (w \alpha)( \rho^\vee ) \right)\alpha^\vee + o(1)\\
= & \log(t)\left( \rho^\vee - w^{-1} \rho^\vee \right) + s_\alpha c_{w s_\alpha} - \log\left( (w \alpha)( \rho^\vee ) \right)\alpha^\vee + o(1)
\end{align*}
Noticing that $s_\alpha c_{w s_\alpha} - \log\left( (w \alpha)( \rho^\vee ) \right)\alpha^\vee = c_w$ concludes the proof.
\end{proof}

\begin{proof}[Proof of theorem \ref{thm:string_params_paths}]
Start with $\pi \in \Cc_0\left( [0, T], \afrak \right)$. It is a high path of type $e$. When composing the geometric Pitman operators $\left( \Tc_\alpha \right)_{\alpha}$ while respecting the weak Bruhat order, we obtain paths whose types are climbing the Hasse diagram, until we reach $\Tc_{w_0}\pi$. At each step, exactly one positive real number is lost using proposition \ref{proposition:extracting_string}. 
\end{proof}

This drawing sums up the situation in the case of $A_2$:

\setcounter{figure}{ \value{equation} }
\addtocounter{equation}{1}
\begin{figure}[htp!]
\begin{center}
\begin{tikzpicture}[scale=0.7]
\useasboundingbox (0,0) rectangle (7.0cm,10.0cm);
\definecolor{black}{rgb}{0.0,0.0,0.0}
\definecolor{white}{rgb}{1.0,1.0,1.0}
\Vertex[style={minimum size=1, shape=circle}, x=2.5, y=10, L=$w_0$]{v0}
\Vertex[style={minimum size=1, shape=circle}, x=4,   y=7, L=$s_1 s_2$]{v1}
\Vertex[style={minimum size=1, shape=circle}, x=1,   y=7, L=$s_2 s_1$]{v2}
\Vertex[style={minimum size=1, shape=circle}, x=1,   y=4, L=$s_1$]{v3}
\Vertex[style={minimum size=1, shape=circle}, x=4,   y=4, L=$s_2$]{v4}
\Vertex[style={minimum size=1, shape=circle}, x=2.5, y=1, L=$e$]{v5}
\Edge[lw=0.1cm,style={post, color=black,},](v5)(v3)
\Edge[lw=0.1cm,style={post, color=black,},](v5)(v4)
\Edge[lw=0.1cm,style={post, color=black,},](v3)(v2)
\Edge[lw=0.1cm,style={post, color=black,},](v4)(v1)
\Edge[lw=0.1cm,style={post, color=black,},](v2)(v0)
\Edge[lw=0.1cm,style={post, color=black,},](v1)(v0)
\Edge[lw=0.1cm,style={post, bend left,color=black,}, label=$\Tc_{\alpha_1}$](v5)(v3)
\Edge[lw=0.1cm,style={post, bend left,color=black,}, label=$\Tc_{\alpha_2}$](v3)(v2)
\Edge[lw=0.1cm,style={post, bend left,color=black,}, label=$\Tc_{\alpha_1}$](v2)(v0)
\end{tikzpicture}
\end{center}
\label{lbl:string_hasse_diagram}
\caption{Extracting string parameters and climbing Hasse diagram of type $A_2$} 
\end{figure}
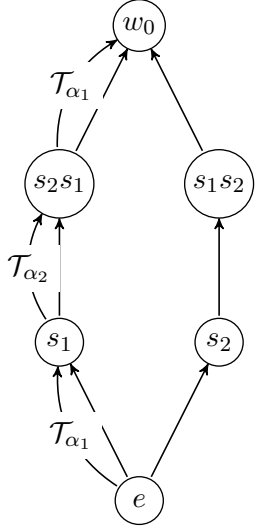

\subsubsection{Inversion lemma}
The inversion lemma is a bijective correspondence between $N_T(\pi)$ and the string parameters $\varrho_{\bf i}^K(\pi)$. Its proof is inspired from theorem 6.5 in \cite{bib:BBO2}.

\begin{thm}
\label{thm:string_inversion_lemma}
For ${\bf i} \in R(w_0)$ and $\pi \in \Cc_0\left( [0, T], \afrak \right)$:
$$ x_{\bf -i} \circ \varrho_{\bf i}^K\left( \pi \right) = \left[\bar{w_0}^{-1} N_T(\pi) \right]_{0+}^T$$
or equivalently:
$$ N_T(\pi) = \left[ \left( \bar{w}_0 \left( x_{\bf -i} \circ \varrho_{\bf i}^K\left( \pi \right) \right)^T \right)^{-1} \right]_-^{-1}$$
\end{thm}
\begin{proof}
In fact, this works with $w \in W$ and ${\bf i} = \left(i_1, \dots, i_j\right) \in R(w)$. Write:
\begin{align*}
  & \left[\bar{w}^{-1} N_T(\pi)\right]_{0+}\\
= & \left[\bar{w}^{-1} N_T(\pi)\right]_-^{-1} \bar{w}^{-1} N_T(\pi)\\
= & \left[\bar{w}^{-1} N_T(\pi)\right]_-^{-1} \bar{s}^{-1}_{i_j} \dots \bar{s}^{-1}_{i_1} N_T(\pi)\\
= &     \left[\bar{w}^{-1} N_T(\pi)\right]_-^{-1} \bar{s}^{-1}_{i_j} [\overline{ s_{i_1} \dots s_{i_{j-1}} }^{-1} N_T(\pi)]_-\\
  & \ \ \left[\overline{ s_{i_1} \dots s_{i_{j-1}} }^{-1} N_T(\pi)\right]_-^{-1} \bar{s}^{-1}_{i_{j-1}} [\overline{ s_{i_1} \dots s_{i_{j-2}} }^{-1} N_T(\pi)]_-\\
  & \ \ \left[\overline{ s_{i_1} \dots s_{i_{j-2}} }^{-1} N_T(\pi)\right]_-^{-1} \bar{s}^{-1}_{i_{j-2}} \dots \\
  & \ \ \dots \bar{s}^{-1}_{i_2} [\bar{s}^{-1}_{i_1} N_T(\pi)]_-\\
  & \ \ \left[\bar{s}^{-1}_{i_1} N_T(\pi)\right]_-^{-1} \bar{s}^{-1}_{i_1} N_T(\pi)\\
= &     \left[\bar{s}^{-1}_{i_j}     [\overline{ s_{i_1} \dots s_{i_{j-1}} }^{-1} N_T(\pi)]_- \right]_{0+}\\
  & \ \ \left[\bar{s}^{-1}_{i_{j-1}} [\overline{ s_{i_1} \dots s_{i_{j-2}} }^{-1} N_T(\pi)]_- \right]_{0+}\\
  & \ \ \dots \\
  & \ \ \left[\bar{s}^{-1}_{i_1}     N_T(\pi)\right]_{0+}
\end{align*}
Notice that each element $x_k = \left[\bar{s}^{-1}_{i_{k}} [\bar{s}^{-1}_{i_{k-1}} \dots \bar{s}^{-1}_{i_1} N_T(\pi)]_- \right]_{0+}$, $ 1 \leq k \leq j$, in the previous product, belongs to the reduced Bruhat cell $N \bar{s}^{-1}_{i_k} N \cap B^+$. Using theorem 4.5 in \cite{bib:BZ01}, we know $x_k$ is of the form $x_k = y_{-i_k}(c_k)$ where:
\begin{align*}
c_k^{-\alpha_{i_k}^\vee} = & [x_k]_0 \\
= & [\bar{s}^{-1}_{i_{k}} [\overline{ s_{i_1} \dots s_{i_{k-1}} }^{-1}N_T(\pi)]_- ]_{0}\\
= & [\bar{s}^{-1}_{i_{k}} [\overline{ s_{i_1} \dots s_{i_{k-1}} }^{-1} B_T(\pi)]_- ]_{0}\\
= & [\overline{ s_{i_1} \dots s_{i_{k}} }^{-1} B_T(\pi) [\overline{ s_{i_1} \dots s_{i_{k-1}} }^{-1} B_T(\pi)]_{0+}^{-1} ]_{0}\\
= & e^{ \Tc_{s_{i_1} \dots s_{i_k}} \pi - \Tc_{s_{i_1} \dots s_{i_{k-1}}} \pi }
\end{align*}
Hence we obtain exactly the string parameters
$$ c_k = \frac{1}{\int_0^T e^{-\alpha_{i_k}( \Tc_{s_{i_1} \dots s_{i_{k-1}}} \pi )}}$$
And
$$[\bar{w}^{-1} N_T(\pi)]_{0+} = y_{-i_j}(c_j) \dots y_{-i_1}(c_1)$$
Taking the transpose concludes the proof.

Finally, in order to see that the second expression can be deduced from the first, write 
\begin{align*}
  & N_T(\pi)\\
= & \left[ \left( \bar{w}_0 [\bar{w_0}^{-1} N_T(\pi)]_-^{-1} \bar{w}_0^{-1} N_T(\pi) \right)^{-1} \right]_-^{-1}\\
= & \left[ \left( \bar{w}_0 [\bar{w_0}^{-1} N_T(\pi)]_{0+} \right)^{-1} \right]_-^{-1}\\
= & \left[ \left( \bar{w}_0 \left( x_{\bf -i} \circ \varrho_{\bf i}^K\left( \pi \right) \right)^T \right)^{-1} \right]_-^{-1}
\end{align*}
\end{proof}

As a corollary, we get the commutativity of the right side in the diagram \ref{fig:parametrizations_diagram}:
\begin{corollary}
\label{corollary:commutative_diagram_kashiwara}
$$ \forall {\bf i} \in R(w_0), \forall \pi \in \Cc_0\left( [0,T], \afrak \right), x_{\bf -i} \circ \varrho_{\bf i}^K\left( \pi \right) = \varrho^K \circ p (\pi)$$
\end{corollary}
\begin{proof}
Theorem \ref{thm:string_inversion_lemma} and definition \ref{def:crystal_parameter} because $ [p(\pi)]_- = N_T(\pi)$.
\end{proof}

\subsection{Lusztig parameters for paths}
\label{subsection:lusztig_param}
In the same fashion, we define Lusztig parameters for a path and show how to extract them.

\subsubsection{Definition}
Let ${\bf i} \in R(w_0)$. For any path $\pi \in \Cc_0\left( [0, T], \afrak \right)$, define $\varrho_{\bf i}^L(\pi)$ as the sequence of numbers ${ \bf t } = \left( t_1, t_2, \dots, t_m \right)$ recursively as:
$$ t_k = \frac{1}{\int_0^T e^{-\alpha_{i_k}\left( e_{s_{i_1} \dots s_{i_{k-1}}}^{-\infty} \pi \right)}}$$
\begin{thm}
\label{thm:lusztig_params_paths}
The map $\varrho_{\bf i}^L$ is well-defined on $\Cc_0\left( [0, T], \afrak \right)$ and takes values in $\R_{>0}^m$. Moreover, for $\pi \in \Cc_0\left( [0, T], \afrak \right)$, the $m$-tuple $\varrho_{\bf i}^L(\pi)$ allows to recover $\pi$ from the lowest weight path $e_{w_0}^{-\infty} \pi$.
\end{thm}
Let us now explain how to prove this theorem carefully using duality.

\subsubsection{Extracting Lusztig parameters}
Let $w \in W$. It is obvious that if $\pi' \in \Cc^{high}_w\left( (0, T], \afrak\right)$ then $\iota(\pi')$ makes sense and belongs to $\Cc^{low}_w\left( [0, T), \afrak\right)$. The symmetric situation is not quite true, since one cannot apply the duality map $\iota$ to low paths. But we still have:
\begin{lemma}[Extended duality lemma]
\label{lemma:extended_duality}
Let $w \in W$. If $\pi \in \Cc^{low}_w\left( [0, T), \afrak\right)$ is a low path of type $w$, there is a unique $\pi' \in \Cc^{high}_w\left( (0, T], \afrak\right)$ such that:
$$ \pi = \left( \pi' \right)^\iota$$
\end{lemma}
\begin{proof}
Simply take for $0 \leq t < T$, $\pi'(t) = \pi(T-t) - C_\pi$, where $C_\pi$ is the constant appearing in the definition of low paths. It is immediate that it satisfies all the requirements.
\end{proof}
\begin{rmk}
Mapping $\pi$ to $\pi'$ can be thought of as an extension of the duality map to low paths.
\end{rmk}
\begin{rmk}
Putting together the previous lemma and lemma \ref{lemma:high_low_duality}, one sees that the path transforms $\left( e^{-\infty}_w \right)_{w \in W}$ and $\left( \Tc_w \right)_{w \in W}$ are in extended duality.
\end{rmk}

Thus, we can prove analogous statements to the previous case:
\begin{proposition}
\label{proposition:extracting_lusztig}
Let $w \in W$ and $\alpha \in \Delta$ such that $\ell(w s_\alpha) = \ell(w) + 1$.

(1) If $\pi$ is a low path of type $w$ then $c = \frac{1}{\int_0^T e^{-\alpha(\pi)}} > 0$ and $\eta = e_\alpha^{-\infty} \pi$ has type $w s_\alpha$.

(2) Reciprocally, given $\eta$, a low path with type $w s_\alpha$, and a positive $c>0$, there is a unique low path $\pi$ of type $w$ such that $c = \frac{1}{\int_0^T e^{-\alpha(\pi)}} > 0$ and $\eta = e_\alpha^{-\infty} \pi$. It is given by:
$$ \forall 0 \leq t < T, \pi(t) = T_{x_\alpha(c)} \eta(t) = \eta(t) + \log\left( 1 + c \int_0^t e^{-\alpha(\eta)} \right)\alpha^\vee$$
\end{proposition}
\begin{proof}
(1) Using lemma \ref{lemma:extended_duality}, we get $\pi'$ a high path of type $w$. Using proposition \ref{proposition:extracting_string}, we have:
$$c = \frac{1}{\int_0^T e^{-\alpha(\pi)}} = \frac{e^{-\alpha(\pi'(T))}}{\int_0^T e^{-\alpha(\pi')}} > 0$$
And:
$$ \eta = e_\alpha^{-\infty} \pi = e_\alpha^{-\infty} \circ \iota(\pi') = \iota \circ \Tc_\alpha(\pi') $$
which is a low path of type $w s_\alpha$ since $\Tc_\alpha(\pi') \in C_{w s_\alpha}^{high}\left( (0, T], \afrak \right)$.

(2) Again, using the extended duality lemma, there exists $\eta' \in C_{w s_\alpha}^{high}\left( (0, T], \afrak \right)$ such that $\eta = \iota(\eta')$. The result is proven by using (2) from proposition \ref{proposition:extracting_string}. In order to recover $\pi$ from $\eta$, rather than rearranging the formula from proposition \ref{proposition:extracting_string}, let us direcly solve:
$$ \eta(t) = \pi(t) + \alpha^\vee \log\left( 1 - c \int_0^t e^{-\alpha( \pi )} \right) $$
with $\pi$ as unknown and $\eta$ as variable. By evaluating the $e^{-\alpha(.)}$ on each side of the previous equality:
$$ e^{-\alpha\left(\eta(t)\right)} = \frac{ e^{-\alpha\left(\pi(t)\right)} }{ \left( 1 - c \int_0^t e^{-\alpha( \pi)}\right)^2 }$$
This expression can be integrated and rearranged as:
$$ \left( 1 + c \int_0^t e^{ -\alpha(\eta) } \right) \left( 1 - c \int_0^t e^{-\alpha(\pi)} \right) = 1$$
As such, since $\eta$ has type $w s_\alpha$, we have that $\int_0^T e^{-\alpha(\eta)} = +\infty$. Hence:
$$ c = \frac{1}{\int_0^T e^{-\alpha(\pi)} }$$
Also by replacing $\log\left( 1 - c \int_0^t e^{-\alpha(\pi)} \right)$ by $-\log\left( 1 + c \int_0^t e^{ -\alpha(\eta) } \right)$, we get the result:
$$ \pi(t) = \eta(t) + \log\left(1 + c \int_0^t e^{-\alpha(\eta)} \right) \alpha^\vee = T_{x_\alpha\left( c \right)} \eta (t)$$
\end{proof}

\begin{proof}[Proof of theorem \ref{thm:lusztig_params_paths}]
Apply iteratively proposition \ref{proposition:extracting_lusztig}. Successive projections give low paths whose types goes down the Hasse diagram. At every composition, exactly one positive real parameter is lost.  
\end{proof}

This drawing illustrates the situation in the case of $A_2$:

\setcounter{figure}{ \value{equation} }
\addtocounter{equation}{1}
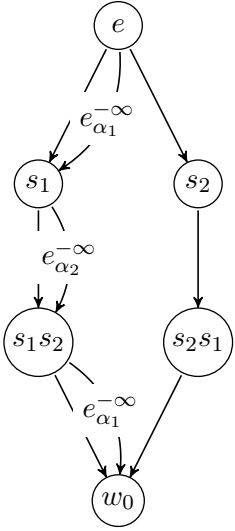
\begin{figure}[htp!]
\label{lbl:lusztig_hasse_diagram}
\caption{Extracting Lusztig parameters and going down the Hasse diagram of type $A_2$} 
\begin{center}
\begin{tikzpicture}[scale=0.7]
\useasboundingbox (0,0) rectangle (7.0cm,10.0cm);
\definecolor{black}{rgb}{0.0,0.0,0.0}
\definecolor{white}{rgb}{1.0,1.0,1.0}
\Vertex[style={minimum size=1, shape=circle}, x=2.5, y=10, L=$e$]{v0}
\Vertex[style={minimum size=1, shape=circle}, x=4,   y=7, L=$s_2$]{v1}
\Vertex[style={minimum size=1, shape=circle}, x=1,   y=7, L=$s_1$]{v2}
\Vertex[style={minimum size=1, shape=circle}, x=1,   y=4, L=$s_1 s_2$]{v3}
\Vertex[style={minimum size=1, shape=circle}, x=4,   y=4, L=$s_2 s_1$]{v4}
\Vertex[style={minimum size=1, shape=circle}, x=2.5, y=1, L=$w_0$]{v5}
\Edge[lw=0.1cm,style={post, color=black,},](v3)(v5)
\Edge[lw=0.1cm,style={post, color=black,},](v4)(v5)
\Edge[lw=0.1cm,style={post, color=black,},](v2)(v3)
\Edge[lw=0.1cm,style={post, color=black,},](v1)(v4)
\Edge[lw=0.1cm,style={post, color=black,},](v0)(v2)
\Edge[lw=0.1cm,style={post, color=black,},](v0)(v1)
\Edge[lw=0.1cm,style={post, bend left,color=black,}, label=$e^{-\infty}_{\alpha_1}$](v3)(v5)
\Edge[lw=0.1cm,style={post, bend left,color=black,}, label=$e^{-\infty}_{\alpha_2}$](v2)(v3)
\Edge[lw=0.1cm,style={post, bend left,color=black,}, label=$e^{-\infty}_{\alpha_1}$](v0)(v2)
\end{tikzpicture}
\end{center}
\end{figure}
 
This time, a path transform we already encountered appears:
\begin{thm}
\label{thm:from_lowest_path_2_path}
Given a lowest path $\eta$, a reduced word ${\bf i} = \left( i_1, \dots, i_m\right) \in R\left(w_0\right)$ and strictly positive parameters $\left(t_1, \dots, t_m\right) \in \R_{>0}^m$, there is a unique path $\pi \in \Cc_0\left( [0, T], \afrak \right)$ such that:
\begin{itemize}
 \item $e^{-\infty}_{w_0} \pi = \eta$
 \item $\eta_j = e^{-\infty}_{s_{i_1} \dots s_{i_j}} \pi = e^{-\infty}_{\alpha_{i_j}} \cdot \eta_{j-1}$
 \item $t_j = \frac{1}{\int_0^T e^{-\alpha_{i_j}( \eta_{j-1} ) } }$
\end{itemize}
  It is given by:
$$ \pi = T_z \eta$$
where
$$ z = x_{i_1}\left( t_1 \right) \dots x_{i_m}\left( t_m \right) \in U^{w_0}_{>0}$$
\end{thm}
\begin{proof}
At each level, $ \eta_{j-1} = T_{x_{\alpha_{i_j}}\left( \xi_j \right)} \eta_j$. Then, using the composition property among properties \ref{lbl:path_transform_properties}:
$$ \pi = T_{x_{\alpha_{i_1}}\left( \xi_1 \right)} \circ T_{x_{\alpha_{i_2}}\left( \xi_2 \right)} \circ \dots \circ T_{x_{\alpha_{i_j}}\left( \xi_j \right)}\left( \eta \right)
       = T_z\left( \eta \right)$$
\end{proof}
\begin{rmk}
A similar statement holds for paths of type $w$, using group elements in $U^{w}_{>0}$. And this $z \in U^{w_0}_{>0}$ is the Lusztig parameter, as we will see shortly.
\end{rmk}

\begin{corollary}
\label{corollary:connected_component_bijection}
A connected component generated by a path $\pi_0$ can be parametrized by the totally positive part $U^{w_0}_{>0}$ 
thanks to the bijection:
 $$\begin{array}{ccc}
    U^{w_0}_{>0} & \rightarrow & \langle \pi_0 \rangle\\
          z      & \mapsto     & T_z\left( e^{-\infty}_{w_0} \cdot \pi_0 \right) 
   \end{array}$$
\end{corollary}
\begin{proof}
Recall that $\eta = e^{-\infty}_{w_0} \pi$ does not depend on $\pi \in \langle \pi_0 \rangle$, but only on the connected component. And every path in $\pi \in \langle \pi_0 \rangle$ is uniquely determined by an $u \in U^{w_0}_{>0}$ such that $\pi = T_u \eta$ thanks to the previous theorem.
\end{proof}

\subsubsection{Inversion lemma}
Again, we have a bijective correspondence between $N_T(\pi)$ and the Lusztig parameters $\varrho_{\bf i}^L(\pi) \in \R_{>0}^{\ell(w_0)}$:

\begin{thm}
\label{thm:inversion_lemma_lusztig}
For $\eta \in \Cc^{low}_{w_0}\left( [0, T), \afrak \right)$, let $\pi = T_g \eta$ be is a crystal element with Lusztig parameters encoded by $g = x_{i_1}\left( t_1 \right) \dots x_{i_m}\left( t_m \right) = x_{\bf i} \circ \varrho_{\bf i}^L(\pi) \in U^{w_0}_{>0}$. Then:
$$ N_T\left( \pi \right) = [ g \bar{w}_0 ]_-$$
or equivalently
$$ g = [\bar{w}_0^{-1} N_T\left( \pi \right)^\iota ]_+^\iota = [\bar{w}_0^{-1} B_T\left( \pi \right)^\iota ]_+^\iota$$
\end{thm}
\begin{proof}
This can be deduced from theorem \ref{thm:string_inversion_lemma}. However, we choose to give a separate proof that is easily adapted to the case of $T=\infty$. First, in order to see that both identities are equivalent, since $\bar{w}_0^{-1} \left([g \bar{w}_0]_{0+}^{-1}\right)^\iota \bar{w}_0 \in B$, write:
\begin{align*}
g = & [g^\iota]_+^\iota\\
= & [\bar{w}_0^{-1} \left([g \bar{w}_0]_{0+}^{-1}\right)^\iota \bar{w}_0 g^\iota]_+^\iota\\
= & [\bar{w}_0^{-1} \left( g \bar{w}_0 [g \bar{w}_0]_{0+}^{-1}\right)^\iota]_+^\iota
\end{align*}
Therefore:
$$N_T\left( \pi \right) = [ g \bar{w}_0 ]_- = g \bar{w}_0 [g \bar{w}_0]_{0+}^{-1}$$
if and only if:
$$ g = [\bar{w}_0^{-1} N_T\left( \pi \right)^\iota]_+^\iota $$
One can also add the torus part and write:
$$ g = [\bar{w}_0^{-1} N_T\left( \pi \right)^\iota ]_+^\iota = [\bar{w}_0^{-1} B_T\left( \pi \right)^\iota ]_+^\iota$$

Now, let us prove the above statement using a similar decomposition to the one used in the proof of theorem \ref{thm:string_inversion_lemma}:
\begin{align*}
      & \left[\bar{w}_0^{-1} B_T\left( \pi^\iota \right) \right]_+\\
    = & \left[\bar{w}_0^{-1} B_T\left( \pi^\iota \right) \right]_{-0}^{-1} \bar{w}_0^{-1} B_T\left( \pi^\iota \right) \\
    = & \left[\bar{w}_0^{-1} B_T\left( \pi^\iota \right) \right]_{-0}^{-1} \bar{s}_{i_m}^{-1}
         \left[\overline{s_{i_1} \dots s_{i_{m-1}}}^{-1} B_T\left( \pi^\iota \right) \right]_{-0}\\
& \cdot \left[\overline{s_{i_1} \dots s_{i_{m-1}}}^{-1} B_T\left( \pi^\iota \right) \right]_{-0}^{-1} \bar{s}_{i_{m-1}}^{-1}
         \left[\overline{s_{i_1} \dots s_{i_{m-2}}}^{-1} B_T\left( \pi^\iota \right) \right]_{-0}\\
& \cdot \left[\overline{s_{i_1} \dots s_{i_{m-2}}}^{-1} B_T\left( \pi^\iota \right) \right]_{-0}^{-1} \bar{s}_{i_{m-2}}^{-1}
         \left[\overline{s_{i_1} \dots s_{i_{m-3}}}^{-1} B_T\left( \pi^\iota \right) \right]_{-0}\\
& \dots \dots \dots \dots\\
& \cdot \left[\overline{s_{i_1}}^{-1} B_T\left( \pi^\iota \right) \right]_{-0}^{-1} \bar{s}_{i_1}^{-1}
         B_T\left( \pi^\iota \right)\\
    = & \left[ \bar{s}_{i_m}^{-1}
         \left[\overline{s_{i_1} \dots s_{i_{m-1}}}^{-1} B_T\left( \pi^\iota \right) \right]_{-0} \right]_+\\
& \cdot \left[ \bar{s}_{i_{m-1}}^{-1}
         \left[\overline{s_{i_1} \dots s_{i_{m-2}}}^{-1} B_T\left( \pi^\iota \right) \right]_{-0} \right]_+\\
& \cdot \left[ \bar{s}_{i_{m-2}}^{-1}
         \left[\overline{s_{i_1} \dots s_{i_{m-3}}}^{-1} B_T\left( \pi^\iota \right) \right]_{-0} \right]_+\\
& \dots \dots \dots \dots\\
& \cdot \left[ \bar{s}_{i_1}^{-1}
         B_T\left( \pi^\iota \right) \right]_+
\end{align*}
In the previous equation, we have a product of $m$ terms, each of the form $x_j = [ \bar{s}_{i_j}^{-1} b ]_{-0}^{-1} \bar{s}_{i_j}^{-1} b = [ \bar{s}_{i_j}^{-1} b ]_+$, $1 \leq k \leq m$ with $b = \left[ \overline{s_{i_1} \dots s_{i_{j-1}}}^{-1}B_T\left( \pi^\iota\right) \right]_{-0}$. As it belongs to the reduced double Bruhat cell $U \cap B s_{i_j} B$, we can use theorem theorem 4.5 in \cite{bib:BZ01} and write $x_j = x_{\alpha_{i_j}}\left( t_j \right)$. The quantity $t_j$ will be computed at the end.
Hence:
$$ [\bar{w}_0^{-1} B_T\left( \pi^\iota \right) ]_+ = x_{i_m}\left( t_m \right) \dots x_{i_2}\left( t_2 \right) x_{i_1}\left( t_1 \right)$$
or equivalently:
$$ [\bar{w}_0^{-1} B_T\left( \pi \right)^\iota ]_+^\iota = x_{i_1}\left( t_1 \right) x_{i_2}\left( t_2 \right) \dots x_{i_m}\left( t_m \right)$$

Now, all that is left is to prove that the Lusztig parameters for $\pi$ are nothing but the quantities $t_j, j=1, \dots, m$. We have:
\begin{align*}
  & x_{\alpha_{i_j}}\left( t_j \right)\\
= & \left[ \bar{s}_{i_j}^{-1} \left[\overline{s_{i_1} \dots s_{i_{j-1}}}^{-1} B_T\left( \pi^\iota \right) \right]_{-0} \right]_+\\
= & \left[\overline{s_{i_1} \dots s_{i_{j}}}^{-1} B_T\left( \pi^\iota \right) \right]_{-0}^{-1}
    \bar{s}_{i_j}^{-1} \left[\overline{s_{i_1} \dots s_{i_{j-1}}}^{-1} B_T\left( \pi^\iota \right) \right]_{-0}\\
= & e^{-\Tc_{s_{i_1} \dots s_{i_j}} \circ \iota(\pi)(T)}
    \left[\overline{s_{i_1} \dots s_{i_{j}}}^{-1} B_T\left( \pi^\iota \right) \right]_{-}^{-1}
    \bar{s}_{i_j}^{-1} \left[\overline{s_{i_1} \dots s_{i_{j-1}}}^{-1} B_T\left( \pi^\iota \right) \right]_{-}
    e^{\Tc_{s_{i_1} \dots s_{i_{j-1}}} \circ \iota(\pi)(T)}\\
= & e^{-\Tc_{s_{i_1} \dots s_{i_j}} \circ \iota(\pi)(T)} y_j e^{\Tc_{s_{i_1} \dots s_{i_{j-1}}} \circ \iota(\pi)(T)}
\end{align*}
where $y_j = y_{-\alpha_{i_j}}(c_j) \in B^+ \cap N \bar{s}_{i_j}^{-1} N$. Necessarily:
\begin{align*}
c_j^{-\alpha_{i_j}^\vee} & = [y_j]_0\\
& = e^{\Tc_{s_{i_1} \dots s_{i_j}} \circ \iota(\pi)(T) - \Tc_{s_{i_1} \dots s_{i_{j-1}}} \circ \iota(\pi)(T)}\\
& = \exp\left( \log\int_0^T e^{ -\alpha_{i_j}(\Tc_{s_{i_1} \dots s_{i_{j-1}}} \circ \iota(\pi))} \alpha_{i_j}^\vee\right)
\end{align*}
Therefore:
$$c_j = \frac{1}{\int_0^T e^{ -\alpha_{i_j}(\Tc_{s_{i_1} \dots s_{i_{j-1}}} \circ \iota(\pi))}}$$
And:
\begin{align*}
t_j & = c_j \exp\left( -\alpha_{i_j}\left( \Tc_{s_{i_1} \dots s_{i_{j-1}}} \circ \iota(\pi)(T) \right) \right)\\
& = \frac{1}{\int_0^T e^{ -\alpha_{i_j}(\iota \circ \Tc_{s_{i_1} \dots s_{i_{j-1}}} \circ \iota(\pi))}}\\
& = \frac{1}{\int_0^T e^{ -\alpha_{i_j}( e^{-\infty}_{s_{i_1} \dots s_{i_{j-1}}}(\pi) )}}
\end{align*}
\end{proof}

Again, as a corollary, we have the commutativity of the left side in the diagram \ref{fig:parametrizations_diagram}:
\begin{corollary}
\label{corollary:commutative_diagram_lusztig}
$$ \forall {\bf i} \in R(w_0), \forall \pi \in \Cc_0\left( [0,T], \afrak \right), x_{\bf i} \circ \varrho_{\bf i}^L\left( \pi \right) = \varrho^L \circ p (\pi)$$
\end{corollary}
\begin{proof}
 Theorem \ref{thm:inversion_lemma_lusztig} and definition \ref{def:crystal_parameter}.
\end{proof}

\subsection{Crystal actions in coordinates}

The actions $\left( e^._\alpha \right)_{\alpha \in \Delta}$ have a very simple expression in the appropriate charts for a connected crystal. Clearly, this property is a geometric lifting of the equations from  \eqref{eqn:kashiwara_operator_f_1} to \eqref{eqn:kashiwara_operator_e_2} for Kashiwara operators.
 
\begin{proposition}
\label{proposition:actions_in_coordinates}
Consider a path $\pi \in \Cc_0\left( [0, T], \afrak \right)$, ${\bf i} \in R(w_0)$ and $\alpha = \alpha_{i_1}$. If:
$$ \varrho^L_{\bf i}(\pi) = \left( t_1, \dots, t_m \right)$$
$$ \varrho^K_{\bf i}(\pi) = \left( c_1, \dots, c_m \right)$$
Then for every $\xi \in \R$:
$$ \varrho^L_{\bf i}\left( e^\xi_\alpha \cdot \pi \right) = \left( e^\xi t_1, \dots, t_m \right)$$
$$ \varrho^K_{\bf i}\left( e^\xi_\alpha \cdot \pi \right) = \left( e^\xi c_1, \dots, c_m \right)$$
\end{proposition}

\begin{rmk}
Here, we do not require twisting the Kashiwara operators like in \cite{bib:BZ01} section 5.2.
\end{rmk}

We will need the following important property of the highest and lowest path transforms.
\begin{lemma}
For $\pi \in \Cc_0\left( [0, T], \afrak \right)$, $\alpha \in \Delta$ and $\xi \in \R$::
$$ e^{-\infty}_\alpha \cdot \left( e^\xi_\alpha \cdot \pi \right) = e^{-\infty}_\alpha \cdot \pi$$
$$ \Tc_\alpha \cdot \left( e^\xi_\alpha \cdot \pi \right) = \Tc_\alpha \cdot \pi$$
\end{lemma}
\begin{proof}
The first identity is obvious and has already been referred to. The second one can be proved either by direct computation or by using the extended duality between $e^{-\infty}_\alpha$ and $\Tc_\alpha$.
\end{proof}

\begin{proof}[Proof of proposition \ref{proposition:actions_in_coordinates}]
Write:
$$ \varrho^L_{\bf i}\left( e^\xi_\alpha \cdot \pi \right) = \left( t_1', \dots, t_m' \right)$$
$$ \varrho^K_{\bf i}\left( e^\xi_\alpha \cdot \pi \right) = \left( c_1', \dots, c_m' \right)$$
The previous lemma along with the definitions for Lusztig and Kashiwara parameters (see subsections \ref{subsection:lusztig_param} and \ref{subsection:string_params}) tell us that:
$$ \forall j \geq 2, t_j = t_j', c_j = c_j'$$
Moreover:
$$ t_1 = c_1 = \frac{1}{\int_0^T e^{-\alpha(\pi)} }$$
and:
$$ t_1' = c_1' = \frac{1}{\int_0^T e^{-\alpha(e^\xi \cdot \pi)} }
        = \frac{1}{e^{\varepsilon_\alpha(e^\xi \cdot \pi) }}
        = \frac{1}{e^{\varepsilon_\alpha(\pi) - \xi}}
        = e^\xi t_1
        = e^\xi c_1$$
\end{proof}

\subsection{Connectedness criterion}
Because of the previous investigations, it is easy to give a simple criterion that forces two paths to belong to the same connected component.

\begin{thm}
\label{thm:connectedness_criterion}
Consider two paths $\pi$ and $\pi'$ in $\Cc_0\left( [0,T], \afrak \right)$. The following propositions are equivalent:
\begin{itemize}
 \item[(i)] $\pi$ and $\pi'$ are connected.
 \item[(ii )] $$ \left( e^{-\infty}_{w_0}(\pi )_t, 0 \leq t < T \right)
               = \left( e^{-\infty}_{w_0}(\pi')_t, 0 \leq t < T \right) $$
 \item[(iii)] $$ \left( \Tc_{w_0}(\pi )_t, 0 < t \leq T \right)
               = \left( \Tc_{w_0}(\pi')_t, 0 < t \leq T \right) $$
\end{itemize}
\end{thm}
\begin{proof}
Proposition \ref{proposition:partial_connectedness_criterion} says that (i) implies (ii), while corollary \ref{corollary:connected_component_bijection} gives the converse.

In order to prove the equivalence between (i) and (iii), notice that $\pi$ and $\pi'$ are connected if and only if the same holds for their duals. Therefore, we have an equivalence between (i) and:
$$ \left( e^{-\infty}_{w_0} \circ \iota (\pi )_t, 0 \leq t < T \right)
               = \left( e^{-\infty}_{w_0} \circ \iota (\pi')_t, 0 \leq t < T \right) $$
Applying lemma \ref{lemma:high_low_duality}, we have the result.
\end{proof}

\subsection{The case of infinite time horizon}
\label{subsection:infinite_T}

Most of the previous results carry on the case where $T=\infty$. We will explain how to proceed in order to construct Lusztig parameters for a path $\pi \in C_0(\R_+, \afrak)$.

Clearly, low path transforms can be applied to paths in $\Cc_0\left( \R_+, \afrak \right)$. And:
\begin{lemma}
Let $\pi \in \Cc_0\left( \R_+, \afrak \right)$ such that $\pi(t) \underset{t \rightarrow \infty}{\sim} \mu t$ with $\alpha\left(\mu\right)>0$. Then:
$$\int_0^\infty e^{-\alpha(\pi)} < \infty$$
And
$$ e^{-\infty}_\alpha \pi(t) \underset{t \rightarrow \infty}{\sim} s_\alpha(\mu) t$$
\end{lemma}
\begin{proof}
The first assertion is clear. For the second, note that:
$$ \log \int_t^\infty e^{-\alpha\left( \pi(s) \right)}ds \underset{t \rightarrow \infty}{\sim} -\alpha\left(\mu\right) t $$
Therefore:
\begin{align*}
e^{-\infty}_\alpha \pi(t) & = \pi(t) + \log\left( 1 - \frac{ \int_0^t e^{\alpha(\pi)} }{ \int_0^\infty e^{\alpha(\pi)} } \right) \alpha^\vee\\
& = \pi(t) + \log\left(\int_t^\infty e^{\alpha(\pi)}\right) \alpha^\vee 
           - \log\left(\int_0^\infty e^{\alpha(\pi)} \right) \alpha^\vee\\
& \underset{t \rightarrow \infty}{\sim} s_\alpha(\mu) t
\end{align*}
\end{proof}

Hence the idea, that in this case, path types should depend on asymtotical behavior. 
\begin{definition}[Low path types in infinite horizon]
\label{def:low_path_types_infinite}
We say that a path $\pi \in \Cc_0\left( \R_+, \afrak \right)$ is a low path of type $w \in W$ when it has a drift in $wC$:
$$ \exists \mu \in wC, \pi(t) \underset{t \rightarrow \infty}{\sim} \mu t$$
The set of all low paths of type $w \in W$ in $\Cc\left( \R_+, \afrak \right)$ is denoted $\Cc^{low}_w\left( [0, T), \afrak \right)$.
\end{definition}

Lusztig parameters also have a straightforward definition. Let ${\bf i} \in R(w_0)$. For any path $\pi \in \Cc_0\left( \R_+, \afrak \right)$ with drift in the Weyl chamber, define $\varrho_{\bf i}^L(\pi)$ as the sequence of numbers ${ \bf t } = \left( t_1, t_2, \dots, t_m \right)$ recursively as:
$$ t_k = \frac{1}{\int_0^\infty \exp\left(-\alpha_{i_k}\left( e_{s_{i_1} \dots s_{i_{k-1}}}^{-\infty} \pi \right) \right)}$$
Thanks to the previous lemma, all $t_j$ are $>0$. Finally the inversion lemma \ref{thm:inversion_lemma_lusztig} is valid with infinite horizon. It is proven by simply taking $T$ to infinity.

\section{Isomorphism results}
\label{section:isom_results}

It is now easy to see that the projection map $p$ restricts to an isomorphism on connected components.
\begin{thm}
\label{thm:static_geometric_rs}
Let $\pi_0 \in \Cc_0\left([0,T], \afrak\right)$ and $\langle \pi_0 \rangle$ be the connected crystal it generates. Set $\lambda = \Tc_{w_0} \pi_{0}\left( T \right) $. Then the projection is an isomorphism of crystals:
$$ \begin{array}{cccc}
    p: & \langle \pi_0 \rangle & \rightarrow & \Bc\left( \lambda \right)\\
       &         \pi           &   \mapsto   & B_T\left( \pi \right)
   \end{array}
 $$
\end{thm}
\begin{proof}
The target set is indeed the appropriate one as:
$$ \forall \pi \in \langle \pi_0 \rangle, \hw\left( B_T\left( \pi \right) \right) = \Tc_{w_0} \pi\left( T \right) = \Tc_{w_0} \pi_{0}\left( T \right) = \lambda$$
thanks to theorem \ref{thm:connectedness_criterion}. The inverse map is constructed using the corollary \ref{corollary:connected_component_bijection}. 
\end{proof}

\subsection{Littelmann's independence theorem}

As a corollary, we have the geometric analogue of Littelmann's independence theorem ( Isomorphism Theorem in \cite{bib:Littelmann95}):
\begin{thm}[Geometric Littelmann independence theorem]
\label{thm:geometric_littelmann_independence}
For any connected geometric path crystal $L \subset \Cc_0\left( [0, T], \afrak \right)$, the crystal structure only depends on the endpoint $\lambda = \Tc_{w_0} \pi \left( T \right)$ of the highest path $\Tc_{w_0} \pi$.
\end{thm}

\subsection{Minimality of group picture}
\label{subsection:minimality}
Thanks to Littelmann's independence theorem, we have seen that there are a lot of different but isomorphic path crystals. And all of them project (thanks to the map $p$) to a certain $\Bc(\lambda)$, what we called the group picture. Now one can ask the question of how minimal this group picture is. A reasonable answer can be the fact that there are very few crystal morphisms on the group picture $\Bc(\lambda)$. Lifting the problem at the level of the path model allows an easy proof of the following.

\begin{thm}
\label{thm:minimality}
Let $f: \Bc(\lambda) \rightarrow \Bc(\mu)$ be a map such that:
 $$ \left\{
    \begin{array}{ccc}
    \forall \alpha \in \Delta, \forall c \in \R, & f \circ e^c_\alpha = e^c_\alpha \circ f\\
    \forall \alpha \in \Delta,                   & \varepsilon_\alpha \circ f = \varepsilon_\alpha\\
   \end{array} \right.$$
If $\lambda = \mu$, then $f=id$.
\end{thm}
\begin{proof}
Let $x \in \Bc(\lambda)$ and $y = f(x) \in \Bc(\mu)$. We start by lifting the problem to the path model. This means that we consider $\pi$ and $\pi'$ in $\Cc_0\left( [0,T], \afrak \right)$ such that:
$$ x = p(\pi) = B_T(\pi), y = p(\pi') = B_T(\pi')$$
and
$$ \Tc_{w_0} \pi(T) = \lambda, \Tc_{w_0} \pi(T) = \mu$$
Therefore, we see $\Bc(\lambda) \approx \langle \pi \rangle$ and $\Bc(\mu) \approx \langle \pi' \rangle$ as path crystals. Let us prove that for all $k \in \N$, ${\bf i} \in I^k$ and $\left( c_1, \dots, c_k \right) \in \R^k$:
$$ \frac{1}{\int_0^T \exp\left( -\alpha_{i_k}\left( e^{c_{i_1}}_{\alpha_{i_1}} \cdots e^{c_{i_{k-1}}}_{\alpha_{i_{k-1}}} \pi' \right) \right) }
 = \frac{1}{\int_0^T \exp\left( -\alpha_{i_k}\left( e^{c_{i_1}}_{\alpha_{i_1}} \cdots e^{c_{i_{k-1}}}_{\alpha_{i_{k-1}}} \pi \right) \right) }
$$
Indeed, using the same notations in the path model and in the group picture, this is equivalent to:
\begin{align*}
  & \varepsilon_{\alpha_{i_k}}\left( e^{c_{i_1}}_{\alpha_{i_1}} \cdots e^{c_{i_{k-1}}}_{\alpha_{i_{k-1}}} \pi' \right)\\
= & \varepsilon_{\alpha_{i_k}}\left( e^{c_{i_1}}_{\alpha_{i_1}} \cdots e^{c_{i_{k-1}}}_{\alpha_{i_{k-1}}} y \right)\\
= & \varepsilon_{\alpha_{i_k}}\left( e^{c_{i_1}}_{\alpha_{i_1}} \cdots e^{c_{i_{k-1}}}_{\alpha_{i_{k-1}}} f(x) \right)\\
= & \varepsilon_{\alpha_{i_k}}\circ f \left( e^{c_{i_1}}_{\alpha_{i_1}} \cdots e^{c_{i_{k-1}}}_{\alpha_{i_{k-1}}} x \right)\\
= & \varepsilon_{\alpha_{i_k}}\left( e^{c_{i_1}}_{\alpha_{i_1}} \cdots e^{c_{i_{k-1}}}_{\alpha_{i_{k-1}}} \pi \right)
\end{align*}
Taking all the $c_j$ to $-\infty$ and ${\bf i}$ a reduced word for $w_0$, one finds that the Lusztig parameters of $\pi$ and $\pi'$ coincide (see subsection \ref{subsection:lusztig_param}). Then the same goes for $x$ and $y$. And if $\lambda = \mu$, we have $x=y$.

Note that the condition $\lambda = \mu$ can be deduced from elsewhere if for instance $\gamma(x) = \gamma(y)$. Hence the following remark.
\end{proof}

\begin{rmk}
\label{rmk:minimality}
$\lambda = \mu$ is for instance implied by:
$$ \gamma \circ f = \gamma $$
\end{rmk}

\begin{rmk}
The previous theorem means that in the path model, there is a relatively large amount of isomorphic crystals. In the group picture, however the only crystal automorphism of $\Bc\left( \lambda \right)$ is the identity. In that sense, the group picture is ``minimal''.
\end{rmk}

\section{Robinson-Schensted correspondence(s)}
\label{section:rs_correspondences}
Before stating the geometric Robinson-Schensted correspondence, we recall the classical one and its generalizations using crystal bases.

\paragraph{Classical Robinson-Schensted correspondence (see \cite{bib:Fulton97}):}
An integer partition $\lambda$ is a tuple $\lambda_1 \geq \lambda_2 \geq \dots \geq 0$ such that $ |\lambda| := \sum_i \lambda_i < \infty $. Every partition $\lambda$ can be visually represented as a Young diagram, a collection of left-justified cells: $\lambda_1$ cells on the first row, $\lambda_2$ on the second etc... 

\begin{example}
The Young diagram associated to the partition $\lambda = (5,3,2)$ is:
$$
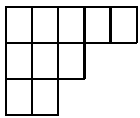
\def\lr#1{\multicolumn{1}{|@{\hspace{.6ex}}c@{\hspace{.6ex}}|}{\raisebox{-.3ex}{$#1$}}}
\raisebox{-.6ex}{\begin{array}[b]{ccccc}
\cline{1-1}\cline{2-2}\cline{3-3}\cline{4-4}\cline{5-5}
\lr{ \ }&\lr{ \ }&\lr{ \ }&\lr{ \ }&\lr{ \ }\\
\cline{1-1}\cline{2-2}\cline{3-3}\cline{4-4}\cline{5-5}
\lr{ \ }&\lr{ \ }&\lr{ \ }\\
\cline{1-1}\cline{2-2}\cline{3-3}
\lr{ \ }&\lr{ \ }\\
\cline{1-1}\cline{2-2}
\end{array}} $$
\end{example}

Consider an alphabet of $r$ letters $\A = \left\{ 1, \dots, r \right\}$. A semi-standard (resp. standard) Young tableau is a filling of a Young diagram using the alphabet, such that the entries are weakly (resp. strictly) increasing from left to right and strictly increasing down the columns. We will use the abbreviations SST (resp. ST) for ``semi-standard tableau'' (resp. ``standard tableau''). In any case, the shape of a tableau $P$ is the integer partition obtained by erasing its entries and is denoted $\sh(P)$.

Now let us describe an operation called row insertion. For a tableau $T$ and a letter $x \in \A$, one forms a new tableau that has one more entry labelled by $x$. Start with the first row and insert $x$ at the leftmost position that is strictly larger than $x$, in order to preserve the weakly increasing property from left to right. If that position is taken by a letter $y > x$, replace it by $x$ and continue the same procedure with $y$ on the next row. We say that $y$ has been bumped. If $x$ is at least as large as all the entries of the current row, $x$ is appended at the end. The procedure then stops.

The Robinson-Schensted correspondence is given by applying the algorithm with the following specifications:
\begin{itemize}
\item Input: a word $w \in \A^{ (\N) }$
\item Algorithm: do row insertions of letters to obtain a tableau $P$ and record the growth in $Q$.
\item Output: a pair $(P, Q)$ of tableaux where $P$ is a SST and $Q$ is a ST.
\end{itemize}

\begin{thm}[see \cite{bib:Fulton97} RSK correspondence (ii) p.40]
The Robinson-Schensted correspondence:
$$ \begin{array}{cccc}
  RS : &  \A^{ (\N) } & \longrightarrow & \left\{ \left( P, Q \right) , \textrm{ $P$ SST, $Q$ ST} \ | \ \sh(P) = \sh(Q) \right\}
   \end{array}
 $$
is bijection.
\end{thm}

It is well-known that the combinatorics of Young tableaux are strongly connected to the representation theory of $GL_r$. For example, irreducible representations of $GL_r$ are labeled by Young diagrams $\lambda$ and are denoted $V\left( \lambda \right)$. Each $V\left( \lambda \right)$ has a basis indexed by semi-standard tableaux $P$ of shape $\lambda$ (the Gelfand-Tsetlin basis for instance). Moreover, if $V = \C^r$ is the vector representation, irreducible submodules of $V^{\otimes T}$ for $T \in \N$ are labelled by standard tableaux $Q$ with $T$ boxes and at most $r$ rows:
\begin{align}
V^{\otimes T} & = \bigoplus_{Q \textrm{ ST with $T$ boxes}} V_Q
\label{eqn:tensor_product_decomposition}
\end{align}
where each $V_Q$ is isomorphic to a certain $V\left( \lambda \right)$ with $\lambda = \sh\left( Q \right)$. The tableau $Q$ is constructed by recording which box is added upon tensoring by a new $V$.

Let us observe the above equality. The right-side has a basis labelled by pairs of tableaux of same shape. The left-hand side has  monomial tensors as a basis and these naturally match words in $r$ letters. Therefore, the Robinson-Schensted correspondence is a bijection between good bases. A more precise relationship was first discovered by Date, Jimbo and Miwa \cite{bib:DJM90} and involves the quantum group $\Uc_q\left( \mathfrak{gl}_r \right)$ at $q=0$. This was the starting point of Kashiwara's work on crystal bases.

The representation theory of $\Uc_q\left( \mathfrak{gl}_r \right)$ is essentially the same as that of $GL_r$ and equation \eqref{eqn:tensor_product_decomposition} holds as an equality between $\Uc_q\left( \mathfrak{gl}_r \right)$-modules. In the limit as $q \rightarrow 0$, each $V_Q$ is spanned by monomial tensors. As these are labelled by words, one can ask which words span a module $V_Q$. Date, Jimbo and Miwa understood that the answer is exactly given the Robinson-Schensted correspondence: one has to consider the words that are mapped to a pair of tableaux with $Q$ being the second member.

\paragraph{Kashiwara crystal generalization:}
Let $\Bfrak\left( \kappa \right)$ be a highest weight crystal, with highest weight $\kappa \in P^+$. The previous paragraph corresponds to the particular case of $\kappa = n \omega_1$ in type $A$. Now fix a time horizon $T \in \N$. For every $0 \leq t \leq T$, given a tensor product of crystal elements $b_1 \otimes b_2 \otimes \dots b_t \in \Bfrak\left( \kappa \right)^{\otimes t}$, one can record as $\eta_t \in P^+$ the isomorphism class of the crystal it generates. Then define $p\left( b \right) \in \Bfrak$ as $b = b_1 \otimes b_2 \otimes \dots b_T$ seen as a element in the highest weight crystal it generates. Hence the map:
$$ \begin{array}{cccc}
  RS : &  \Bfrak\left( \kappa \right)^{\otimes T}    & \longrightarrow & \left\{ (b, \eta) \in \Bfrak \times \left( P^+ \right)^{T} \ | \ \hw(b) = \eta(T) \right\}\\
       &  b_1 \otimes b_2 \otimes \dots b_T          &   \mapsto       & \left( b, \left( \eta_t; t=0, 1, 2, \dots, T \right) \right)
   \end{array}
$$
The previous map is a bijection only when $\kappa$ is minuscule. This fact is more easily seen after reformulating the above thanks to the Littelmann path model \cite{bib:Littelmann95}, as we will now explain.  Let $\Pi$ be the set of piece-wise linear paths in $P \otimes \Q \subset \afrak$, also known as the Littelmann module. Let $\Pi^+$ be the set of dominant paths i.e the subset of piece-wise linear paths with values in the closed Weyl chamber. If $\pi_\kappa$ is a dominant path in $\afrak$ with $\kappa$ as endpoint, thanks to Littelmann's independence theorem, we have a crystal isomorphism $\langle \pi_\kappa \rangle \approx \Bfrak\left( \kappa \right)$. Since tensor product is simply modelled by concatenation, we can identify 
$\Bfrak\left( \kappa \right)^{\otimes T} \approx \langle \pi_\kappa \rangle^{\otimes T}$ with the subset of $\Pi$ made of paths that are concatenations of $T$ smaller parts from $\langle \pi_\kappa \rangle$. 

In this identification, the input of the Robinson-Schensted correspondence is just a path $\pi = \pi_1 \ast \dots \ast \pi_T$ in the Littelmann module, with each $\pi_i \in \langle \pi_\kappa \rangle$. The first member in the output is $\pi$ seen as an element of $\langle \pi \rangle \approx \Bc\left( \lambda \right) \hookrightarrow \Bc$. One can for instance just encode $\pi$ thanks to its Kashiwara parameters. This projection map is denoted $p: \Pi \rightarrow \Bc$ and $p(\pi)$ can be seen as an integer point lying in a dual string cone $\Cc_{\bf i}^\vee$ ( for the dual group compared to subsection \ref{subsection:Lusztig_canonical_basis}). The second member records the only dominant path in the intermediate crystals generated by $\pi_1 \ast \dots \dots \pi_t$ for $t=0,1, \dots, T$. It is given by the Pitman transform $\Pc_{w_0} \pi$ \cite{bib:BBO}. In that paper, Biane, Bougerol and O'Connell recognized that the path transform Pitman had introduced with probabilistic motivations is the crystal operator 
that gives the dominant path in a connected Littelmann module (in rank $1$ case). The Robinson-Schensted map 
then becomes:
$$ \begin{array}{cccc}
  RS : &  \Pi & \longrightarrow & \left\{ (b, \eta) \in \Bfrak \times \Pi^+ \ | \ \hw(b) = \eta(T) \right\}\\
       &  \pi &   \mapsto       & \left( p\left( \pi \right), \left(\Pc_{w_0} \pi_t; 0 \leq t \leq T \right) \right)
   \end{array}
$$
In the minuscule case, one only needs to record the path $\left(\Pc_{w_0} \pi_t; 0 \leq t \leq T \right)$ at integer times because $\langle \pi_\kappa \rangle$ is made of straight paths.

\paragraph{Continuous case:}
Biane, Bougerol and O'Connell \cite{bib:BBO, bib:BBO2} noticed that path crystals can be defined in a continuous setting. A continuous time horizon has to be considered and the Littelmann module $\Pi$ is replaced by the space of all continuous functions $\Cc_0\left( [0, T], \afrak \right)$. Dominant paths are the ones valued in the closed Weyl chamber. From their work, one can formulate a Robinson-Schensted with exactly the same form as above. The projection map $p$ is understood as the map $p_{\bf i}$ that associates to every path its string parameters and highest weight continuous crystals $\Bc\left( \lambda \right)$ are convex polytopes inside the dual string cone $\Cc_{\bf i}^\vee$.

Notice that the Pitman transform is the same for both the continuous and combinatorial settings. In fact, one can embed the classical Littelmann path model in the continuous one.

\paragraph{Geometric case:} As we have proved, concatenation is still a model for tensor product for geometric crystals. Hence, a path $\pi \in \Cc_0\left( [0, T], \afrak \right)$ can be seen as a tensor product of smaller paths. The path crystal it generates can be projected onto its group picture, which is an intrinsic realization of a $\Bc\left( \lambda \right)$ for a certain $\lambda \in \afrak$.

\begin{thm}[Geometric Robinson-Schensted correspondence]
\label{thm:dynamic_rs_correspondence}
For each $T>0$, we have a bijection:
$$ \begin{array}{cccc}
  RS : & \Cc_0\left( [0, T], \afrak \right) & \longrightarrow & \left\{ (x, \eta) \in \Bc \times \Cc^{high}_{w_0}\left( [0,T], \afrak \right) \ | \ \hw(x) = \eta(T) \right\}\\
       &         \pi                      &   \mapsto       & \left( B_T\left( \pi \right), \left(\Tc_{w_0} \pi_t; 0 < t \leq T \right) \right)
   \end{array}
 $$
\end{thm}
\begin{proof}
By theorem \ref{thm:string_params_paths}, the knowledge of a path $\pi \in \Cc_0\left( [0, T], \afrak \right)$ is exactly equivalent to that of knowing the string parameters and the highest weight path. And thanks to corollary \ref{corollary:commutative_diagram_lusztig}, the string parameters along with the highest weight $\lambda = \Tc_{w_0} \pi(T)$ are encoded by $B_T(\pi) \in \Bc(\lambda)$.
\end{proof}

After this presentation, the analogy with the classical Robinson-Schensted correspondence is clear. The path $\pi$ plays the role of a word. Elements in $\Bc$, the crystal elements, play the role of semi-standard tableaux. Highest paths play the role of shape dynamic. Finally the condition:
$$ \hw(x) = \Tc_{w_0}\pi(T)$$
is the equivalent of saying that the $P$ tableau and the $Q$ tableau have the same shape.

\section{Involutions and crystals }
\label{section:involutions}

\subsection{Kashiwara involution}
\label{subsection:involution_iota}
The Kashiwara involution was defined at the level of the enveloping algebra as the unique anti-automorphism satisfying
$$e_\alpha^\iota = e_\alpha, \quad f_\alpha^\iota = f_\alpha \quad h_\alpha^\iota = -h_\alpha $$
It can naturally be lifted to a group anti-automorphism, and we have seen on the path model that it is the group picture counterpart of duality: $\pi \mapsto \pi\left( T-t \right) - \pi\left( T \right)$.

\subsection{Sch\"utzenberger involution}
\label{subsection:involution_S}
This map was originally introduced by Sch\"utzenberger as an involution on semi-standard tableaux of a given shape. As tableaux with $n$ letters of a given shape $\lambda$ can be identified with the highest weight crystal $\Bc\left( \lambda \right)$ of type $A_n$, it can be seen as an involution on highest weight crystals.

\subsubsection{Definition in the group picture}

\begin{definition}[Sch\"utzenberger involution on $G$]
$$ \forall x \in G, S\left( x \right) = \overline{w}_0^{-1} ( x^{-1} )^{iT} \overline{w}_0 = \overline{w}_0 ( x^{-1} )^{iT} \overline{w}_0^{-1} $$ 
\end{definition}

\begin{rmk}
 This is defined ``coordinate free'' on the entire group.
\end{rmk}
\begin{rmk}
 Both definitions agree because $\overline{w}_0^2$ belongs to $\mathcal{Z}(G)$, the center of $G$, as a consequence of the following lemma.
\end{rmk}
\begin{lemma}
 For each $w$, define $t(w) = \overline{w} \overline{w^{-1}}$. We have:
$$ t(w) = e^{i\pi \left( \rho^\vee - w \rho^\vee \right)}$$
In particular:
$$ t(w_0) = \overline{w}_0^2 = e^{i 2 \pi \rho^\vee} $$
\end{lemma}
\begin{proof}
Using the notations from preliminaries, if ${\bf i} = \left(i_1, \dots, i_k\right) \in R\left( w \right)$ and $\left( \beta^\vee_1, \dots, \beta^\vee_k\right)$ the associated positive roots enumeration:
\begin{align*}
t(w) & = \bar{s}_{i_1} \dots \bar{s}_{i_k} \bar{s}_{i_k} \dots \bar{s}_{i_k}\\
& = \phi_{i_k}\left( e^{i \pi h} \right)^{s_{i_1} \dots s_{i_{k-1}}}
    \phi_{i_{k-1}}\left( e^{i \pi h} \right)^{s_{i_1} \dots s_{i_{k-2}}}
    \dots
    \phi_{i_2}\left( e^{i \pi h} \right)^{s_{i_1}}
    \phi_{i_1}\left( e^{i \pi h} \right)\\
& = \exp\left( i\pi \sum_{k=1}^k s_{i_1} \dots s_{i_{j-1}} h_{i_j} \right)\\
& = \exp\left( i\pi \sum_{k=1}^k \beta_k^\vee \right)\\
& = \exp\left( i\pi \left( \rho^\vee - w \rho^\vee\right) \right)
\end{align*} 
\end{proof}

\begin{properties}
\phantomsection
\label{properties:involution_S}
\begin{itemize}
 \item $S$ is an involutive anti-automorphism on the group.
 \item $$ \forall k \in \N, \left(i_1, \dots, i_k\right) \in I^k, S\left( x_{i_1}\left( t_1 \right) \dots x_{i_k}\left( t_k \right) \right) = x_{i_k^*}\left( t_k \right) \dots x_{i_1^*}\left( t_1 \right)$$
 \item $$ S\left( \bar{w_0} \right) = \bar{w}_0$$
 \item $$ \forall x \in B^+ w_0 B^+, \hw\left( S\left( x \right) \right) = \hw\left( x \right)$$
\end{itemize}
\end{properties}
\begin{proof}
It is easy to see that $S$ is an anti-automorphism as the composition of three anti-automorphisms (inverse, transpose and Kashiwara involution $\iota$) and an automorphism (conjugation by $\bar{w_0}$), hence the first property. The second property is known to Berenstein and Zelevinsky (relation 6.4 in \cite{bib:BZ01}) and is a consequence of $\Ad\left( \bar{w_0} \right)e_i = -f_{i^*}$ (proposition \ref{proposition:w_0_action_ad}). The rest is easy to check by direct computation. 
\end{proof}

Since $S$ stabilizes the geometric crystal $\Bc$ and preserves the highest weight, it is indeed an involution on highest weight crystals $\Bc\left( \lambda \right)$. But it is not a morphism of crystals.

\subsubsection{Definition in the path model}
\begin{definition}[Sch\"utzenberger involution on paths]
$$ \forall \pi \in \Cc_0\left( [0,T], \afrak \right), \ S\left( \pi \right) := - w_0 \pi^\iota $$ 
\end{definition}
As usual, this definition in fact agrees with the group picture after projection.
\begin{thm}
\label{thm:involution_S_automorphism}
 $$ p \circ S = S \circ p$$
 where on the right-hand side, $S$ stands for the Sch\"uzenberger involution on the group, and the left-hand side it is considered  in the path model.
\end{thm}
\begin{proof}
 Consider a smooth path $\pi$. Since $S\circ \iota$ is an automorphim, the left-invariant equation solved by $S \circ \iota\left( B_t\left(\pi \right) \right)$ is:
\begin{align*}
d S \circ \iota\left( B_t\left(\pi \right) \right) & = S \circ \iota\left( B_t\left(\pi \right) \right)
                                                       d(S \circ \iota)\left( \sum f_\alpha + d\pi_t \right) \\
& = S \circ \iota\left( B_t\left(\pi \right) \right) \left( \sum_{\alpha \in \Delta} f_{\alpha^*} -w_0 d\pi_t \right)\\
& = S \circ \iota\left( B_t\left(\pi \right) \right) \left( \sum_{\alpha \in \Delta} f_{\alpha  } -w_0 d\pi_t \right)
\end{align*}
Then:
$$ S\circ \iota\left( B_t\left(\pi\right) \right) = B_t\left( -w_0 \pi \right)$$
Replacing $\pi$ by $\pi^\iota$ gives the result for all smooth paths:
$$ S\left( B_t\left(\pi\right) \right) = B_t\left( -w_0 \pi^\iota \right)$$
The smoothness assumption can then be discarded.
\end{proof}

\begin{rmk}
In the path model, $S$ does not preserve connected components, but thanks to the previous theorem, it stabilizes highest weight crystals in the group picture. As such it preserves isomorphism classes of path crystals using Littelmann's independence theorem.
\end{rmk}

As pointed out in \cite{bib:BBO2} p. 1552 lemma 4.19 the following can be taken as a definition for the Sc\"utzenberger involution for $A_n$ crystals. We prove the analogous statement in the geometric setting:
\begin{thm}
The Sch\"utzenberger involution is the unique map $S$ on geometric crystals (resp. path crystals up to crystal isomorphism) such that:
\begin{itemize}
 \item $\gamma \circ S\left( x \right) = w_0 \gamma\left( x \right)$
 \item $\varepsilon_\alpha \circ S\left( x \right) = \varphi_{\alpha^*}\left( x \right) $
       or equivalently
       $\varphi_\alpha \circ S\left( x \right) = \varepsilon_{\alpha^*}\left( x \right) $
 \item $\forall c \in \R, e^c_\alpha \cdot S\left(x\right) = S\left( e^{-c}_{\alpha^*} \cdot x \right)$
\end{itemize}
\end{thm}
\begin{proof}
Computations can be carried out very easily both in the group or on the path model.

For uniqueness, if $S$ and $S'$ satisfy those properties, then $S S'$ is an automorphism of (path) crystals. In the group picture, there is no crystal automorphism aside from the identity (subsection \ref{subsection:minimality}), hence the uniqueness up to isomorphism.
\end{proof}

\appendix
\section{Kostant's Whittaker model}
\label{appendix:whittaker_model}
For more algebraic details, we refer to the first section in \cite{bib:Sevostyanov00} as it gives a very good summary of Kostant's work on the Whittaker model and Whittaker modules. Here, we will mainly be interested in the image of the Casimir operator in the Whittaker model, seen as a left-invariant differential operator on the lower Borel subgroup $B$.

\paragraph{The universal enveloping algebra:}

\subparagraph{Invariant differential operators:} 
Every $X \in \gfrak$ can be viewed as a left-invariant differential operator of order $1$. Its action on smooth functions is given by:
$$ \forall f \in \Cc^\infty\left( G \right), X f(g) := \lim_{t \rightarrow 0} \frac{f(g e^{tX})-f(g)}{t}$$
From such a point of view, it is easy to envision invariant different operators of arbitrary order. They should be obtained by composing elements $X_1, X_2, \dots, X_k$ in $\gfrak$ acting as differential operators. Their identification is subject to possible relations due to the Lie bracket $[\ ,\ ]$.

This notion is formalized in algebra as the universal enveloping algebra $\Uc(\gfrak)$.

\subparagraph{Definition from universal property:} 
The universal enveloping algebra of $\gfrak$ is constructed as the quotient of the tensor algebra $\bigoplus_n \gfrak^{\otimes n}$ by the two sided ideal generated by $ab - ba - [a, b]$, $a, b \in \gfrak$. It has the universal property that any Lie algebra homomorphism $f: \gfrak \rightarrow A$, where $A$ is a unital algebra, factors into $f = g \circ i$. $g: \Uc(\gfrak) \rightarrow A$ is unique and $i: \gfrak \rightarrow \Uc(\gfrak)$ is the inclusion.

\subparagraph{Definition with generators and relations:} 
An alternative definition uses generators and relations, with only the Cartan matrix $A = (a_{ij})_{1 \leq i,j \leq n}$ as input data.
$\Uc(\gfrak)$ is the unital associative algebra generated by $F_i, H_i, E_i$ with $1 \leq i \leq n$ with relations:
\begin{align*}
 [H_i, H_i] & = 0\\
 [E_i, F_j] & = \delta_{i,j} H_i\\
 [H_i, E_j] & = a_{ij} H_i\\
 [H_i, F_j] & = -a_{ij} H_i\\
 \textrm{Serre relations} & \textrm{ for } i \neq j:\\
 0 & = \sum_{s=0}^{1-a_{ij}} (-1)^s \binom{1-a_{ij}}{s}E_i^{1-a_{ij}-s} E_j E_i^s\\
 0 & = \sum_{s=0}^{1-a_{ij}} (-1)^s \binom{1-a_{ij}}{s}F_i^{1-a_{ij}-s} F_j F_i^s
\end{align*}

\paragraph{Center:} 
$$\Zc(\mathfrak{g}) := \left\{ x \in \Uc(\gfrak) | \forall y \in \Uc(\gfrak) [x, y] = 0 \right\}$$
The center $\Zc(\gfrak)$ forms a commutative algebra. It is at the heart of both classical and quantum integrable systems. 

In Hamiltonian mechanics the Lie bracket is interpreted as a Poisson bracket and the center is an algebra of Poisson commuting functions. These functions are the observables that are integrals of motion. In quantum mechanics, the story is a bit different. Observables are differential operators acting on the Hilbert space of wave functions. Commuting observables give simultaneously measurable observables, which is a very desirable property. The center is a commutative algebra of differential operators, and the Lie bracket is simply the commutator $[A, B] = AB - BA$. 

The integrability property means that we have a maximal number of independent invariants. Chevalley's theorem tells us that the maximal number of independent central elements is $r$, the rank of Lie algebra.

\begin{thm}[Chevalley]
 $\Zc(\mathfrak{g})$ is a polynomial algebra with $r$ independent generators $I_1, I_2, \dots, I_r$. 
$\Zc(\mathfrak{g}) = \C[I_1, I_2, \dots, I_r]$
\end{thm}

\paragraph{Casimir element:} 
The only element of order $2$ in the center is the Casimir element $\Omega$. If $(X_1, \dots, X_n)$ is an orthonormal basis of $\hfrak$ with respect to the Killing form, then:
\begin{align}
\label{lbl:casimir}
\begin{array}{ll}
\Omega & := \half \sum_{i=1}^n X_i^2 +  \half \sum_{ \beta \in \Phi^+}\left( f_\beta e_\beta + e_\beta f_\beta \right)\\
& = \half \sum_{i=1}^n X_i^2 +  \sum_{ \beta \in \Phi^+} f_\beta e_\beta + \rho^\vee\\
\end{array}
\end{align}
The second expression uses the Weyl covector $\rho^\vee$, which is the vector in $\afrak$ such that $\alpha\left( \rho^\vee \right) = 1$ for every simple root $\alpha$. $\rho^\vee$ is also the half sum of all positive coroots. In a way, $\Omega$ is the simplest and most important element. In representation theory, it is used in order to prove the reducibility of certain classes of representations. In analysis, because it is of order $2$, it can be considered as a heat kernel, when elliptic. In order to see how it is related to the left-invariant differential equation \eqref{lbl:process_B_ode} on $B$, we need to exhibit an element of $\Uc(\bfrak)$.

\paragraph{Reduction to $\Uc(\bfrak)$:} 
Let $\chi: U \longrightarrow \C$ be the standard (additive) character on $U$:
$$\forall \alpha \in \Delta, \chi\left( e^{t e_\alpha} \right) = t$$
Let $\Cc^\infty\left(G\right)$ be the space of continuously differentiable functions on $G$ and consider the subspace:
$$ \Cc^\infty_\chi(G) := \left\{ f \in \Cc^\infty\left(G\right) \ | \ \forall u \in U, f(g u) = f(g) e^{\chi(u)} \right\}$$
The subset $B B^+$, the cell where a Gauss decomposition holds, is dense in $G$. Hence, any function in $\Cc^\infty_\chi$ is entirely determined by its restriction to the lower Borel subgroup $B$. Moreover, differential operators in $\Uc\left( \mathfrak{g} \right)$ are reduced to elements of $\Uc\left( \bfrak \right)$ when acting on such functions. By simple differentiation and restriction to $\Cc^\infty(B)$, $\Omega$ reduces to:
\begin{align}
\label{lbl:casimir_chi}
\Omega_\chi & = \sum_{i=1}^n X_i^2 + \sum_{\alpha \in \Delta} f_\alpha + \rho^\vee
\end{align}

\paragraph{Whittaker model $W(\bfrak)$:} 
The algebraic construction by Kostant tantamounts to reducing central elements to elements in $\Uc(\bfrak)$. We reproduce the presentation of \cite{bib:Sevostyanov00} while keeping the same notations. $\chi$, at the level of the Lie algebra, extends to $\Uc\left( \ufrak \right)$ and gives a direct sum:
$$ \Uc\left( \ufrak \right) = \C \mathds{1} \oplus \ker \chi$$
Since $\Uc(\gfrak) = \Uc(\bfrak)\otimes \Uc\left( \ufrak \right)$ because of the Poincar\'e-Birkhoff-Witt basis theorem, we have:
$$ \Uc(\gfrak) = \Uc(\bfrak) \oplus I_\chi$$
where $I_\chi = \Uc(\gfrak) \ker \chi$ is the left ideal generated by $\ker \chi$. Now let $\rho_\chi$ be the 
canonical projection:
$$ \rho_\chi: \Uc(\gfrak) \longrightarrow \Uc(\bfrak)$$

It defines the Whittaker model for the center thanks to:
\begin{thm}[Kostant, Theorem 2.4.2 in \cite{bib:Kostant78}]
Let $W(\bfrak) = \rho_\chi\left( \Zc(\gfrak) \right)$. The map:
$$ \rho_\chi: \Zc(\gfrak) \longrightarrow W(\bfrak)$$
is an isomorphism.
\end{thm}
One can easily compute the image of the Casimir element $\Omega_\chi = \rho_\chi\left( \Omega \right)$: if $\beta$ is a simple root, $e_\alpha$ acts like $1$ after reduction while if $\beta \in \Phi^+ - \Delta$, it acts like $0$. From the second line in equation \eqref{lbl:casimir}, we recover the same operator as in equation \eqref{lbl:casimir_chi}. This is simply the algebraic derivation of the reduction stated in the previous paragraph.

$\Omega_\chi \in \Uc(\bfrak)$ is interpreted as an operator on the solvable group $B$. The Laplacian $\Delta_\afrak = \half \sum_{i=1}^n  X_i^2$ is the infinitesimal generator of Brownian motion on $\afrak \approx \R^n$. Hence, $\Omega_\chi$ as a whole is the infinitesimal generator of a Markov process driven by a simple Euclidian Brownian motion on $\afrak$. This Markov process is exactly the one defined by equation \eqref{lbl:process_B_ode}, if the driving path $X$ is a Brownian motion. In this context, which we pursue in \cite{bib:chh14c}, equation \eqref{lbl:process_B_ode} is interpreted as a stochastic differential equation.

\paragraph{Quantum Toda Hamiltonian:} 
Let $\chi^-: N \longrightarrow \C$ be the standard (additive) character on $N$:
$$\forall \alpha \in \Delta, \chi^-\left( e^{t f_\alpha} \right) = t$$
We can consider a further reduction to the space of $\chi^-\chi$-binvariant functions:
$$ \Cc^\infty_{\chi^- \chi}(G) := \left\{ f \in \Cc^\infty\left(G\right) \ | \ \forall n \in N, u \in U, f(n g u) = e^{-\chi^-(n)} f(g) e^{\chi(u)} \right\}$$
Since a function in $\Cc^\infty_{\chi^- \chi}(G)$ is entirely determined by its values on $A = \exp(\afrak)$, we have the identification $\Cc^\infty_{\chi^- \chi}(G) \approx \Cc^\infty(\afrak)$. In this identification, the operator $\Omega_\chi$ reduces to a Sch\"odinger operator on $\afrak \approx \R^r$ known as the quantum Toda Hamiltonian:
\begin{align}
\label{lbl:quantum_toda}
H & = \Delta_\afrak + \rho^\vee - \sum_\alpha e^{-\alpha(x)}
\end{align}
 The quantum Toda Hamiltonian will also appear in \cite{bib:chh14c}, being related to the highest weight dynamic.

\bibliographystyle{halpha}
\bibliography{Bib_Thesis}

\end{document}